\documentclass[11pt]{amsart}
\usepackage{geometry}                % See geometry.pdf to learn the layout options. There are lots.
\usepackage{graphicx}
\usepackage{amssymb}
\usepackage{epstopdf}
\usepackage{xcolor}

\theoremstyle{plain}
\newtheorem{theo}{Theorem}[section]
\newtheorem{lem}[theo]{Lemma}
\newtheorem{cor}[theo]{Corollary}
\newtheorem{prop}[theo]{Proposition}

\numberwithin{equation}{section}

\theoremstyle{definition}

\newtheorem{remark}[theo]{Remark}

\def\e{\varepsilon}

\def\vf{\varphi}
\def\a{\alpha}

\def\P{\mathfrak P}

\def\d{\delta}

\def\D{\Delta}

\def\g{\gamma}

\def\LL{\mathcal L}
\def\pp{\partial}
\def\l{\lambda}

\def\s{\sigma}

\def\x{\times}
\def \R{\mathbb R}

\def\E{\mathbb E}
\def \Z{\mathbb Z}

\def \P{\mathbb P}
\def\ov{\overline}
\def\un{\underline}

\def\u{\underline}

\def\p{\phi}

\def\om{\omega}
\def\Om{\Omega}

\def\A{\mathcal A}

\def\Q{\mathbb Q}
\def\W{\mathcal W}

\def\wt{\widetilde}
\def\({\biggl(}
\def\){\biggr)}

\def\<{\bold\langle}
\def\>{\bold\rangle}

\def\LL{{{\mathcal L}}}

\def\M{\widetilde {M}}

\def\bfv{{\bf v}}
\def\bfp{{\bf p}}
\def\bfg{{\bf G}}

\def\p{\varpi}

\definecolor{Lin}{RGB}{220, 30, 70}
\definecolor{Francois}{RGB}{70,30, 220}

\begin{document}

\title[Differentiating the stochastic entropy in negatively curved spaces]{Differentiating the stochastic entropy for compact negatively  curved spaces under conformal changes}
\author{Fran\c cois Ledrappier  and  Lin Shu}
\address{Fran\c cois Ledrappier,  Department of Mathematics, University of Notre Dame, IN 46556-4618, USA and  LPMA, Bo\^{i}te Courrier 188, 4, Place Jussieu, 75252 PARIS cedex
05, France}\email{fledrapp@nd.edu}
\address{Lin Shu, LMAM, School
of Mathematical Sciences, Peking University, Beijing 100871,
People's Republic of China} \email{lshu@math.pku.edu.cn}
% \date{September 19, 2013}                                         %Activate to display a given date or no date

\maketitle
\begin{abstract}
We consider the universal cover  of a closed connected Riemannian
manifold of negative sectional curvature. We show that the linear
drift and  the stochastic entropy are differentiable under any $C^3$
one-parameter family of $C^3$ conformal changes of the original
metric.{ \scriptsize\hspace*{-2cm}\tableofcontents}
\end{abstract}

\section{Introduction}\label{Sec-intro}
Let $(M, g)$ be an
$m$-dimensional closed connected Riemannian manifold, and $\pi:
(\wt{M}, \wt{g})\to (M, g)$ its universal cover endowed with the
lifted Riemannian metric. The fundamental group $G=\pi_1(M)$ acts on
$\wt{M}$ as isometries such that $M=\wt{M}/G$.

We consider the Laplacian  $\Delta:={\rm{Div}}\nabla$ %\footnote{I replace $\mbox{Div}$ by ${\rm{Div}}$ and $\mbox{Vol}$ by ${\rm Vol}$ throughout the paper. The later look better in formulas.} 
on smooth
functions on $(\wt{M}, \wt{g})$ and the corresponding heat kernel function
$p(t, x, y), t\in \Bbb R_{+}, x, y\in \wt{M}$, which is the
fundamental solution to the heat equation $\frac{\partial
u}{\partial t}=\Delta u$. Denote by ${\rm{Vol}}$ the Riemannian
volume on $\wt{M}$. The following quantities were introduced by
Guivarc'h (\cite{Gu}) and Kaimanovich (\cite{K1}),  respectively, and are
independent of $x\in \wt{M}$:
\begin{itemize}
  \item the linear drift $\ell:=\lim_{t\rightarrow +\infty} \frac{1}{t}\int d_{\wt{g}}(x, y)p(t, x, y)\
  d{\rm{Vol}}(y)$.
  \item the stochastic entropy $h:=\lim_{t\rightarrow +\infty} -\frac{1}{t}\int \left(\ln p(t, x, y)\right) p(t, x, y)\
  d{\rm{Vol}}(y)$.%\footnote{For peace of my mind,  I get rid of the secondhand colors, at least!}
\end{itemize}

 Let $\{g^{\lambda}=e^{2\vf^{\lambda}} g:\ |\lambda|<1\}$ be a  one-parameter family of   conformal changes of $g^0=g$,  where $\vf^{\l}$'s are real valued functions on $M$ such that  $(\l, x)\mapsto \vf^{\lambda}(x)$ is $C^3$ and $\vf^0\equiv 0$.  Denote by  $\ell_{\lambda}, h_{\lambda}$, respectively,  the
linear drift and the stochastic entropy for $(M, g^{\lambda})$. We show

\begin{theo}\label{main-thm} Let $(M, g)$ be a negatively curved closed connected Riemannian manifold.  With the above notation, the functions $\lambda\mapsto
\ell_{\lambda}$ and $\lambda\mapsto h_{\lambda}$ are differentiable
at $0$.
\end{theo}

For each $\l\in (-1, 1)$, let $\Delta^{\lambda}$ be the Laplacian of
$(\M, \widetilde{g}^{\lambda})$ with heat kernel $p^{\lambda}(t, x,
y), t\in \Bbb R_{+}, x, y\in \wt{M}$,  and the associated Brownian
motion $\omega_{t}^{\lambda}, t\geq 0$.  The relation between
$\Delta^{\lambda}$ and $\Delta$ is easy to be formulated using
$g^{\lambda}=e^{2\vf^{\lambda}} g$:  for $F$ a $C^2$ function on $\M$,
\[
\Delta^{\lambda} F=e^{-2\vf^{\lambda}}\left(\Delta F+(m-2)\langle \nabla\vf^{\l},\nabla F\rangle_g\right)=:
e^{-2\vf^{\lambda}} L^{\lambda}F,
\]
where we still denote $\vf^{\lambda}$ its  $G$-invariant  extension to $\M$. 
Let $\widehat{p}^{\lambda}(t, x, y), t\in \Bbb R_{+}, x, y\in
\wt{M}$,   be the heat kernel of the diffusion process
$\widehat{\omega}_{t}^{\lambda}, t\geq 0,$  corresponding to the
operator $L^{\lambda}$ in $(\M, \wt{g})$.    We define
\begin{itemize}
  \item $\widehat{\ell}_{\lambda}:=\lim_{t\rightarrow +\infty} \frac{1}{t}\int d_{\wt{g}}(x, y)\widehat{p}^{\lambda}(t, x, y)\
  d{\rm{Vol}}(y)$.
  \item $\widehat{h}_{\lambda}:=\lim_{t\rightarrow +\infty} -\frac{1}{t}\int (\ln \widehat{p}^{\lambda}(t, x, y))\widehat{p}^{\lambda}(t, x, y)\
  d{\rm{Vol}}(y)$.
\end{itemize}
It is clear that  the following hold true providing all the limits exist:
\begin{eqnarray*}
(d\ell_{\lambda}/d\lambda)|_{\lambda=0}&=&\lim\limits_{\l\to
0}\frac{1}{\l}(\ell_{\l}-\widehat{\ell}_{\l})+\ \lim\limits_{\l\to
0}\frac{1}{\l}(\widehat{\ell}_{\l}-\ell_0)\  =: {\rm{\bf
(I)}}_{\ell}+{\rm{\bf (II)}}_{\ell},\\
(dh_{\lambda}/d\lambda)|_{\lambda=0}&=&\lim\limits_{\l\to
0}\frac{1}{\l}(h_{\l}-\widehat{h}_{\l})+\lim\limits_{\l\to
0}\frac{1}{\l}(\widehat{h}_{\l}-h_0)=:{\rm{\bf (I)}}_{h}+{\rm{\bf
(II)}}_{h}.
\end{eqnarray*}
Here, loosely speaking, ${\rm{\bf
(I)}}_{\ell}$ and ${\rm{\bf (I)}}_{h}$ are the infinitesimal drift
and entropy affects  of  simultaneous metric change and time
change of the  diffusion (when the generator of the diffusion changes  from $L^{\l}$ to $\Delta^{\l}$), while ${\rm{\bf
(II)}}_{\ell}$ and ${\rm{\bf (II)}}_{h}$ are the infinitesimal  responses to the adding of drifts to  $\omega_{t}^{0}$  (when the generator of the diffusion changes from $\Delta$ to $L^{\l}$).

To analyze ${\rm{\bf (I)}}_{\ell}$ and ${\rm{\bf (I)}}_{h}$,  we express the above  linear drifts and stochastic entropies using the geodesic
spray, the Martin kernel and the exit probability of the Brownian motion
at infinity.  It is known (\cite{K1}) that
\begin{eqnarray}\label{formulas-ell-h}
\ell_{\l}=\int_{\scriptscriptstyle{M_0\times
\pp\M}}\langle{X}^{\l}, \nabla^{\l}\ln
k^{\l}_{\xi}\rangle_{\l}\ d{\bf \wt{m}^{\l}},\ \
h_{\l}=\int_{\scriptscriptstyle{M_0\times \pp\M}} \|\nabla^{\l} \ln
k^{\l}_{\xi}\|_{\l}^2\ d{\bf\wt{m}^{\l}},
\end{eqnarray}
where $M_0$ is a fundamental domain of $\M$, $\pp\M$ is the
geometric boundary of $\M$, % \footnote{By geodesic Sectionspray, it is the horizontal lift of $\overline{X}^{\l}$, right? I delete the term before this sentence.}
 ${X}^{\l}(x,\xi)$ is
the unit tangent vector of the $\wt{g}^{\l}$-geodesic starting from
$x$ pointing at $\xi$, $k^{\l}_{\xi}(x)$ is the Martin kernel
function of $\omega_{t}^{\l}$ and ${\bf{\wt{m}}^{\l}}$ is the
harmonic measure associated with $\D^{\l}$.    (Exact definitions will appear in Section 2.)  Similar formulas also
exist for $\widehat{\ell}_{\l}$ and $\widehat{h}_{\l}$ (see Propositions \ref{formulas-l-h-Y-l},  \ref{formulas-l-h-Y-h} and (\ref{ellhatlambda})) 
\begin{equation}\label{Intro-hat-ell-h}\widehat{\ell}_{\l}=\int\langle{X}^0, \nabla^0\ln
k_{\xi}^{\l}\rangle_0\ d\widehat{{\bf m}}^{\l}, \ \widehat{h}_{\l}=\int \|\nabla^{0}\ln k_{\xi}^{\l}(x)\|_{0}^2\
  d\widehat{{\bf m}}^{\l},\end{equation}
where  $\widehat{{\bf m}}^{\l}$ is the harmonic measure related to
the operator $L^{\lambda}$.   The  quantity  ${\rm{\bf (I)}}_{h}$
turns out to be zero since the norm and the gradient changes cancel
with the measure change, while the Martin kernel function remains
the same under time rescaling  of the diffusion process (see Section
5, (\ref{zeroentropy}) and the paragraph after (\ref{volume})).  But the metric variation  is more involved
in ${\rm{\bf (I)}}_{\ell}$ as we can see from the formulas in
(\ref{formulas-ell-h}) and (\ref{Intro-hat-ell-h}) for $\ell_{\l}$
and $\widehat{\ell}_{\l}$.  In Section 4,  using  the $(g,
g^{\l})$-Morse correspondence maps (see \cite{A,Gr,Mor} and
\cite{FF}),  which are homeomorphisms between the unit tangent
bundle spaces in $g$ and $g^{\l}$ metrics sending $g$-geodesics to
$g^\l $-geodesics,   we are able to identify the differential
\begin{equation}\label{Intro-X-lambda}
\left(\overline{X}^{\l}\right)'_0(x,
\xi):=\lim\limits_{\l\to 0}
\frac{1}{\l}\left(\overline{X}^{\l}(x,
\xi)-\overline{X}^{0}(x, \xi)\right),
\end{equation}
where now $\ov X^\l (x, \xi ) $ is   the horizontal lift of $X^\l (x, \xi )$ to  $T_{(x,\xi) } S\M$  (see below Section 2.4), 
using  the stable and unstable Jacobi tensors and a family of Jacobi
fields arising naturally from the infinitesimal Morse correspondence
(Proposition \ref{lem-W}  and Corollary \ref{X-lam-Lem}).
 As a consequence, we can express ${\rm{\bf
(I)}}_{\ell}$ using $k_{\xi}^0$, $\wt{{\bf m}}^0$  and these terms (see
the proof of Theorem \ref{Main-formulas}).

If we continue to analyze  ${\rm{\bf (II)}}_{\ell}$ and
${\rm{\bf (II)}}_{h}$ using (\ref{formulas-ell-h}) and (\ref{Intro-hat-ell-h}), we have the problem of showing  the regularity in $\l$  of the gradient of the Martin kernels.  We avoid this by using an idea from Mathieu (\cite{Ma})  to  study ${\rm{\bf (II)}}_{\ell}$ and
${\rm{\bf (II)}}_{h}$ along  the diffusion processes.   For every point $x\in \M$ and almost every (a.e.)
$\wt{g}$-Brownian motion path $ \omega^{0}$ starting from $x$,  it is known (\cite{K1}) that
\begin{equation}\label{Intro-d-Green-formula}
\lim\limits_{t\rightarrow +\infty}\frac{1}{t} d_{\wt{g}}(x,
\omega_{t}^{0})=\ell_{0}, \ \ \lim\limits_{t\rightarrow
+\infty}-\frac{1}{t}\ln G(x, \omega_{t}^{0})=h_{0},
\end{equation}
where $G(\cdot, \cdot)$ on $\M\times \M$ denotes the Green function for
$\wt{g}$-Brownian motion.  A further study on  the convergence of the limits of (\ref{Intro-d-Green-formula}) in \cite{L5}
showed  that there are positive numbers $\sigma_{0}, \sigma_{1}$ so that the
distributions of the variables
\begin{equation}\label{intro-CLT}
{\rm Z}_{\ell, t}(x)=\frac{1}{\sigma_0\sqrt{t}}\left[d_{\wt{g}}(x,
 \omega_{t}^{0})-t\ell_{0}\right], \ \ {\rm Z}_{h,
t}(x)=\frac{1}{\sigma_1\sqrt{t}}\left[-\ln G(x, \omega_{t}^{0})-
th_{0}\right]
\end{equation}are asymptotically close to the normal distribution as $t$ goes
to infinity.  Moreover, these %types of the Central Limit Theorems for  the linear drift and the stochastic entropy 
limit theorems  have some uniformity when we vary the original metric locally in the space of negatively curved metrics.  This allows us to identify ${\rm{\bf (II)}}_{\ell}$ and ${\rm{\bf (II)}}_{h}$ respectively with the limits \[
\left(\E_{\l}(\frac{1}{\sqrt{t}}d_{\wt{g}}(x,
 \omega_{t}^{0}))-\E_{0}(\frac{1}{\sqrt{t}}d_{\wt{g}}(x,
\omega_{t}^{0}))\right), \ \ -\left(\E_{\l}(\frac{1}{\sqrt{t}}\ln G(x, \omega_{t}^{0}))-\E_{0}(\frac{1}{\sqrt{t}}\ln G(x,  \omega_{t}^{0}))\right),
\]
where we take $\l=\pm 1/\sqrt{t}$ and $\E_{\l}$ is the expectation with respect to the transition probability of the $L^{\l}$ process.  %\lin{on $SM$. } 
(More details of the underlying idea will be exposed in Section \ref{sec-4.2} after we introduce the corresponding notations.) Note that all $\widehat{\omega}_{t}^{\lambda}$ starting from $x$
can be simultaneously represented as   random processes on the
probability space $(\Theta, \Bbb Q)$ of a standard $m$-dimensional
Euclidean Brownian motion.  By using the Girsanov-Cameron-Martin
formula on manifolds (cf. \cite{El}),  we are able to compare $\E_{\l}$ with $\E_0$ on the same probability space of continuous path spaces. As a consequence, we show
 \begin{equation*}\label{entropy-2}
\mbox{\bf (II)}_{\ell}=\lim_{t\to +\infty}\Bbb E_0 ({\rm Z}_{\ell,
t}{\rm M}_t)\ \ \mbox{and} \ \ \mbox{\bf (II)}_{h}=\lim_{t\to
+\infty}\Bbb E_0 ({\rm Z}_{h, t}{\rm M}_t),
\end{equation*}
where  each ${\rm M}_t$ is a
random process on $(\Theta, \Bbb Q)$ recording the change of metrics
along the trajectories of Brownian motion to be specified in Section 5.
We will further specify $\mbox{\bf (II)}_{\ell}$ and $\mbox{\bf
(II)}_{h}$ in Theorem \ref{Main-formulas} using properties of
martingales and the Central Limit Theorems for the  linear drift and the
stochastic entropy.

An immediate consequence of  Theorem \ref{main-thm} is that
$D_{\lambda}:=h_{\lambda}/\ell_\lambda$, which is proportional
\footnote{ $D_{\l} $ is $\frac{1}{\iota} $  the Hausdorff dimension  of the exit measure
for the $\iota$-Busemann distance (cf.  Section \ref{sec-4.2}).} to the Hausdorff dimension of the distribution of
the Brownian motion $\omega^{\lambda}$ at the infinity boundary of
$\M$ (\cite{L1}),  is also differentiable in $\lambda$.  Let
$\Re(M)$ be  the manifold of negatively curved $C^3$
metrics on $M$.  Another consequence of Theorem
\ref{main-thm} is that

\begin{theo}\label{critical}
 Let $(M, g)$ be a negatively curved compact connected Riemannian manifold.  If it is locally symmetric, then for any $C^3$ curve  $\lambda\in (-1, 1)\mapsto g^{\lambda}\in \Re(M)$  of  conformal changes of the metric $g^0=g$ with constant volume, $$(dh_{\lambda}/d\lambda)|_{\lambda=0}=0, \quad (d\ell_{\lambda}/d\lambda)|_{\lambda=0}=0.$$
\end{theo}

In case $M$  is a Riemannian surface, the stochastic entropy remains
the same for $g\in \Re(M)$ with constant volume. This is because any
$g\in \Re(M)$ is a conformal change of a metric with constant
curvature by the Uniformization Theorem, metrics with the same
constant curvature have the same stochastic entropy by
(\ref{formulas-ell-h}) and the constant curvature depends only on
the volume by the Gauss-Bonnet formula.  Indeed, our formula (Theorem
\ref{Main-formulas}, (\ref{entropyderivative})) yields
$dh_{\lambda}/d\lambda\equiv 0$ in the case of surfaces if the
volume is constant.

When $M$ has dimension at least $3$, it is interesting to know
whether the converse direction of Theorem \ref{critical} for the stochastic entropy
holds.  We have the following question.

\emph{ Let $(M, g)$ be a negatively curved compact connected
Riemannian manifold with dimension greater than $3$. Do we have that
$(M, g)$ is  locally symmetric if and only if
 for any $C^3$ curve  $\lambda\in (-1, 1)\mapsto g^{\lambda}\in \Re(M)$ of constant volume with $g^{0}=g$,  the mapping $\lambda\mapsto h_{\lambda}$  is differentiable and has a critical point at
 $0$?}

We will present the proof of  Theorem \ref{main-thm} and the above
discussion in a more general setting. Indeed, whereas the statements so far deal only with the
Brownian motion on $\M$, proofs of the limit theorems such as
(\ref{Intro-d-Green-formula}) or (\ref{intro-CLT}) involve the
laminated Brownian motion associated with the stable foliation of
the geodesic flow on the unit tangent bundle $\p : SM \to M$.    As recalled
in Section 2.1, the {\it {stable foliation}} $\W$ of the geodesic
flow is a H\"older continuous lamination, the leaves of which are
locally identified with $\M$. A differential operator $\mathcal{L}$
on (the smooth functions on) $SM$ with continuous coefficients and
$\LL 1=0$ is said to be \emph{subordinate to the stable foliation
$\mathcal{W}$}, if for every smooth function $F$ on $SM$ the value
of $\LL(F)$ at $v\in SM$ only depends on the restriction of $F$ to
$W^s(v)$.  We are led to consider the  family $\LL^\l $  of
subordinated operators to the stable foliation, given, for $F$
smooth on $SM$, by \[\LL ^\l F \; = \; \D F + (m-2) \langle \nabla
(\vf ^\l \circ \p), \nabla F \rangle,
\] where Laplacian, gradient and scalar product are taken along the
leaves of the lamination and for the metric lifted from the metric
$\wt g$ on $\M$. Diffusions associated to a general subordinated
operator of the form $\D + Y$, where $Y$ is a laminated vector
field, have been studied by  Hamenst\"{a}dt (\cite{H2}).  We recall her results and
several tools in Section 2. In particular, the diffusions associated
to $\LL ^\l$ have a drift $\overline \ell _\l$ and an entropy
$\overline h_\l$ that coincide with respectively $\widehat \ell _\l $
and $\widehat h_\l .$  Convergences (\ref{Intro-d-Green-formula}) and
(\ref{intro-CLT}) are now natural in this framework. Then,  our
strategy is to construct all the laminated diffusions associated to
the different $\l$ and starting from the same point on the same
probability space and to compute the necessary limits as
expectations of quantities on that probability space that are
controlled  by probabilistic arguments. For each ${\bf{v}} \in S\M,$
the stable manifold $W^s(\bf v)$ is identified with $\M$ (or a
$\Z$ quotient of $\M$).\footnote{When $\bfv$ is on a periodic orbit, then $W^s(\bf v)$  is a cylinder identified with the quotient of $\M$ by the action of one element 
of $G$
 represented by the closed geodesic.} As recalled
in Subsection 2.5, the diffusions are constructed on $\M$ as
projections of solutions of stochastic differential equations  on
the orthogonal frame bundle $O(\M)$ with the property that
only the drift part depends on $\l$ (and on $\bf v$). The quantities
$\overline \ell _\l$ and $\overline h_\l$ can be read now on the
directing probability space, so that we can compute $(\rm {\bf
{II}})_\ell $ and $({\rm {\bf{II}}})_h$ in Section 4.   We cannot do
this computation in such a direct manner for a general perturbation
$\l \mapsto g^{\l} \in \Re(M)$ and this is the reason why we restrict
our analysis in this paper to the case of conformal change. But
the idea of analyzing the linear drift and   the  stochastic entropy using
the stochastic differential equations can be further polished to
treat the general case (\cite{LS3}).

% In  $\mbox{Sec.}\  2$, we give  some  basic properties  of  harmonic measures
%for an operator $\LL$  subordinate to the stable foliation from \cite{H2}. The corresponding formulas and %the Central Limit Theorems for the linear drift and the stochastic entropy for $\LL$  will appear in %$\mbox{Sec.}\  3$  (Proposition \ref{formulas-l-h-Y} and  Proposition \ref{CLT}).  In $\mbox{Sec.}\  4$, we  $proceed to  study the regularity of the linear drift and the stochastic entropy
%under the change of  the drift part of the operator $\LL$ (Theorem \ref{differential-ell} and Theorem %\ref{differential-h}).    As an application, we will obtain
%(\ref{entropy-2}).    In $\mbox{Sec.}\  5$, we will first introduce a special Morse correspondence map and %characterize it using stable and unstable Jacobi fields  (Proposition \ref{a-b-prop} and Proposition
%\ref{a-b-better in K}).  This step is important to  valid  the differential in  (\ref{Intro-X-lambda}) (Proposition %\ref{lem-W}  and Corollary \ref{X-lam-Lem}).

We thus   will obtain explicit formulas 
for $(d\ell_{\lambda}/d\lambda)|_{\lambda=0}$ and $(d
h_{\lambda}/d\lambda)|_{\lambda=0}$ in Theorem \ref{Main-formulas},
which, in particular,  will imply Theorem \ref{main-thm}.  Finally,
Theorem \ref{critical} can be deduced either using the formulas in
Theorem \ref{Main-formulas} or merely using Theorem \ref{main-thm}
and the existing  results concerning the regularity of volume
entropy for compact negatively curved spaces under conformal changes
from \cite{Ka, KKPW}.

%\footnote{Don't know how many readers survive after the previous
%long paragraph. So, I add one paragraph.}
We will arrange the paper as follow. Section 2 is to
introduce the linear drift and stochastic entropy for a laminated
diffusion of the unit tangent bundle with generator $\Delta+Y$
(\cite{H2}) and to understand them by formulas using pathwise limits
and integral formulas at the boundary, respectively. There are two
key auxiliary properties for the computations of  the differentials
of  $\widehat{\ell}^{\lambda}, \widehat{h}^{\lambda}$   in $\lambda$: one is
the Central Limit Theorems for the linear drift and the stochastic
entropy; the other is the probabilistic pathwise relations between the distributions of the diffusions
of different generators. They will be addressed in Subsections \ref{Subsection-CLT} and \ref{4.1-Girsanov}, respectively.  In Section 3, we will compute separately the
differentials of the linear drift and the stochastic entropy associated
to a one-parameter of laminated diffusions with generators
$\Delta+Y+Z^{\lambda}$. Section 4 is to use the infinitesimal Morse
correspondence (\cite{FF}) to derive the differential
$\lambda\mapsto \overline{X}^{\lambda}$ for any general $C^3$ curve
$\lambda\mapsto g^{\lambda}$ contained in $\Re (M)$. The last
section is devoted to the proof of Theorem \ref{main-thm} and
Theorem \ref{critical} as was mentioned in the previous paragraph.

\section{Foliated diffusions}

In this section, we recall results from the literature and we fix
notations about the stable foliation in negative sectional curvature, the
properties of  the diffusions subordinated to the stable foliation
and the construction of these diffusions as solutions of SDE.

\subsection{Harmonic measures for the stable foliation}

Recall that $(\M, \wt{g})$ is the universal cover space of $(M, g)$, a
negatively curved $m$-dimensional closed connected Riemannian
manifold with fundamental group $G$. Two geodesics in $\M$ are said to be equivalent if they remain a
bounded distance apart and the space of equivalent classes of unit
speed geodesics is the geometric boundary $\partial\M$. For each
$(x, \xi)\in \M\times \pp\M$,  there is a unique unit speed geodesic
$\g_{x,\xi}$ starting from $x$ belonging to $[\xi]$, the equivalent
class of $\xi$. The mapping $\xi\mapsto \dot{\g}_{x, \xi}(0)$ is a
homeomorphism $\pi_{x}^{-1}$ between $\partial \wt{M}$ and the unit
sphere $S_{x}\wt{M}$ in the tangent space at $x$ to $\wt{M}$.  So we
will identify $S\wt{M}$, the unit tangent bundle of $\M$,  with
$\M\times
\partial\M$.

Consider the geodesic flow ${\bf \Phi}_t$ on  $S\wt{M}$.  For each
${\bf v}=(x, \xi)\in S\M$, its \emph{stable manifold} with respect to ${\bf
\Phi}_t$, denoted $ W^s({\bf v})$,  is the collection of initial
vectors ${\bf w}$ of  geodesics $\g_{{\bf w}}\in [\xi]$  and can be
identified with $\M\times\{\xi\}$.   Extend the action of $G$ continuously to $\pp\M$.  Then $SM$,  the
unit tangent bundle of $M$,  can be identified with the quotient of
$\M\times
\partial\M$ under the diagonal action of $G$.  Clearly, for $\psi \in G$, $\psi(W^s({\bf v})) = W^s({D\psi ({\bf v}}))$ so that the collection of
 $W^s({\bf v})$ defines a lamination $\W$ on $SM$, the so-called \emph{stable foliation} of $SM$.
 The leaves of the stable foliation $\W$ are discrete quotients of $\M$, which are naturally endowed with the Riemannian metric induced from $\wt g$.  For ${v}\in SM$, let $W^s(v)$ be the leaf of $\W$ containing $v$.
Then $W^s(v)$ is a $C^2$
immersed submanifold of $SM$ depending continuously on $v$ in the
$C^2$-topology (\cite{SFL}). (More properties of the stable foliation and of  the geodesic flow  will appear in Section \ref{Sec-integral formula}.)

Let  $\mathcal{L}$ be an operator on (the smooth functions on) $SM$
with continuous coefficients which is subordinate to the stable
foliation $\mathcal{W}$. A Borel measure ${\bf m}$ on $SM$ is called
\emph{$\LL$-harmonic}  if it satisfies
\[
\int \LL(f)\ d{\bf m}=0
\]
for every smooth function $f$ on $SM$.  If the restriction of $\LL$ to each leaf is elliptic, it is true by \cite{Ga} that  there always exist
harmonic measures and the set of harmonic probability measures is a non-empty
weak$^*$ compact convex set of measures on $SM$. A harmonic
probability measure ${\bf m}$ is \emph{ergodic} if it is extremal
among harmonic probability measures.

In this paper, we are interested in the case $\LL=\Delta+Y$,
where $\Delta$ is the laminated Laplacian  and $Y$ is a section
of the tangent bundle of $\W$ over $SM$ of class $C^{k, \alpha}_{s}$
for some $k\geq 1$ and $\alpha\in [0, 1)$ in the sense that $Y$ and
its leafwise jets up to order $k$ along the leaves of $\W$ are
H\"{o}lder continuous with exponent $\alpha$ (\cite{H2}). Let ${\bf
m}$ be an $\LL$-harmonic  measure. We can characterize it by
describing its lift on $S\M$.

Extend $\LL$ to a $G$-equivariant operator on $S\wt{M}=\M\times
\partial \M$ which we shall denote with the same symbol. It defines
a Markovian family of probabilities on $\wt
{\Omega}_{+}$, the space of %\footnote{Since we only use the paths in $\wt{\Om}_+$ in this section. So I correct their elements to be $\wt{\om}$.  Then I can add one line next page to emphasize $\om$ denotes the paths in $SM$ and whenever we take integrals on functions on $\om$ means we restrict ourselves to $M_0\times \partial \M$.}
paths of $\wt{\om}: [0, +\infty)\to S\M$,  equipped with the smallest
$\sigma$-algebra $\A$ for which the projections $R_t:\ \wt{\om}\mapsto
\wt{\om}(t)$ are measurable. Indeed, for ${\bf{v}}=(x, \xi)\in S\M$, let
$\LL_{{\bf v}}$ denote the laminated operator of $\LL$ on $W^s({\bf
v})$. It can be regarded as an operator on $\M$ with corresponding
heat kernel functions  $p_{{\bf v}}(t, y, z)$, $t\in \Bbb R_{+}, y, z\in \M$.
Define
\[
{\bf p}(t, (x, \xi), d(y, \eta))=p_{{\bf v}}(t, x,
y)d{\rm{Vol}}(y)\delta_{\xi}(\eta),
\]
where $\delta_{\xi}(\cdot)$ is the Dirac function at $\xi$. Then the
diffusion process on $W^s({\bf v})$ with infinitesimal operator
$\LL_{{\bf v}}$ is given by a Markovian family $\{\P_{{\bf
w}}\}_{{\bf w}\in \M\times \{\xi\}}$, where for every $t>0$ and
every Borel set $A\subset \M\times \partial\M$ we have
\[
\P_{{\bf w}}\left(\{\wt{\om}:\ \wt{\om}(t)\in A\}\right)=\int_{A}\bfp(t, {\bf
w}, d(y, \eta)).
\]

The following concerning $\LL$-harmonic measures holds true.

\begin{prop}\label{harmonic-mea}(\cite{Ga,H2}) Let $\wt{{\bf m}}$ be the $G$-invariant measure which extends
an  $\LL$-harmonic  measure ${\bf m}$ on $\M\times \partial \M$. Then
\begin{itemize}
\item[i)] the measure $\wt{{\bf m}}$ satisfies, for all $f\in C_{c}^{2}(\wt{M}\times \partial
  \M)$,
  \[
\int_{\scriptscriptstyle{\M\times
\partial\M}}\left(\int_{\scriptscriptstyle{\M\times \partial\M}}f(y,
\eta)\bfp(t, (x, \xi), d(y, \eta))\right) d\wt{{\bf m}}(x,
\xi)=\int_{\scriptscriptstyle{\M\times \partial\M}}f(x, \xi)\
d\wt{{\bf m}}(x, \xi);
  \]
  \item[ii)] the measure $\wt \P=\int \P_{{\bf v}}\ d\wt{{\bf m}}({\bf v})$ on
  $\wt{\Om}_{+}$ is invariant under the shift map $\{\sigma_t\}_{t\in
  \R_+}$ on $\wt{\Om}_+$,  where $\sigma_t(\wt{\om}(s))=\wt{\om}(s+t)$ for $s>0$ and $\wt{\om}\in
  \wt{\Om}_{+}$;
   \item[iii)] the measure $\wt{{\bf m}}$ can be
expressed locally at ${\bf v}=(x, \xi)\in S\M$ as $d\wt{{\bf
m}}=k(y, \eta)(dy\times d\nu(\eta))$, where $\nu$ is a finite
measure on $\partial \M$ and, for $\nu$-almost every $\eta$, $k(y,
\eta)$ is a positive function on $\M$ which satisfies $\Delta(k(y,
\eta))-{\rm{Div}}(k(y, \eta)Y(y, \eta))=0$.
\end{itemize}
\end{prop}

The group $G$ acts naturally and discretely on the space $ \wt \Om
_+$ of continuous paths in $S\M$ with quotient  the space $\Om _+$ of
continuous paths in $SM$, and this action commutes with the shift
$\s_t, t\geq 0$. Therefore, the measure $\wt \P$ is the extension of
a finite, shift invariant  measure $\P$ on $\Om _+.$ Note that $SM$ can be identified with $M_0\times \partial\M$, where $M_0$ is a fundamental domain of $\M$. Hence we can also identify $\Om_+$ with the lift of its elements  in $\wt{\Om}_+$ starting from $M_0$. Elements in $\Om_+$ will be denoted by $\om$.   We will also clarify the notions whenever there is an ambiguity.  In all the paper, we will normalize the harmonic measure $\bf m$ to be a probability measure, so that the measure $\P$ is also a probability measure. We denote by $\E_\P$ the corresponding expectation symbol.

Call
$\LL$ \emph{weakly coercive},  if $\LL_{{\bf v}}$, ${\bf v}\in S\M$,
are weakly coercive in the sense that there are a number $\e>0$
(independent of ${\bf v}$)  and, for each $\bf v$, a positive
$(\LL_{{\bf v}}+\e)$-superharmonic function $F$ on $\M$ (i.e. $(\LL_{{\bf v}}+\e)F\geq 0$).  For instance,
if $Y\equiv 0$, then $\LL=\Delta$ is weakly coercive and it has a
unique $\LL$-harmonic measure ${\bf m}$, whose lift in $S\M$
satisfies  $d\wt{{\bf m}}=dx\times d\wt{{\bf m}}_x$,  where $dx$ is
proportional to the volume element and $\wt{{\bf m}}_x$ is the
hitting probability at $\pp\M$ of the Brownian motion starting at
$x$.  Consequently, in this case, the function $k$ in Proposition
\ref{harmonic-mea} is the Martin kernel function.  This relation is
not clear for general weakly coercive $\LL$. %\footnote{Actually, we do not use Proposition \ref{harmonic-mea} iii), except implicitly page 54 and 55, when we use the foliated integration by parts formula. Then it is for $\bf m $ or ${\bf {m}}^\l $, so this remark is important...  Do you want to elaborate, or leave the clever reader figure out by him/her self? Don't forget that the clever reader might be us (you) in 20 years! But what you expect me to say. Even the referee does not know what is $k$ here. I think the remarks before Proposition 2.16 is enough. }

A nice property for the laminated diffusion associated with a weakly coercive operator is that  the semi-group $\s_t, t\geq 0,$ %\footnote{Should we add comma after it? Several places this page. Yes}
of transformations of $\Om_+$ has strong ergodic properties with
respect to the probability $\P$,  provided $Y$ satisfies some mild condition. Recall that a measure preserving
semi flow $\s_t, t\geq 0,$ of transformations of a probability  space
$(\Om, \P)$ is called mixing if for any bounded measurable functions
$F_1,F_2$ on $\Om$, 
\[ \lim\limits_{t\to +\infty} \E_{\P} (F_1\,(F_2\circ \s_t)) \; = \; \E_{\P}( F_1) \E_{\P}(F_2) .\]

\begin{prop}\label{mixing}Let $\LL= \D + Y $ be subordinated to the stable foliation and  such
that $Y^*$, the dual of $Y$ in the cotangent bundle to the stable foliation  over  $SM$,
%\footnote{I miss that one at the previous readings. For us (and for \cite{H2}),  \emph{$Y$ is a section
%of the tangent bundle of $\W$ over $SM$ of class $C^{k, \alpha}_{s}$
%for some $k\geq 1$ and $\alpha\in [0, 1)$ in the sense that $Y$ and
%its leafwise jets up to order $k$ along the leaves of $\W$ are
%H\"{o}lder continuous with exponent $\alpha$} (see previous page). In particular, it is not globally $C^1$, so the dual $Y^\ast$ can only be defined leafwise, for the metric lifted from $\M$, and the hypothesis is that this one-form, seen as a 1-form on $\M$, is closed. This is the meaning of `leafwisely', but maybe, we should be more explicit?---Yes, agree. we need to explain the condition on $Y$ and $Y^*$. It is too abstract. I can't understand very well. Let me promote your footnote to be the context first.}
  satisfies $dY^*=0$ leafwisely. Assume that $\LL$ is weakly coercive.
Let $\bf {m}$ be the unique invariant measure, $\P$  the associated
probability measure on $\Om _+$. The shift semi-flow $\s_t, t \geq 0, $ is mixing on
$(\Om _+ , \P).$
\end{prop}

%\lin{Note that $Y$,  a section
%of the tangent bundle of $\W$ over $SM$ of class $C^{k, \alpha}_{s}$,  it is not necessarily globally $C^1$, so the dual $Y^\ast$ can only be defined leafwise, for the metric lifted %from $\M$, and the hypothesis is that this 1-form, seen as a 1-form on $\M$, is closed. }
(Note that $Y$ is   a section
of the tangent bundle of $\W$ over $SM$ of class $C^{k, \alpha}_{s}$ and that $Y^\ast$ is a section
of the cotangent bundle of $\W$ over $SM$ of class $C^{k, \alpha}_{s}$, the duality being defined by the metric  inherited  from $\M$. The hypothesis is that this 1-form, seen as a  
 1-form on  $\M$, is closed.)
 %\footnote{I add ().  Another question, it is $\M$ or $M$? Are you considering each slice of $SM$ as a quotient of $\M$? Then is my correction OK? No, the 1-form is not $G-$invariant  as a form on $\M$. For instance, he value of $\nabla u_0 $  at $gx , x\in M_0 , g \not = e,$ is $\nabla u_0 (gx, \xi) = \nabla u_0 (x, g^{-1} \xi )$. }

\begin{proof} The classical proof that  a weakly coercive subordinated operator admits a unique harmonic measure
(see \cite{Ga}, \cite{L5}, \cite{Y}
 for the case of $\D$) shows in fact the mixing property if $F_1$ and $F_2$ are  functions on $\Om _+$
  that depends only on the starting point of the  path and are continuous as functions on $SM$.
  The mixing property is extended first  to  bounded measurable  functions on $\Om _+$ that depends only
   on the starting point of the  path by ($L^2$, say) density, then to functions depending on a finite number
    of coordinates in the space of paths by the Markov property and finally to all bounded measurable functions by $L^2$ density.
    \end{proof}
\subsection{ Linear drift and stochastic entropy for laminated diffusion}

 Let ${\bf m}$ be an $\LL$-harmonic measure and $\wt{{\bf m}}$ be its
  $G$-invariant extension 
 in $S\M$. Choose a fundamental domain $M_0$ of $\M$ and identify $SM$ with
$M_0\times \pp\M$. We normalize $\wt{{\bf m}}$ so that $\wt{{\bf
m}}(M_0\times \pp\M)=1$ (so that  the measure $\P$ is a probability;
we denote  $\E_\P$ the corresponding expectation). Let $d_{\W}$ denote
the leafwise metric on the stable foliation of $S\M$. Then it can be
identified with $d_{\wt{g}}$ on $\M$ on each leaf. We define
\begin{eqnarray*}
  && \ell_{\LL}({\bf m}):=\lim\limits_{t\to +\infty}\frac{1}{t}\int_{\scriptscriptstyle{M_0\times \pp\M}}d_{\W}((x,\xi), (y, \eta))\bfp(t, (x, \xi), d(y, \eta))\ d\wt{{\bf m}}(x, \xi),\\
 && h_{\LL}({\bf m}):=\lim\limits_{t\to +\infty}-\frac{1}{t}\int_{\scriptscriptstyle{M_0\times \pp\M}}
\left(\ln \bfp(t, (x, \xi), (y, \eta))\right) \bfp(t, (x, \xi), d(y,
\eta))\ d\wt{{\bf m}}(x, \xi).
\end{eqnarray*}
Equivalently, by using $\wt \P$ in Proposition \ref{harmonic-mea},
we see that
\begin{eqnarray*}
&& \ell_{\LL}({\bf m})=\lim\limits_{t\to +\infty}\frac{1}{t}\int_{\scriptscriptstyle \om(0)\in M_0\times\pp\M} d_{\W}(\om(0), \om(t))\ d \wt\P(\om) = \lim\limits_{t\to +\infty}\frac{1}{t}\E_\P(d_{\W}(\om(0), \om(t))),\\
&& h_{\LL}({\bf m})=\lim\limits_{t\to
+\infty}-\frac{1}{t}\int_{\scriptscriptstyle \om(0)\in
M_0\times\pp\M}\ln \bfp(t, \om(0), \om(t))\ d\wt \P(\om) =\lim\limits_{t\to +\infty}-\frac{1}{t}\E_\P(  \ln \bfp(t, \om(0), \om(t))).
\end{eqnarray*}
Call $\ell_{\LL}({\bf m})$ the \emph{linear drift of $\LL$ for ${\bf m}$},
and $h_{\LL}({\bf m})$ the \emph{ (stochastic) entropy of $\LL$ for ${\bf m}$}.  In case
there is a unique $\LL$-harmonic measure ${\bf m}$,  we will write
$\ell_{\LL}:=\ell_{\LL}({\bf m})$ and $h_{\LL}:=h_{\LL}({\bf m})$
and call them the \emph{linear drift} and the \emph{(stochastic) entropy for $\LL$}, respectively.

Clearly, $h_{\LL}({\bf m})$ is nonnegative by definition.  We are
interested in the case that $h_{\LL}({\bf m})$ is positive. When
$\LL=\Delta$, this is true (\cite[Theorem 10]{K1}).  % (Actually,  note that $G=\pi_1(M)$ is non-amenable for any compact connected negatively curved $M$ and hence $\lambda_0$, the bottom of the spectrum of Laplacian, is positive by Brooks' result (\cite{Br81}). So, the entropy $h_{\D}({\bf m})$, which is not smaller than $2 \lambda_0$(\cite{K1}), is positive as well.)
%\footnote{Could we just cite \cite{K1} for $h_\D >0$?  Yes, if it is proved there! Did he?? I can't find the paper. }
 In general, there exist weakly coercive
$\LL$'s which admit uncountably many harmonic measures with zero
entropy (\cite{H2}). 

 Let $\LL$ be such
that $Y^*$, the dual of $Y$ in the cotangent bundle  to the stable foliation  over $SM$,
satisfies $dY^*=0$ leafwisely. For $v\in SM$,  let $\overline X(v) $ be the tangent vector to $W^{s}(v)$   that projects on $v$ and let%\footnote{In this paper ${\bfv}$ denotes elements from $S\M$ and $v\in SM$, I didn't use $\rm v$ for the latter since it is similar to $\bf v$. If you consider everything on $SM$, then I think it should be $v$. Following this, maybe we use $\Phi_t$ for the geodesic flow on $SM$ and $\bf \Phi_{t}$ for the geodesic flow on $S\M$.}
\[\mbox{pr}(-\langle \overline{X}, Y\rangle):=\sup\left\{h_{\mu}-\int \langle \overline {X},
Y \rangle \ d\mu:\ \ \mu\in \mathcal{M}\right\} \] be the pressure
of the function $-\langle \overline {X}, Y\rangle$ on $SM$ with respect to the
geodesic flow ${ \Phi}_t$,   where $\mathcal{M}$ is the set of
${ \Phi}_t$-invariant probability measures on $SM$ and $h_{\mu}$ is the entropy
of $\mu$ with respect to ${ \Phi}_t$.    Then,
\begin{prop}[\cite{H2}]% \label{weaklycoercive}
Let $\LL= \D + Y $ be subordinated to the stable foliation and  such
that $Y^*$, the dual of $Y$ in the cotangent bundle  to the stable foliation over  $SM$,
satisfies $dY^*=0$ leafwisely. Then,
$h_{\LL}({\bf m})$ is positive if and only if $\mbox{pr}(-\langle \overline{X},
Y\rangle)$ is positive, and each one of the two  positivity properties
 implies that $\LL$ is weakly coercive, ${\bf m}$ is the unique
$\LL$-harmonic measure and $\ell_{\LL}({\bf m})$ is positive.
\end{prop}
In particular, when we  consider $\D + Z^\l , $ where $Z^\l :=
(m-2)\nabla (\vf ^\l \circ \p)$ and $\vf^{\l}$'s are real valued functions on $M$ such that  $(\l, x)\mapsto \vf^{\lambda}(x)$ is $C^3$ and $\vf^0\equiv 0$,  the pressure of $-\langle Z^\l , \overline {X}\rangle
$ is positive for $\l$ close enough to 0.

\subsection{Linear drift and stochastic entropy for laminated diffusions: pathwise limits}
By ergodicity of the shift semi-flow,  it  is possible to  evaluate the  linear drift and stochastic entropy along typical paths. Let $\LL=\Delta+Y$ be  such
that $Y^*$, the dual of $Y$ in the cotangent bundle  to the stable foliation over  $SM$,
satisfies $dY^*=0$ leafwisely  and
$\mbox{pr}(-\langle \overline{X}, Y\rangle)>0$. Let ${\bf m}$ be the unique
$\LL$-harmonic measure.  By Proposition \ref{mixing} the measure $\P$ associated to $\bf m$ is ergodic for the shift flow on $\Om _+$. The following well known fact follows then from Kingman's Subadditive  Ergodic Theorem  (\cite{Ki}).  For
$\P$-almost all paths $\om\in \Om_+$, we still denote by $\om $ its lift in $\wt \Om _+$ with $\om (0) \in M_0$ and we have
\begin{equation}\label{Dis-ell}
\lim_{t\to +\infty}\frac{1}{t}d_{\W}(\om(0),
  \om(t))=\ell_{\LL}.
\end{equation}
Similarly, we can characterize $h_{\LL}$  using the Green function
along the trajectories. For each ${\bf v}=(x, \xi)\in \M\times
\partial\M$, we can regard $\LL_{\bf v}$ as an operator on $\M$.
Since it is weakly coercive, there exists the corresponding Green
function $G_{{\bf v}}(\cdot, \cdot)$ on $\M\times \M$,  defined for $x \not = y$ by
\[ G_{{\bf v}} (x,y) := \int_0^\infty  p_{\bf{v}}(t,x,y) \, dt.\]
 Define the
\emph{Green function $\bfg(\cdot, \cdot)$ on $S\M\times S\M$} as being
\[
\bfg((y, \eta), (z, \zeta)):=G_{(y, \eta)}(y,
z)\delta_{\eta}(\zeta), \ \ \mbox{for}\ (y, \eta), (z, \zeta)\in
S\M,
\]
where $\delta_{\eta}(\cdot)$ is the Dirac function at $\eta$.  We
have the following proposition concerning $h_{\LL}$.

\begin{prop}\label{prop-h-in p-g} Let $\LL=\Delta+Y$ be  such
that $Y^*$, the dual of $Y$ in the cotangent bundle  to the stable foliation over  $SM$,
satisfies $dY^*=0$ leafwisely and $\mbox{pr}(-\langle \overline{X}, Y\rangle)>0$.  Then for $\P$-a.e. paths
$\om\in \Om_+$, we have
\begin{eqnarray}
h_{\LL}&=&\lim\limits_{t\to +\infty}-\frac{1}{t} \ln \bfp(t, \om(0),
\om(t))\label{h-in p-t}\\
&=&\lim\limits_{t\to +\infty}-\frac{1}{t}\ln \bfg(\om(0),
\om(t)).\label{h-in g}
\end{eqnarray}
\end{prop}
Contrarily to the distance, the function  $-\ln \bf p$ is not elementarily subadditive along the trajectories and the argument used to establish (\ref{Dis-ell})  has to be modified.  
We will use
the trick of \cite{L6} to show that there exists a convex function
$h_{\LL}(s)$, $s>0$,  such that  for $\P$-a.e. paths $\om\in \Om_+$,
for any $s>0$,
\begin{eqnarray}\label{h-s-in p-t}
h_{\LL}(s)=\lim\limits_{t\to +\infty}-\frac{1}{t} \ln \bfp(st,
\om(0), \om(t)).
\end{eqnarray}
Setting $s=1$ in (\ref{h-s-in p-t}) gives that $\lim_{t\to +\infty}-\frac{1}{t} \ln \bfp(t, \om(0),
\om(t))$ exists and is $h_{\LL} (1) $. Moreover, $h_{\LL} (1) \leq h_{\LL}$ by Fatou's Lemma. Then, (\ref{h-in g}) and (\ref{h-in p-t})  will follow once we show that for $\P$-almost all paths $\om\in \Om_+$,
\begin{equation}\label{G-inf-h-s}
\lim\limits_{t\to +\infty}-\frac{1}{t}\ln \bfg(\om(0),
\om(t)) =  \inf_{s>0}\{h_{\LL}(s)\}\geq h_{\LL}.
\end{equation}

    To show (\ref{h-s-in p-t}) and (\ref{G-inf-h-s}), we need some detailed descriptions of $p_{{\bf v}}(t, x,
y)$. First, we have a variant of Moser's parabolic Harnack
inequality (\cite{Mos})  (see \cite{Sta,T} and also
    \cite{Sa}).

\begin{lem}\label{Harnack}
  There exist $A, \varsigma>0$ such that for  any ${\bf v}\in S\M$, $t\geq 1$, $\frac{1}{2}\leq t'\leq
  1$, $x, x', y, y'\in \M$ with $d(x, x')\leq \varsigma$, $d(y, y')\leq
  \varsigma$,
  \begin{equation}\label{equ-harnack}
p_{{\bf v}}(t, x, y)\geq A p_{{\bf v}}(t-t', x', y').
  \end{equation}
\end{lem}

Next, we have the exponential decay property of $p_{{\bf v}}(t, x,
y)$ in time $t$.

\begin{lem}(\cite[p.76]{H2})%\label{p-t-e}
  There exist $B, \e>0$
independent of ${\bf v}$ such that
\begin{equation}\label{p-t-x-y} p_{{\bf v}}(t, x,
y)\leq B\cdot e^{-\e t}, \ \mbox{for all}\ y\in \M\ \mbox{and} \
t\geq 1.
\end{equation}
\end{lem}

%\footnote{I exchange the following two lemmas.}
A Gaussian like upper bound  for $p_{{\bf v}}(t, x, y)$ is also
valid. 
\begin{lem}(\cite[Theorem 6.1]{Sa})\label{Sa-Thm-6.1}There exist constants $C_1, C_2,
K_1$ such that for any ${\bf v}\in S\M$,  $t>0$ and $x, y\in \M$, we
have
\begin{equation*}\label{S-p-t-upper-b}
p_{{\bf v}}(t, x, y)\leq\frac{1}{\rm{Vol}(x, \sqrt{t})\rm{Vol}(y,
\sqrt{t})}\exp\left[C_1(1+bt+\sqrt{K_1 t})-\frac{d^2(x, y)}{C_2
t}\right].
\end{equation*}
\end{lem}

 Let $b>0$ be an
upper bound of $\|Y\|$. We have the following lower bound for $p_{{\bf
v}}(t, x, y)$.

\begin{lem}(\cite[Theorem 3.1]{Wa}) Let $\beta=\sqrt{K}(m-1)+b$,  where $K\geq 0$ is such that $\mbox{Ricci}\geq -K (m-1)$.  Then for  any ${\bf v}\in S\M$,  $t, \sigma>0$ and $x, y\in \M$, we have
\begin{equation}\label{W-p-t-lower-b}
p_{{\bf v}}(t, x, y)\geq (4\pi
t)^{-\frac{m}{2}}\exp\left[-(\frac{1}{4t}+\frac{\sigma}{3\sqrt{2t}})d^2(x,
y)-\frac{\beta^2
t}{4}-\left(\frac{\beta^2}{4\sigma}+\frac{2m\sigma}{3}\right)\sqrt{2t}\right].
\end{equation}
\end{lem}

\bigskip

\begin{proof}[Proof of Proposition \ref{prop-h-in p-g}] We first
show (\ref{h-s-in p-t}). Given $s>0$, for $\om\in \Omega_+$, define
\[
F(s, \om, t):=-\ln ({\bf p}(st-1, \om(0), \om(t))\cdot \wt{A}),
\]
where $\wt{A}=A^2\inf_{z\in \M}{\rm{Vol}} (B(z, \varsigma))$ and $A,
\varsigma$ are as in Lemma \ref{Harnack}.  Then for $t, t'\geq 1/s$,
$\om\in \Omega_+$,
\begin{equation*}
F(s, \om, t+t')\leq F(s, \om, t)+ F(s, \sigma_t(\om), t').
\end{equation*}
This follows by the semi-group property of ${\bf p}$ and
(\ref{equ-harnack}) since
\begin{eqnarray*}
{\bf p}\left(s(t+t')-1, \om(0), \om(t+t')\right)&=&\int {\bf p}(st-\frac{1}{2}, \om(0), z){\bf p}(st'-\frac{1}{2}, z, \om(t+t'))\ dz\\
&\geq&\int_{\scriptscriptstyle B(\om(t), \varsigma)}{\bf p}(st-\frac{1}{2}, \om(0), z){\bf p}(st'-\frac{1}{2}, z, \om(t+t'))\ dz\\
&\geq&  \wt{A} {\bf p}(st-1, \om(0), \om(t)){\bf p}(st'-1, \om(t),
\om(t+t')).
\end{eqnarray*}
For $0<t_1<t_2<+\infty$,  by (\ref{W-p-t-lower-b}), there exists a
constant $C>0$, depending on $t_1, t_2$ and the curvature bounds,
such that  for any ${\bf v}\in S\M$, $x, y\in \M$, any $t, t_1 \leq
t \leq t_2$,
\[
C \exp\left[-(\frac{1}{4t_1}+\frac{\sigma}{3\sqrt{2t_1}})d^2(x,
y)\right]\leq p_{{\bf v}}(t, x, y).
\]
As  a consequence,  we have
\[
\E \left(\sup\limits_{1+\frac{1}{s}\leq t\leq 2+\frac{1}{s}} F(s,
\om,
t)\right)\leq (\frac{1}{4s}+\frac{\sigma}{3\sqrt{2s}})\E\left(\sup\limits_{1+\frac{1}{s}\leq
t\leq 2+\frac{1}{s}}d^2(\om(0), \om(t))\right)-\ln (C\wt{A}),
\]
where the second expectation term is bounded by a multiple of its value in  a hyperbolic
space with curvature the lower bound curvature of $M$ and is finite (cf.
\cite{DGM}). So by the Subadditive Ergodic Theorem applied to the subadditive cocycle $F(s, \om , t)$,  there exists
$h_{\LL}(s)$ such that for  $\P$-a.e. $\om\in \Om_+$, and for $\wt {\bf m} $-a.e. $\bf v$,
\begin{equation}\label{h-s-p-t}
h_{\LL}(s)=\lim\limits_{t\to +\infty}-\frac{1}{t} \ln \bfp(st-1,
\om(0), \om(t)) = \lim\limits_{t\to +\infty}-\frac{1}{t} \int_{\M} \bfp _{\bf v} (t,x,y) \ln \bfp _{\bf v} (st-1,
x, y) \,  d{\rm{Vol}}(y).
\end{equation}
Using the semi-group property of ${\bf p}$ and (\ref{equ-harnack})
again,  we obtain that  for $0<a<1$, $s_1, s_2>0$,
\begin{eqnarray*}
&&{\bf p}((as_1+(1-a)s_2)t-1, \om(0), \om(t))\\
&&\ \ \  \  \ \ \ \ \ \ \ \ \ \ \ \ \geq \wt{A}{\bf p}(as_1t-1,
\om(0), \om(at)){\bf p}((1-a)s_2 t-1, \om(at), \om(t)).
\end{eqnarray*}
It follows that $h_{\LL}(\cdot)$ is a convex function on $\Bbb R_+$
and hence is continuous. This allows us to pick up a full measure set of $\omega$ such that (\ref{h-s-in p-t}) holds true for all positive $s$.  Let $D$ be a countable dense subset of $\R_+$. There
is a measurable set $E\subset \Om_+$ with $\P(E)=1$ such
 that for $\om\in E$,  (\ref{h-s-p-t}) holds true for any $s\in D$.
Let $\om\in \Om_+$ be such an orbit. Given any $s_1<s_2$, let $t>0$ be large,  then we have by (\ref{equ-harnack})
that
\[
{\bf p}(s_1t, \om(0), \om(t))\leq A^{(s_1-s_2)t+1}{\bf p}(s_2t-1,
\om(0), \om(t)).\]  So for $s'<s<s''( s', s''\in D)$, and $\om \in
E$,
\begin{eqnarray*}
h_{\LL}(s'')+(s''-s)\ln A&\leq&\liminf_{t\to +\infty}-\frac{1}{t}\ln {\bf p}(st, \om(0), \om(t))\\
&\leq& \limsup_{t\to +\infty}-\frac{1}{t}\ln {\bf p}(st, \om(0), \om(t))\\
&\leq& h_{\LL}(s')-(s-s')\ln A.
\end{eqnarray*}
Letting $s', s''$ go to $s$ on both sides,  it gives (\ref{h-s-in
p-t}) by continuity of the function $h_{\LL}$. 
Moreover, given $\om \in E$, the convergence is uniform  for $s$ in any closed interval $[s_1, s_2], 0 < s_1 < s_2 < +\infty .$

To show the first equality in (\ref{G-inf-h-s}), we use the observation  that for any $t\in \Bbb R_+$, \[ \bfg(\om(0), \om(t))=t\int_{0}^{+\infty} \bfp(st, \om(0), \om(t))\
ds.
\]
Let $s_0 \in (0, \infty ) $ such that $h_{\LL} (s_0 ) = \inf _{s>0 } h_{\LL}(s) .$
For any $\e >0$, there exists $\d, 0 <\d \leq \e,$ such that for $|s-s_0| < \d, h_{\LL} (s) \leq h_{\LL} (s_0) + \e.$ Write
\[
\bfg (\om(0), \om(t))\geq t\int_{s_0+\frac{1}{t}}^{s_0+\delta}\bfp (st, \om(0), \om(t))\ ds
\]  
and note that for $s_0 +\frac{1}{t} <s <s_0 + \d$, $\om \in \Om_+,$ we have as above by (\ref{equ-harnack})
that
\[
{\bf p}(st, \om(0), \om(t))\geq A^{(s-s_0)t+1}{\bf p}(s_0t-1,
\om(0), \om(t)).\]
Moreover, for $t $ large enough and $\om \in E$,  $ {\bf p}(s_0t-1,
\om(0), \om(t)) \geq e^{ -t (h_{\LL} (s_0) + \e)} $. Therefore:
\[ \left(\bf G(\om (0), \om(t)) \right)^{1/t} \geq t^{1/t} A^{1/t} \left( \int _{1/t }^{ \d} A^{st}ds \right)^{1/t} e^{-(h_{\LL} (s_0) + \e)}.\] It follows that for $\om \in E,$
\[
\limsup\limits_{t\to +\infty}-\frac{1}{t}\ln \bfg(\om(0),
\om(t))\leq \inf_{s>0}\{h_{\LL}(s)\}.
\]

For the reverse inequality, we cut the integral  $\int_{0}^{+\infty} \bfp(st, \om(0), \om(t))\
ds $ into three parts.  Fix $\varepsilon _1\in (0,  h_{\LL})$.   We first claim that for $s_1 >0 $ small enough, for $\P$-a.e. paths
$\om\in \Om_+$ and $t$ large enough,
\begin{equation}\label{small-s} 
 \int_{0}^{s_1} \bfp(st, \om(0), \om(t))\ ds\leq
\frac{1}{t} e^{-(\inf_{s>0}\{h_{\LL}(s)\}- \varepsilon_1)t}. \end{equation}
Indeed, by Lemma \ref{Sa-Thm-6.1}, there exists a constant $C'$ such that
\begin{eqnarray}
 \int_{0}^{s_1} \bfp(st, \om(0), \om(t))\ ds &\leq&  C' e^{C't} \int_0^{s_1} \frac{1}{(st)^{m/2}} e^{-\frac {d^2(\om (0) , \om(t))}{ C'st}} \, ds \notag \\ & = & \frac{C' e^{C't}}{t} \int _{1/(s_1t)}^{+\infty } u^{m/2 +2} 
 e^{-\frac {d^2(\om (0) , \om(t))}{ C'} u}\, du \notag \\ &\leq & \frac{C' e^{C't}}{t} \, Q\left(d^2(\om (0) , \om(t) )\right)\, e^{-\frac {d^2(\om (0) , \om(t))}{ C's_1t }}, \label{p-st-x-y-upper-bound}\end{eqnarray}
 where $Q$ is some polynomial of degree $[m/2] +3.$
 For $\P$-a.e. paths  
$\om\in \Om_+$, for large enough $t$,  \[0 <  \frac{\ell _\LL}{2} \leq  \frac{1}{t} d(\om (0) , \om(t)) \leq \frac{3 \ell _\LL}{2} .\]
It follows that for those paths, given $\varepsilon _1\in (0, h_{\LL}) $, for any $s_1\in (0, \frac{\ell_{\LL}^2}{4C'}\cdot \frac{1}{C'+h_{\LL}-\frac{1}{2}\varepsilon_1})$,  the quantity in (\ref{p-st-x-y-upper-bound}) is bounded from above by 
\[
\frac{1}{t}\cdot C'Q\left(d^2(\om (0) , \om(t) )\right)\cdot e^{-(\inf_{s>0}\{h_{\LL}(s)\}-\frac{1}{2}\varepsilon_1)t}.
\]
Consequently,  (\ref{small-s}) is satisfied for those paths, for $t$ large enough.

Then  observe that  for $s_2, t>1$, we
have by (\ref{p-t-x-y}) that  
\begin{eqnarray*}
  \int_{s_2}^{+\infty} \bfp(st, \om(0), \om(t))\ ds\leq  B\int_{s_2}^{+\infty}
  e^{-\varepsilon st}\ ds=\frac{1}{\varepsilon t}Be^{-\varepsilon s_2 t}.
\end{eqnarray*}
So for any $\varepsilon _1\in (0, h_{\LL})$,   if $s_2$ and $t$ are large enough,
then
\[
 \int_{s_2}^{+\infty} \bfp(st, \om(0), \om(t))\ ds\leq
 \frac{1}{t}e^{-(\inf_{s>0}\{h_{\LL}(s)\}-\varepsilon_1)t}.
\]

Moreover, using the uniform convergence in  (\ref{h-s-in
p-t}) on the interval $[s_1, s_2 ]$, we get,  for $\om \in E$ and $t$ large enough,
\begin{eqnarray*}
 \int_{s_1}^{s_2} \bfp(st, \om(0), \om(t))\, ds &\leq&  (s_2 - s_1 ) e^{-\left(\inf_{s>0}\{h_{\LL}(s)\} -\frac{1}{2} \varepsilon _1\right) t}\\
 &\leq&  e^{-\left(\inf_{s>0}\{h_{\LL}(s)\} -\varepsilon _1\right) t}.
\end{eqnarray*}
Putting everything together, we obtain
\begin{eqnarray*}
\liminf\limits_{t\to +\infty}-\frac{1}{t}\ln \bfg(\om(0), \om(t))
\geq \inf_{s>0}\{h_{\LL}(s)\}.
\end{eqnarray*}

Finally, we have $\inf_{s>0}\{h_{\LL}(s)\}\geq h_{\LL}$ since for any
typical ${\bf v}\in SM$,
\begin{eqnarray*}
  h_{\LL}(s)-h_{\LL}&=& \lim\limits_{t\to +\infty}-\frac{1}{t}\int
  p_{{\bf v}}(t, x, y)\ln \frac{p_{{\bf v}}(st, x, y)}{p_{{\bf v}}(t, x,
  y)}\  d{\rm{Vol}}(y)\\
  &\geq& \lim\limits_{t\to +\infty}\frac{1}{t}\int p_{{\bf v}}(t, x,
  y)\left(1- \frac{p_{{\bf v}}(st, x, y)}{p_{{\bf v}}(t, x,
  y)}\right)\  d{\rm{Vol}}(y)\\
  &\geq& 0.
\end{eqnarray*}\end{proof}

%\footnote{I moved the part on Green function here and split formulas into two. What do yo think?}
 \subsection{Linear drift and stochastic entropy for laminated diffusions: integral formulas}\label{Sec-integral formula}
The interrelation between the underlying geometry of the manifold and the linear drift and the stochastic entropy is not well exposed in the pathwise limit expressions (\ref{Dis-ell}) and (\ref{h-in g}). The purpose of this subsection is  to  establish the generalization of formulas (\ref{Intro-hat-ell-h}) for the linear drift and the stochastic entropy, respectively,  and set up the corresponding notations.

We begin with $\ell_{\LL}$.  We will express it using the Busemann function at the geometric boundary and the $\LL$-harmonic measure. Recall the geometric boundary $\partial\M$ of $\M$ is the collection of equivalent classes of geodesics, where two geodesics $\g_1, \g_2$ of $\M$ are said to be equivalent (or asymptotic) if $\sup_{t\geq 0}d(\g_1, \g_2)<+\infty$. Let $\LL=\Delta+Y$ be such  that $Y^*$, the dual of $Y$ in the cotangent bundle  to the stable foliation over  $SM$,
satisfies $dY^*=0$ leafwisely and $\mbox{pr}(-\langle \overline{X}, Y\rangle)>0$.
  For  $\P$-a.e. paths $\om\in \Om_+$,   $\om(t)$ converges to the geometric boundary as $t$ goes to infinity (\cite{H2}),  where we still denote by $\om$ its projection to $\M$.   Let $\gamma_{\om(0), \om(\infty)}$ be the geodesic ray starting from  $\om(0)$ asymptotic to  $\om(\infty):=\lim_{t\to +\infty}\om(t)$.   Then,  loosely speaking, $\om$ stays close to  $\gamma_{\om(0), \om(\infty)}$ (see Lemma \ref{travel-along-geodesic}).   The Busemann function to be introduced will be very helpful to record the movement of the `shadow' of $\om(t)$ on  $\gamma_{\om(0), \om(\infty)}$.
 
 Let $x\in \M$ and define for $y\in \M$ the \emph{Busemann function}
$b_{x, y}(z)$ on $\M$ by letting
\[
b_{x, y}(z):=d(y, z)-d(y, x), \ {\rm{for}}\  z\in \M.
\]
The assignment of $y\mapsto b_{x, y}$ is continuous, one-to-one and
takes value in a relatively compact set of functions for the
topology of uniform convergence on compact subsets of $\M$. The
Busemann compactification of $\M$ is the closure of $\M$ for that
topology. In the negative curvature case, the Busemann
compactification coincides with the geometric compactification  (see \cite{Ba}).  So
for each ${\bf v}=(x, \xi)\in \M\times \pp\M$, the \emph{Busemann function
at ${\bf v}$}, given by
\[
b_{{\bf v}}(z):=\lim\limits_{y\to \xi}b_{x, y}(z), \ \mbox{for}\
z\in \M,
\]
is well-defined.  For points on the geodesic $\g_{x, \xi}$, its Busemann function value is negative its flow distance with $x$. In other words, for $s, t\geq 0$,
\begin{equation}\label{Buseman-gradient}
b_{{\bf v}}(\g_{x,\xi}(t))-b_{{\bf v}}(\g_{x, \xi}(s))=s-t.
\end{equation}
The equation (\ref{Buseman-gradient}) continues to hold if we replace $\g_{x, \xi}$ with geodesic $\g_{z, \xi}$ starting from $z\in\M$ which is asymptotic to $\xi$ (\cite{EO}). Note that the absolute value of the difference of the Busemann function at two points are always less than their distance. It follows that,  if we consider the Busemann function $b_{\bfv}$ as a function defined on $W^s(x,\xi) ,$
\begin{equation}\label{Buse-geo}
\nabla b_{{\bf v}}(z)=-\overline{X}(z, \xi),
\end{equation}
where $\overline{X}(z, \xi)$ is the tangent vector to $W^s(\bfv )$ which projects to $(z, \xi)=\dot\gamma_{z, \xi}(0)$.  This relationship explains why the Busemann function is involved in the analysis of geometric and dynamical quantities: the variation of $\overline{X}$ is related to variation of asymptotic geodesics, the theory of Jacobi fields; while the vector field  $\overline{X}$ on $S\M$ defines   the geodesic flow. 

We are going to use both interpretations of $\overline{X}$ to see how the linear drift is related the geometry. Since we only discuss $C^3$ metrics in this paper, we will state the results in this setting. But most results have corresponding versions for $C^k$ metrics.

We begin with the theory of Jacobi fields and Jacobi tensors.  Most notations will also be used in Section 4.  Recall the Jacobi fields along a geodesic $\gamma$ are vector fields $t\mapsto J(t)\in T_{\gamma(t)}M$ which describe infinitesimal variation of geodesics around $\gamma$.  It is well-known that $J(t)$ satisfies the Jacobi equation
\begin{equation}\label{Jacobi  field-definition}
\nabla_{\dot{\gamma}(t)}\nabla_{\dot{\gamma}(t)} J(t)+R(J(t), \dot{\gamma}(t))\dot{\gamma}(t)=0
\end{equation}
and is uniquely determined by the values of $J(0)$ and $J'(0)$.  (Here for vector fields $Y,Z$  along $\M$, we denote $\nabla _Y Z$ and $R(Y,Z)$ the {\it {covariant derivative}}  and the {\it {curvature tensor}} associated to the Levi-Civita connection of $\wt g$.)  Let $N(\gamma)$ be the normal bundle of $\gamma$:
\[
N(\gamma):=\cup_{t\in \Bbb R}N_t(\gamma), \ \mbox{where}\ N_t(\gamma)=\{Y\in T_{\gamma(t)}M:\ \langle Y, \dot{\gamma}(t)\rangle=0\}.
\]
It follows from (\ref{Jacobi field-definition}) that if $J(0)$ and $J'(0)$ both belong to $N_0(\g)$, then $J(t)$ and $J'(t)$ both belong to $N_t(\g)$, for all $t \in \R$.
Also,  it is easy to deduce from (\ref{Jacobi  field-definition}) that the Wronskian  of two Jacobi  fields $J$ and  $\widetilde{J}$ along $\g$:
\[
W(J, \widetilde{J}):=\langle J', \widetilde{J}\rangle-\langle J, \widetilde{J}'\rangle
\]
is constant. 

A \emph{$(1, 1)$-tensor along $\gamma$} is a family $V=\{V(t), \ t\in \Bbb
R\}$, where $V(t)$ is an endomorphism of $N_t(\gamma)$ such that for
any family $Y_t$ of parallel vectors along $\gamma$, the covariant
derivative $\nabla_{\dot{\gamma}(t)} (V(t)Y_t)$ exists.  The
curvature tensor $R$ induces a symmetric $(1, 1)$-tensor along
$\gamma$ by $R(t)Y=R(Y, \dot{\gamma}(t))\dot{\gamma}(t)$. A $(1, 1)$-tensor $V(t)$ along $\gamma$ is called a \emph{Jacobi tensor} if it
satisfies
\[
\nabla_{\dot{\gamma}(t)}\nabla_{\dot{\gamma}(t)} V(t)+R(t)V(t)=0.
\]
If $V(t)$ is a Jacobi tensor along $\gamma$, then $V(t)Y_t$ is a Jacobi field for any parallel field $Y_t$.

For each $s>0$, ${\bf v} \in S\M$,  let $S_{\bfv, s}$ be the Jacobi tensor along the geodesic $\gamma _\bfv$ with the boundary conditions $S_{\bfv, s}(0)={\mbox{Id}}$ and $S_{\bfv, s}(s)=0$.  Since $(\M, \wt{g})$ has no conjugate points, the limit $\lim_{s\to +\infty}S_{\bfv, s}=:S_{\bfv}$ exists (\cite{Gre}).
 The tensor
$S_{\bfv}$ is called the \emph{stable tensor} along the geodesic
$\gamma _\bfv$. Similarly, by reversing the time $s$, we obtain the
\emph{unstable tensor} $U_{\bfv}$ along the geodesic $\gamma _\bfv$.

To relate the stable and unstable tensors to  the dynamics of the geodesic flow, we first recall the metric structure of the tangent space $TT\M$ of  $T\M$.    For $x\in \M$
and ${\rm v}\in T_x\M$, an element $w \in T_{\rm v} T\M$   is
{\it {vertical }} if its projection on $T_x\M$ vanishes. The vertical subspace $V_{\rm v} $   is  identified with $T_x\M$.
The connection defines a {\it {horizontal }} complement $H_{\rm v}$,   also identified with $T_x\M.$
This gives a  horizontal/vertical Whitney sum decomposition
\[
TT\M=T\M\oplus T\M.
\]
Define the inner product on $TT\M$ by
\[
\langle (Y_1, Z_1), (Y_2, Z_2)\rangle_{\wt{g}}:=\langle Y_1,
Y_2\rangle_{\wt{g}}+\langle Z_1, Z_2\rangle_{\wt{g}}.
\]
It induces a Riemannian metric on $T\M$, the so-called Sasaki
metric. The  unit tangent bundle $S\M$ of
the universal cover $(\M, \wt{g})$ is a subspace of $T\M$  with
tangent space
\[
T_{(x, {\rm v})}S\M=\{(Y, Z):\ Y, Z\in T_x\M, Z\perp  {\rm v}\},\  \mbox{for}\
x\in \M,  {\rm v}\in S_x\M.
\]

Assume ${\bf v}=(x, {\rm v})  \in S\M$.  Horizontal vectors in $T_{{\bf v}}S\M$ correspond to pairs $(J(0),0)$.  In particular, the geodesic spray $ \overline X_{{\bf v}}$  at $\bf v$ is the horizontal vector associated with $({\rm v},0)$.  A vertical vector in $T_{{\bf v}}S\M$   is a vector tangent to $S_x\M$.  It corresponds to a pair $(0, J'(0))$, with $J'(0) $ orthogonal to  ${\rm v}$. The orthogonal space to $\overline{X}_{{\bf v}}$ in $T_{\bfv} SM$ corresponds to pairs $({\rm v}_1, {\rm v}_2), {\rm v}_i \in N_0(\g_\bfv)$ for $i =1,2.$ 

The dynamical feature of Jacobi fields can be seen using the geodesic flow on the unit tangent bundle. Let ${\bf \Phi}_t$  be the time $t$ map of the geodesic flow on $S\M$,  in coordinates,
\begin{equation*}\label{geodesic flow}
{\bf\Phi}_t(x, \xi)=(\gamma_{x, \xi}(t), \xi), \ \forall  (x, \xi)\in S\M.
\end{equation*}
Let $D{\bf \Phi}_t$ be the tangent map of ${\bf \Phi}_t$.   Then, if $(J(0), J'(0) )$ is the horizontal/vertical decomposition of  ${\bf w}\in T_{(x, \xi)}S\M$,  $(J(t), J'(t) )$ is the horizontal/vertical decomposition of  $D {\bf \Phi}_t{\bf w}\in T_{{\bf \Phi}_t(x, \xi)}S\M$. 

Due to the negative curvature nature of the metric, the geodesic flow on $S\M$ is \emph{Anosov}: the tangent bundle $TS\M$ decomposes into the Witney sum of three $D{\bf \Phi}_t$-invariant   subbundles ${\bf E}^{\rm c}\oplus {\bf E}^{\rm ss}\oplus {\bf E}^{\rm uu}$, where ${\bf E}^{\rm c}$ is the 1-dimensional subbundle tangent to the flow and $ {\bf E}^{\rm ss}$  and $ {\bf E}^{\rm uu}$ are the strongly contracting and expanding subbundles, respectively, so that there are constants $C, c>0$ such that
\begin{itemize}
\item[i)] $\|D{\bf\Phi}_t {\bf w}\|\leq Ce^{-ct}\|{\bf w}\|$ for ${\bf w}\in {\bf E}^{\rm ss}$, $t>0$.\\
\item[ii)] $\|D{\bf \Phi}_t ^{-1}{\bf w}\|\leq Ce^{-ct}\|{\bf w}\|$   for ${\bf w}\in {\bf E}^{\rm uu}$, $t>0$.
\end{itemize}
For any geodesic $\bfv = (x, \xi ) \in S\M$, let $S_{\bfv},
U_{\bfv}$ be the stable and unstable tensors along $\gamma _\bfv$,
respectively. The stable subbundle ${\bf E}^{\rm ss}$ at $\bfv$
is the graph of the mapping $S'_{\bfv}(0)$, considered as a map
from $ \overline{N_0(\gamma _\bfv )}$ to $V_\bfv $
%$\overline{T_x\M}$, 
sending $Y$ to
$S'_{\bfv}(0)Y$, where  $\overline{N_0(\gamma _\bfv )} := \{ {\bf {w}}, {\bf {w}} \in H_\bfv, {\bf {w}} \perp \overline X_\bfv \}$. %\footnote{We also use $H_{v}$ for the horizontal lift of $v$ to $TO(N)$. So I replace the later by ${\bf H}$. See page 25 and 26.}
% is the lift of
%$N_0(\gamma)$ into the horizontal space of $T_{(x, \xi)}S\M$ and
%$\overline{T_x\M}$ is the lift of $T_x\M$ to the vertical part of
%$T_{(x, \xi)}S\M$. 
Similarly, the unstable subbundle ${\bf E}^{\rm uu}$
at $\bfv$ is the graph of the mapping $U'_{\bfv}(0)$
considered as  a map from $ \overline{N_0(\gamma _\bfv )}$ to $V_\bfv .$
%$\overline{T_x\M}$. 
Due to the Anosov property of the geodesic flow, the distributions of ${\bf E}^{\rm ss}, {\bf E}^{\rm uu}$ (and hence ${\bf E}^{\rm c}\oplus {\bf E}^{\rm ss}, {\bf E}^{\rm c}\oplus {\bf E}^{\rm uu}$) are H\"{o}lder continuous (this is first proved by Anosov  (\cite{Ano}),   see \cite[Proposition 4.4]{Ba} for a similar but simpler argument by Brin). As a consequence, the $(1,1)$-tensors  $S_{
\bfv}, S'_{\bfv}, U_{\bfv}, U'_{\bfv}$ are also H\"{o}lder
continuous with respect to $\bfv$. 

    We are in a situation to see the relation between the Busemann function and the geodesic flow. Let $x_0\in \M$ be a reference point  and for any $\xi\in \partial\M$ consider $b_{x_0, \xi}(\cdot):=\lim_{z\rightarrow \xi}b_{x_0, z}(\cdot).$   For any ${\bf v}=(x, \xi)\in \M\times \partial \M$,  the set
\[\left\{(y, \xi):\  b_{x_0, \xi}(y)=b_{x_0, \xi }(x) \right\}\]
 turns out to coincide with the   \emph{strong stable  manifold at ${\bf v}$}, denoted $W^{ss}({\bf v})$,  which is 
\begin{equation*}%\label{stable manifold}
W^{ss}({\bf v}):=\left\{(y, \eta):\ \limsup\limits_{t\to +\infty}\frac{1}{t}\log {\mbox{dist}}\left(\Phi_t(y, \eta),  \Phi_t(\bf v)\right)<0\right\}.
\end{equation*}
(The  \emph{strong unstable manifold  at ${\bf v}$}, denoted $W^{su}({\bf v})$,  is defined  by reversing the time.)
In other words, the collection of the foot points $y$ such
that $(y, \xi)\in W^{ss}(x, \xi)$ form the \emph{stable horosphere},
which is a level set of Busemann function.  Note that $W^{ss} (\bf v) $
 locally is a $C^{2}$ graph from ${\bf E}_{\bfv}^{\rm
ss}$ to ${\bf E}_{\bfv}^{\rm c}\oplus{\bf E}_{\bfv}^{\rm
uu}$ and is tangent to ${\bf E}^{\rm ss}_{\bfv}$. So, by the Jacobian characterization of ${\bf E}^{\rm ss}_{\bfv}$ of the previous paragraph and (\ref{Buse-geo}), it is
true (\cite{Esc, HIH}) that
\[
\nabla_{{\bf w}}(\nabla b_{x, \xi})(x)=-S'_{(x,\xi)}(0)({ \bf w}),\ 
 \forall { \bf w}\in T_{x}\M.
\]
Thus, 
\begin{equation*}\label{Div-trace}\Delta_{x}b_{x,\xi}=-{\rm Div}\overline{X}=-{\mbox{Trace of}}\ S'_{(x,\xi) }(0), \end{equation*} which is the mean curvature of the horosphere  $W^{ss}(x, \xi)$  at $x$. 
Note that for  each $\psi\in G$,
\[
b_{x_0, \psi\xi}(\psi x)=b_{x_0 , \xi}(x)+b_{\psi ^{-1}x_0, \xi}( x_0).
\]
Hence  $\Delta_x b_{x_0, \xi}$ satisfies $\Delta_{\psi
x}b_{x_0, \psi\xi}=\Delta_x b_{x_0, \xi}$   and defines a function $B$ on the
unit tangent bundle $SM$, which is called the \emph{Laplacian of the
Busemann function}.  Due to the hyperbolic nature of the geodesic
flow,  the function $B$ is a  H\"{o}lder continuous function on $SM$.

Now,  we can state the integral formula for the linear drift. 
\begin{prop}\label{formulas-l-h-Y-l}Let $\LL=\D+Y$ be  such
that $Y^*$, the dual of $Y$ in the cotangent bundle  to the stable foliation over  $SM$,
satisfies $dY^*=0$ leafwisely  and $\mbox{pr}(-\langle \overline{X}, Y\rangle)>0$. Then we have
\begin{eqnarray}
\ell_{\LL}&=&-\int_{\scriptscriptstyle{M_0\times \pp\M}}
\left({\rm{Div}}\overline{X}+ \langle Y, \overline{X}\rangle\right)\
  d\wt{{\bf m}}.\label{ell-formula}
\end{eqnarray}
\end{prop}
(Observe that the classical formula (\ref{formulas-ell-h})  for the linear drift is obtained from Proposition \ref{formulas-l-h-Y-l} by considering the metric $g^\l$ and $Y\equiv 0$.)

\begin{proof} For $\P$-a.e. path
${\om}\in {\Om}_{+}$, we still denote $\om$ its projection to  $\M$ and let  ${\bf v}:=\om(0)$ and  $\eta:=\lim_{t\to +\infty}{\om}(t)\in \pp\M$.  We see that when $t$
goes to infinity, the process $b_{{\bf v}}({\om}(t))-d(x,
{\om}(t))$ converges $\P$-a.e. to the a.e. finite
number $-2(\xi|\eta)_x$, where
\begin{equation}\label{Gromov-product-def-sec3}
(\xi|\eta)_x:=\lim\limits_{y\to \xi, z\to \eta}(y|z)_x\ \mbox{and}\  (y|z)_x:=\frac{1}{2}\left(d(x,
y)+d(x, z)-d(y, z)\right).
\end{equation}
So for $\P$-a.e. $\om\in \Om_+$, we have
\[
\lim\limits_{t\to +\infty}\frac{1}{t}b_{{\bf
v}}({\om}(t))=\ell_{\LL}.
\]
Using the fact that the $\LL$-diffusion has leafwise infinitesimal
generator $\Delta+Y$ and is ergodic with invariant measure ${\bf m}$ on $SM$, we obtain
\begin{eqnarray*}
  \ell_{\LL}&=&\lim_{t\to +\infty}\frac{1}{t}\int_{0}^{t}\frac{\pp}{\pp s} b_{{\bf v}}({\om}(s))\ ds\\
  &=& \lim_{t\to +\infty}\frac{1}{t}\int_{0}^{t}(\Delta+Y)b_{{\bf v}}({\om}(s))\ ds \, \left(= \, \int _{\scriptscriptstyle{M_0\times \pp\M}}(\Delta+Y)b_{{\bf v}}\, \  d\wt{{\bf m}}\right)\\
  &=& -\int_{\scriptscriptstyle{M_0\times \pp\M}} \left({\rm{Div}}\overline{X}+ \langle Y, \overline{X}\rangle\right)\
  d\wt{{\bf m}}.
\end{eqnarray*}
\end{proof}

\bigskip

The negative of the logarithm of the Green function  has a lot of properties analogous to a distance function.  First of all, let us recall some  classical results concerning Green
functions  from \cite{An}.

\begin{lem}\label{GM-classical Harnack}(see \cite[Remark 3.1]{An}) Let $\LL=\D+Y$ be  such
that $Y^*$, the dual of $Y$ in the cotangent bundle  to the stable foliation over  $SM$,
satisfies $dY^*=0$ leafwisely and $\mbox{pr}(-\langle \overline{X}, Y\rangle)>0$  and let
${\bf G}(\cdot, \cdot)=\{{G}_{{\bf v}}(\cdot, \cdot)\}_{{\bf v}\in
S\M}$ be the Green function of $\LL$.  There exists a constant
$c_0\in (0, 1)$ such that for any ${\bf v}\in S\M$ and any $x, y,
z\in
  \M$ with mutual distances greater than
$1$,
\begin{equation}\label{Green-geq}
G_{{\bf v}}(x, z)\geq c_0G_{{\bf v}}(x, y) G_{{\bf v}}(y, z).
\end{equation}
\end{lem}

%\footnote{As we agreed, $v$ is used for elements of $SM$ and $\rm v$ for elements of $T_y\M$ or $S_y\M$ and $\bfv\in S\M$...}
For  ${\rm v, w}\in S_x \M$, $x\in \M$,  the angle $\angle_x ({\rm v, w})$ is
the unique number $0\leq \theta\leq \pi$ such that $\langle v,
w\rangle=\cos  \theta$.  Given  ${\rm v}\in S_x \M$ and  $0<\theta<\pi$,
the  set
\[\Gamma_x({\rm v}, \theta):=\{y\in \M\cup\pp\M:\ \ \angle_x ({\rm v}, \dot\g_{x, y}(0))<\theta\}\]
 is called  the \emph{cone of vertex $x$, axis ${\rm v}$, and angle $\theta$},  where $\g_{x, y}$ is  the geodesic segment that starts at $x$ and ends at $y$. For any $s>0$, the cone $\Gamma$ with vertex $\gamma_{{\rm v}}(s)$ (where $\gamma_{{\rm v}}$ is the geodesic starting at $x$ with  initial speed ${\rm v}$),
 axis $\dot\g_{{\rm v}}(s)$ and angle $\theta$ is called the
 \emph{$s$-shifted cone} of $\Gamma_x({\rm v}, \theta)$. The following is a special case of the Ancona's inequality  at infinity (\cite{An}).

\begin{lem}\label{GM-Ancona inequality}(see \cite[Theorem 1']{An}) Let  $\LL$ and ${\bf G}$ be as in Lemma \ref{GM-classical Harnack}.  Let $\Gamma:=\Gamma_{x_0}({\rm v},
\frac{\pi}{2})$ be a cone in $\M$ with vertex $x_0$, axis ${\rm v}$ and
angle $\frac{\pi}{2}$.  Let $\Gamma_{1}$ be the $1$-shifted cone of
$\Gamma$ and $x_1$ be the vertex of $\Gamma_{1}$.  There exists a
constant $c_1$  such that for  any ${\bf v}\in S\M$, any $\Gamma$,
all $x\in \M\backslash \Gamma$ and $z\in \Gamma_{1}$,
\begin{eqnarray}
 G_{{\bf v}}(x, z)\leq c_1 G_{{\bf
v}}(x, x_1) G_{{\bf v}}(x_0, z).\label{Har-infinity}
\end{eqnarray}
\end{lem}

We may assume  $c_1=c_0^{-1}$, where $c_0$ is as in Lemma
\ref{GM-classical Harnack}.   As a  consequence of Lemma
\ref{GM-classical Harnack} and Lemma \ref{GM-Ancona inequality},
${\bf G}$ is related to the distance $d$ in the following way.

\begin{lem}\label{Green-d-relation} Let  $\LL$ and ${\bf G}$ be as in Lemma \ref{GM-classical Harnack}. There exist positive numbers $c_2, c_3, \alpha_2, \alpha_3$ such that  for any ${\bf v}\in S\M$ and  any $x, z\in \M$ with $d(x, z)\geq 1$,
\begin{equation}\label{equ-Green-d-relation}
c_2 e^{-\alpha_2 d(x, z)}\leq G_{{\bf v}}(x, z)\leq c_3 e^{-\alpha_3
d(x, z)}.
\end{equation}
\end{lem}
\begin{proof} The upper bound of (\ref{equ-Green-d-relation}) was shown in \cite[Corollary 4.8]{H2} using Ancona's inequality at infinity (cf. Lemma \ref{GM-Ancona inequality}).  For the  lower bound,  we first observe that Lemma \ref{GM-classical Harnack} also holds true if $x, y, z$ satisfies $d(x, z)>1$ and $d(x, y)=1$.  Indeed,  by  the classical
Harnack  inequality (\cite{LY}), there exists $c_4\in (0, 1)$ such
that for any ${\bf v}\in S\M$ and $x, y, z\in \M$ with $d(x, z)>1$
and $d(x, y)\leq1$,
\begin{equation}\label{LY-harnack}
c_4 G_{{\bf v}}(y, z)\leq G_{{\bf v}}(x, z)\leq c_4^{-1} G_{{\bf
v}}(y, z).
\end{equation}
Since $d(x, y)=1$,  by \cite[Proposition 7]{An}, there is  $c_5\in
(0, 1)$ (independent of $x, y$)
with\begin{equation}\label{G-x-y-distance 1} c_5\leq G_{{\bf v}}(x,
y)\leq c_5^{-1}.
\end{equation}
So,  if $c_0\leq c_4c_5$, then (\ref{Green-geq}) holds  true for $x,
y, z\in \M$ with $d(x, z)>1$ and $d(x, y)=1$. Now, for $x, z\in \M$
with $d(x, z)>1$, choose a sequence of points $x_i, 1\leq i\leq n,$
on the geodesic segment $\gamma_{x, z}$ with $x_0=x, x_n=z,$ $d(x_i,
x_{i+1})=1, i=0, \cdots, n-2,$ and $d(x_{n-1}, z)\in [1, 2)$.
Applying (\ref{Green-geq}) successively for $x_i, x_{i+1}, z$, we
obtain
\[
G_{{\bf v}}(x, z)\geq G_{{\bf v}}(x_{n-1}, z)(c_0c_5)^{n-1}\geq
c_4c_5(c_0c_5)^{n-1}\geq c_4c_5 (c_0c_5)^{d(x, y)},
\]
where, to derive the second inequality, we use (\ref{LY-harnack})
and the fact that the lower bound of (\ref{G-x-y-distance 1}) holds
for any $x, y\in \M$ with $d(x, y)\leq 1$.  The lower bound
estimation of (\ref{equ-Green-d-relation}) follows for $c_2=c_4c_5$
and $\alpha_2=-\ln c_0c_5$.
\end{proof}

We may assume the constants $c_2, c_3$ in Lemma
\ref{Green-d-relation} are such that $c_2$ is smaller than $1$ and
$c_3=c_2^{-1}$. For each ${\bf v}\in S\M$, $x, z\in \M$,  let
\[d_{{G}_{{\bf v}}}(x, z):=\left\{
  \begin{array}{ll}
   -\ln \left(c_2 G_{{\bf v}}(x, z)\right), & \hbox{if}\ d(x, z)>1; \\
   -\ln c_2, & \hbox{otherwise.}
  \end{array}
\right.
\]
Although  $d_{{G}_{{\bf v}}}$  is  always greater than the positive
number $\min\{\alpha_3, -\ln c_2\}$ by (\ref{equ-Green-d-relation}),  we still call it a  \emph{`Green
metric'}  for $\LL_{{\bf v}}$ (after \cite{BHM1} for the hyperbolic
groups case) since  it  satisfies an almost triangle inequality in
the following sense.
\begin{lem}\label{lem-triangle-d-G}
  There exists a constant $c_6\in (0, 1)$ such that for all $x, y, z\in \M$,
  \begin{equation}\label{almost-triangle-inequ}
 d_{{G}_{{\bf v}}}(x, z)\leq d_{{G}_{{\bf v}}}(x, y)+d_{{G}_{{\bf v}}}(y, z)-\ln c_6.
  \end{equation}
\end{lem}
\begin{proof} If $d(x, z)\leq 1$,  then (\ref{almost-triangle-inequ}) holds  for $c_6=c_2$.  If  $x, y, z$ have mutual distances greater than $1$,   then (\ref{almost-triangle-inequ}) holds for $c_6=c_0$ by Lemma \ref{GM-classical Harnack}.  If $d(x, z)>1$ and $d(y, z)\leq 1$, using the classical Harnack inequality (\ref{LY-harnack}),  we have
\[
{{G}_{{\bf v}}}(x, z)\geq c_4 {{G}_{{\bf v}}}(x, y)
\]
and hence (\ref{almost-triangle-inequ}) holds with $c_6=c_4$ if, furthermore,
$d(x, y)>1$ or with $c_6=c_4c_5$ otherwise. The case  that $d(x,
z)>1$, $d(x, y)\leq 1$ can be treated similarly.
\end{proof}

By Lemma \ref{Green-d-relation},  $d_{{G}_{{\bf v}}}$ is  comparable
to  the  metric $d$ for  any $x, z\in \M$ with  $d(x, z)>1$:
\begin{equation}\label{d-G-d-greater-1}
\alpha_3 d(x, z)\leq d_{{G}_{{\bf v}}}(x, z)\leq \alpha_2 d(x,
z)-2\ln c_2.
\end{equation}
Using Lemma \ref{GM-Ancona inequality},  we can further obtain that
$d_{{G}_{{\bf v}}}$  is almost additive along the geodesics.

\begin{lem}\label{triangle-along geodesic} Let  $\LL$ and ${\bf G}$ be as in  Lemma \ref{GM-classical Harnack}. There exists a constant $c_7$ such that for any ${\bf v}\in S\M$, any $x, z\in
\M$ and $y$ in the geodesic segment $\g_{x, z}$ connecting $x$ and
$z$,
\begin{equation}\label{almost-add-along-geodesic}
\left|d_{{G}_{{\bf v}}}(x, y)+d_{{G}_{{\bf v}}}(y, z)- d_{{G}_{{\bf
v}}}(x, z)\right|\leq -\ln c_7.
\end{equation}
\end{lem}
\begin{proof}Let $x, z\in
\M$ and $y$ belong to the geodesic segment $\g_{x, z}$. If  $d(x,
y), d(y,z)\leq 1$, then $d(x, z)\leq 2$ and, using
(\ref{d-G-d-greater-1}), we obtain (\ref{almost-add-along-geodesic})
with $c_7=c_2^2 e^{-2\alpha_2}$. If $d(x, y)\leq 1$ and $d(y,z)>1$
(or $d(y,z)\leq 1$ and $d(x, y)>1$), using Harnack's inequality
(\ref{LY-harnack}), we have (\ref{almost-add-along-geodesic}) with
$c_7=c_2c_4$. Finally, if $x, y, z$ have mutual distances greater than $1$, we have by Lemma  \ref{GM-classical
Harnack} and Lemma \ref{GM-Ancona inequality} (where we can use
Harnack's inequality to replace $G_{{\bf v}}(x, x_1)$ in
(\ref{Har-infinity}) by $c_4^{-1}G_{{\bf v}}(x, x_0))$ that \[
\left|\ln G_{{\bf v}}(x, y)+\ln G_{{\bf v}}(y, z)-\ln G_{{\bf v}}(x,
z) \right|\leq -\ln (c_1c_4)
\]
and consequently,
\[
 \left|d_{{G}_{{\bf v}}}(x, y)+d_{{G}_{{\bf v}}}(y, z)- d_{{G}_{{\bf
v}}}(x, z)\right|\leq -\ln (c_1c_2 c_4).
\]
\end{proof}

More is true, as we can see  from Lemma \ref{GM-Ancona inequality}
and Lemma \ref{lem-triangle-d-G} as well.

\begin{lem}\label{triangle-separated by cones} Let  $\LL$ and ${\bf G}$ be as in Lemma \ref{GM-classical Harnack}.
There exists a constant $c_8$ such that for any ${\bf v}\in S\M$, if
$x, y, z\in \M$ are such that $x$ and $z$ are separated by some cone
$\Gamma$ with vertex $y$ and angle $\frac{\pi}{2}$, and $\Gamma_1$,
the 1-shifted cone of $\Gamma$, i.e., $x\in \M\backslash \Gamma$,
$z\in \Gamma_1$, then
\[
\left|d_{{G}_{{\bf v}}}(x, y)+d_{{G}_{{\bf v}}}(y, z)- d_{{G}_{{\bf
v}}}(x, z)\right|\leq -\ln c_8.
\]
\end{lem}

\bigskip

The counterpart of  the Busemann function for the analysis of the pathwise limits for stochastic entropy is  the Poisson kernel function.  Let ${\bf v}=(x, \xi)\in \M\times \partial\M$. A \emph{Poisson kernel
function} $k_{{\bf v}}(\cdot, \eta)$ of $\LL_{{\bf v}}$ at $\eta\in
\partial\M$ is a positive $\LL_{{\bf v}}$-harmonic function on $\M$
such that
\[
k_{{\bf v}}(x, \eta)=1, k_{{\bf v}}(y, \eta)=O(G_{{\bf v}}(x, y)), \
\mbox{as}\ y\to \eta'\not=\eta.
\]
A point $\eta\in \partial\M$ is a \emph{Martin point of $\LL_{{\bf v}}$}  if
it satisfies the following properties:
\begin{itemize}
\item[i)] there exists a Poisson kernel function $k_{{\bf
v}}(\cdot, \eta)$ of $\LL_{{\bf v}}$ at $\eta$,
\item[ii)]the Poisson kernel function is unique, and
\item[iii)] if $y_n\to \eta$, then $\ln G_{{\bf v}}(\cdot, y_n)-\ln G_{{\bf v}}(x, y_n)\to \ln k_{{\bf
v}}(\cdot, \eta)$ uniformly on compact sets.
\end{itemize}
Since  $(M, g)$ is negatively curved and  $\LL_{{\bf v}}$ is weakly
coercive, every point $\eta$ of the geometric boundary $\pp\M$ is a
Martin point by Ancona  (\cite{An}). Hence $k_{{\bf v}}(\cdot, \eta)$ is also called the Martin kernel function at $\eta$. 

The function  $k_{\bfv} (\cdot, \eta) $ should be understood as  a function on $W^s (\bfv)$ for all $\eta$, i.e.  it is identified with 
 $k_{\bfv} (\p(\cdot), \eta) $,  where $\p: S\M\to \M$ is the projection map.  In case $\LL=\Delta$,   all the $k_{\bfv} (\cdot, \eta)$'s are the same as $k_{\eta}(\cdot)$, the Martin kernel function on $\M$ associated to $\Delta$. In general,  $k_{\bfv}$ may vary from leaf to leaf.

Finally,  we can state the integral formula for the stochastic entropy.

\begin{prop}\label{formulas-l-h-Y-h}Let $\LL=\D+Y$ be  such
that $Y^*$, the dual of $Y$ in the cotangent bundle  to the stable foliation over  $SM$,
satisfies $dY^*=0$ leafwisely  and $\mbox{pr}(-\langle \overline{X}, Y\rangle)>0$. Then we have
\begin{eqnarray}
h_{\LL}&=&\int_{\scriptscriptstyle{M_0\times \pp\M}}\|\nabla \ln
k_{{\bf v}}(x, \xi)\|^2\ d\wt{{\bf m}}.\label{h-formula}
\end{eqnarray}
\end{prop}
(Since each $k_{\bfv} (\cdot, \eta) $ is a function on $W^s(\bfv)$, in particular, when $\eta=\xi$,   its gradient (for the lifted metric from $\M$ to $W^s (\bfv)$)  is a tangent vector to $W^s (\bfv)$.  We also observe  that the classical formula (\ref{formulas-ell-h})  for the stochastic entropy is obtained from Proposition \ref{formulas-l-h-Y-h} by considering the metric $g^\l$ and $Y\equiv 0$.)

\begin{proof}
For $\P$-a.e. path
${\om}\in {\Om}_{+}$, we still denote $\om$ its projection to  $\M$ and write  ${\bf v}:=\om(0)$.    When $t$
goes to infinity, we see that \[
\limsup\limits_{t\to +\infty}|\ln G_{{\bf v}}(x, {\om}(t))-\ln
k_{{\bf v}}({\om}(t), \xi)|<+\infty,
\] %\footnote{Do we need to number these relations? There is a line at the beginning of latex called $\backslash usepackage\{refcheck\}$.  Run the tex file twice with this order, you will be able to see those equations not cited (the question marks will appear).  I use it to find that the equation before `Again' is not cited. }  
  Indeed, let $z_t$ be the point on the geodesic ray $\g_{{\om}(t), \xi}$
closest to $x$. Then, as $t \to +\infty $,
\begin{equation}\label{Green-Matin-kernel-function}
 G_{\bf v} (x, {\om}(t)) \asymp G_{\bf v} (z_t, {\om}(t)) \asymp \frac{G_{\bf v} (y, {\om}(t)) }{G_{\bf v} (y, z_t)}
\end{equation}
for all $y$ on the geodesic going from ${\om} (t)$ to $\xi$,
where $\asymp$ means up to some multiplicative constant independent of $t$. The
first $\asymp$  comes from Harnack inequality using the fact
that $\sup_t d(x, z_t)$ is finite $\P$-almost everywhere. (For $\P$-a.e. $\om\in \Om_+$, $\eta=\lim_{t\rightarrow +\infty}\om(t)$ differs from $\xi$ and  $d(x, z_t)$, as $t\rightarrow +\infty$,  tends to  the distance between $x$ and the geodesic asymptotic to $\xi$ and $\eta$ in opposite  directions.)
The second $\asymp$  comes from Ancona's inequality (\cite{An}).
Replace $G_{\bf v} (y, {\om}(t))/G_{\bf v} (y, z_t)$ by its limit
as $y \to \xi $, which is $k_{(z_t, \xi ) } ( {\om}(t), \xi ) $,
which is itself $\asymp k_{\bf v} ( {\om}(t), \xi )$ by
Harnack inequality again.
 By (\ref{h-in g}), it follows that,    for $\P$-a.e.  $\om\in
\Om_+$,
\begin{equation*}%\label{h-by k}
\lim\limits_{t\to +\infty}-\frac{1}{t}\ln k_{{\bf v}}({\om}(t),
\xi)=h_{\LL}.
\end{equation*}
  Again, using the fact that the $\LL$-diffusion
has leafwise infinitesimal generator $\Delta+Y$ and is ergodic, we obtain
\begin{eqnarray*}
  h_{\LL}&=&\lim_{t\to+\infty}-\frac{1}{t}\int_{0}^{t}\frac{\pp}{\pp s}(\ln k_{{\bf
v}}({\om}(s), \xi))\ ds\\
  &=&\lim_{t\to+\infty}\frac{1}{t}\int_{0}^{t}-(\D+Y)\left(\ln k_{{\bf
v}}({\om}(s), \xi)\right)\ ds \left(= - \int_{\scriptscriptstyle{M_0\times \pp\M}}(\D+Y)\left(\ln k_{{\bf
v}}\right)d\wt{{\bf m}}\right)\\
  &=&  \int_{\scriptscriptstyle{M_0\times \pp\M}}\|\nabla \ln k_{{\bf
v}}(\cdot, \xi)\|^2\  d\wt{{\bf m}}.
\end{eqnarray*}%\footnote{If we want to explain the notation $\nabla \ln k = \nabla \ln k_{x, \xi 
%}(\cdot, \xi)$, this is now!}
The last equality comes from the fact that  the Martin
kernel function $k_{{\bf v}}(\cdot, \xi)$ satisfies $(\D +Y)  (k_{{\bf
v}}(\cdot, \xi))=0$.
\end{proof}

\subsection{A Central limit theorem for the linear drift and the stochastic entropy}\label{Subsection-CLT}
With the help of the Busemann function and the  Martin kernel function, we can further describe the distributions of the pathwise limits for time large.  In this subsection, we recall the Central Limit Theorems for
$\ell_{\LL}$ and $h_{\LL}$ and the ingredients of the proof that we will use later.
%Finally, we have the following Central Limit Theorem for
%$\ell_{\LL}$ and $h_{\LL}$
\begin{prop}(\cite{H2})\label{CLT} Let $\LL=\D+Y$ be  such
that $Y^*$, the dual of $Y$ in the cotangent bundle  to the stable foliation over  $SM$,
satisfies $dY^*=0$ leafwisely and $\mbox{pr}(-\langle \overline{X}, Y\rangle)>0$. Then there
are positive numbers $\sigma_0$ and $\sigma_1$ such that the
distributions of the variables
\[
\frac{1}{\sigma_0 \sqrt{t}}\left[ d_{\W}(\om(0),
\om(t))-t\ell_{\LL}\right]\ \ \mbox{and}\ \ \frac{1}{\sigma_1\sqrt{t}}\left[\ln \bfg(\om(0),
\om(t))+th_{\LL}\right]
\]
are asymptotically close to the  normal distribution when
$t$ goes to infinity.
\end{prop}

The proof of the proposition relies on the contraction property of
the  action of the diffusion process on a certain space of
H\"{o}lder continuous functions. Let $Q_t$ $(t\geq 0)$ be the action
of $[0, +\infty)$ on continuous functions $f$ on $SM$ which
describes the $\LL$-diffusion, i.e.,  \[ Q_t(f)(x,
\xi)=\int_{\scriptscriptstyle{M_0\times \pp\M}} \wt{f}(y, \eta)\bfp(t,
(x, \xi), d(y, \eta)),
\]
where $\wt{f}$ denotes  the $G$-invariant extension of $f$ to $S\M$.  For $\iota>0$, define a norm $\|\cdot\|_{\iota}$ on the space of
continuous functions $f$ on $SM$ by letting
\[
\|f\|_{\iota}=\sup\limits_{x, \xi}|\wt{f}(x, \xi)|+\sup\limits_{x, \xi_1,
\xi_2} |\wt{f}(x, \xi_1)-\wt{f}(x, \xi_2)|\exp(\iota(\xi_1|\xi_2)_x),
\]
where $(\xi_1|\xi_2)_x$ is defined as in (\ref{Gromov-product-def-sec3}),
and let $\mathcal{H}_{\iota}$ be the Banach space of continuous
functions $f$ on $SM$ with $\|f\|_{\iota}<+\infty$. It was shown
(\cite[Theorem 5.13]{H2}) that for sufficiently small $\iota>0$, as $t \to \infty $,
$Q_t$ converges to the mapping $f\mapsto \int  f\ d{\bf m}$
exponentially in $t$ for $f\in \mathcal{H}_{\iota}$. As a consequence, one
concludes that for any $f\in \mathcal{H}_{\iota}$ with $\int f\
d{\bf m}=0$, $u=-\int_{0}^{+\infty} Q_t f\ dt$,  is, up to an
additive constant function, the unique element in
$\mathcal{H}_{\iota}$ which solves $\LL u=f$ (\cite[Corollary
5.14]{H2}). Applying this property to $b_{{\bf v}}$  and $k_{{\bf
v}}(\cdot, \xi)$,    where we  observe that both ${\bf v} \mapsto \Delta
b_{{\bf v}}$ and  $\bfv\mapsto \nabla \ln k_{{\bf v}}(\cdot, \xi)$ are $G$-invariant  and descend to
 H\"{o}lder continuous functions on $SM$ (see \cite{A,HPS} and \cite{H1},
respectively), we obtain  two H\"{o}lder continuous functions $u_0,
u_1$ on $SM$ (or on $M_0\times \partial\M$)  such that %\footnote{Should we write $u_i\circ \pi_{SM}$  or $u_i$? I think it should be the later. Since now $u_0$ is a function on $M_0\times \partial \M$.}
%\footnote{Now, I get rid of $\p$ in formulas of $k_{\bfv}$ because of the remarks surrounding Proposition 2.16. }
\begin{eqnarray*}
 \LL (u_0)&=& -\left({\rm{Div}}(\overline{X})+\langle
Y, \overline{X}\rangle\right)+\int_{\scriptscriptstyle{M_0\times
\pp\M}} \left({\rm{Div}}(\overline{X})+\langle
Y, \overline{X}\rangle\right)\ d\wt{{\bf m}}\\
&=& -\left({\rm{Div}}(\overline{X})+\langle Y,
\overline{X}\rangle\right)-\ell_{\LL}, \ \ \rm{by}\ \
(\ref{ell-formula}), \ \rm{and}\\
 \LL (u_1)&=&\|\nabla \ln k_{{\bf
v}}(\cdot, \xi)\|^2-\int_{\scriptscriptstyle M_0\times
\pp\M}\|\nabla \ln k_{{\bf v}}(\cdot, \xi)\|^2\ d\wt{{\bf m}}\\
&=&\|\nabla \ln k_{{\bf v}}(\cdot, \xi)\|^2- h_{\LL}, \ \rm{by}\
 \ (\ref{h-formula}),
\end{eqnarray*}
where we continue to denote $u_0$ and $u_1$ their $G$-invariant extensions to $S\M$.
For  each $\om\in\Om_+$ belonging to a stable leaf of $S\M$,  we also denote   $\om$ its projection to $\M$.  Then for $f=-b_{{\bf v}}+u_0$  (or $\ln k_{{\bf v}}(\cdot, \xi)+u_1$),  $
f({\om}(t))-f({\om}(0))-\int_{0}^{t} (\LL f)({\om}(s))\ ds$
is a martingale with increasing process $2\|\nabla
f\|^2({\om}(t))\ dt$. In other words, we have the following.

\begin{prop}(cf. \cite[Corollary 3]{L5})\label{M-0-M-1}
For any ${\bf v}=(x, \xi)$, the process $({\bf Z}_{t}^{0})_{t\in
\Bbb R_+}$ with $\om(0)={\bf v}$ [respectively, $({\bf
Z}_{t}^{1})_{t\in \Bbb R_+}$ with $\om(0)={\bf v}$],
 \begin{equation*}
 {\bf Z}_{t}^{0}:=-b_{\om(0)}({\om}(t))+t\ell_{\LL}+
u_0 (\om(t))-u_0 (\om(0))
  \end{equation*}
  \big[respectively,
\begin{equation*}
{\bf Z}_{t}^{1}:=\ln k_{{\bf v}}({\om}(t),
\xi)+th_{\LL}+u_1 (\om(t))-u_1 (\om(0))\big]
\end{equation*}
is a martingale with increasing process
\[
2\|\overline{X}+\nabla u_0\|^2({\om}(t))\ dt\ \
[{\rm{respectively}}, \ 2\|\nabla \ln k_{{\bf v}}(\cdot, \xi)+\nabla
u_1\|^2({\om}(t))\ dt].
\]
\end{prop}

The last ingredient in the proof of  Proposition \ref{CLT} is  a Central Limit Theorem for martingales.

\begin{lem}(\cite[Chapter IV, Theorem 1.3]{RY})\label{CLT_martingale} Let $\mathsf{M}=(\mathsf{M}_t)_{t\geq 0}$ be a continuous, square-integrable centered martingale with respect to an increasing filtration $(\mathfrak{F}_t)_{t\geq 0}$ of a probability space, with stationary increments.  Assume that $\mathsf{M}_0=0$ and
\begin{equation}\label{check-CLT for martingale}
\lim\limits_{t\to +\infty}\E\left(\left|\frac{1}{t}\langle \mathsf{M}, \mathsf{M}\rangle_t-\sigma^2\right|\right)=0
\end{equation}
for some  real  number $\sigma^2$, where $\langle \mathsf{M}, \mathsf{M}\rangle_t$ denotes the quadratic variation of $\mathsf{M}_t$.  Then  the laws of $\mathsf{M}_t/\sqrt{t}$ converge in distribution to a centered normal law with variance $\sigma^2$.
\end{lem}

Now we see that both  ${\bf Z}_{t}^{0}$ and ${\bf Z}_{t}^{1}$ are continuous and square integrable.  The respective average variances converge  to, respectively,   $\sigma_0^2$ and   $\sigma_1^2$, where
\begin{eqnarray*}
  \sigma_0^2&=&2\int_{\scriptscriptstyle{M_0\times \pp\M}} \|\overline{X}+\nabla u_0\|^2\ d\wt{{\bf m}}, \\
  \sigma_1^2&=& 2\int_{\scriptscriptstyle{M_0\times \pp\M}} \|\nabla \ln k_{{\bf v}}(\cdot, \xi)+\nabla
u_1\|^2\ d\wt{{\bf m}}.
\end{eqnarray*} 
By Proposition \ref{mixing}, the $\LL$-diffusion system is  mixing,  (\ref{check-CLT for martingale}) holds for ${\bf Z}_{t}^{0}$ and ${\bf Z}_{t}^{1}$ with $\sigma=\sigma_0$ or $\sigma_1$, respectively.  Hence both  $(1/(\sigma_0 \sqrt{t})){\bf
Z}_{t}^{0}$ and  $(1/(\sigma_1
\sqrt{t})){\bf Z}_{t}^{1}$ will converge to the normal distribution as $t$ tends to infinity.  Note  that in the proof of Propositions \ref{formulas-l-h-Y-l} and  \ref{formulas-l-h-Y-h} we
have shown that for $\P$-a.e. $\om\in {\Om}_+$,  % \footnote{I replace $x$ by $\om(0)$ and use $d_{\W}$ to be keep the consistency. }
$b_{{\bf
v}}({\om}(t))-d_{\W}(\om(0), {\om}(t))$  converges  to a finite number
and that  \[ \limsup\limits_{t\to +\infty}|\ln G_{{\bf v}}(\om(0),
{\om}(t))-\ln k_{{\bf v}}({\om}(t), \xi)|<+\infty.
\]
As a consequence,  we see from  Proposition \ref{M-0-M-1}  that
$(1/(\sigma_0 \sqrt{t}))\left[ d_{\W}(\om(0),
\om(t))-t\ell_{\LL}\right]$ and  $(1/(\sigma_0 \sqrt{t})){\bf
Z}_{t}^{0}$ (respectively,  $(1/(\sigma_1 \sqrt{t}))\left[\ln
\bfg(\om(0), \om(t))+th_{\LL}\right]$ and  $(1/(\sigma_1
\sqrt{t})){\bf Z}_{t}^{1}$) have the same asymptotical distribution,
which is normal,  when $t$ goes to infinity.

\subsection{Construction of the diffusion processes}\label{4.1-Girsanov} So far, we know that both the linear drift and the stochastic entropy are quantities concerning the average behavior of diffusions and they can be evaluated along typical paths.  To see how they vary  when we change the generators of the diffusions from $\LL$ to $\LL+Z$ (which also fulfills the requirement of  Proposition \ref{formulas-l-h-Y-l}  (or Proposition \ref{formulas-l-h-Y-h})),  our very first step is to  understand the change of distributions of the corresponding diffusion processes on the path spaces. For this, we
use techniques of stochastic differential equation (SDE) to construct on the same probability space all the diffusion processes.

We begin with the general theories of SDE on a  smooth manifold $\bf N.$ Let
$X_1, \cdots, X_d, V$ be bounded $C^1$ vector fields on a $C^3$ Riemannian
manifold $({\bf N}, \langle\cdot, \cdot\rangle)$.  Let
$B_t=(B_{t}^{1}, \cdots, B_{t}^{d})$ be a real $d$-dimensional Brownian
motion  on a probability space $(\Theta, \mathcal{F}, \mathcal{F}_t,
\Bbb Q)$ with generator $\Delta$. An ${\bf N}$-valued semimartingale ${\bf x}=({\bf
x}_t)_{t\in \Bbb R_+}$ defined up to a stopping time $\tau$ is said
to be a solution of the following Stratonovich SDE
\begin{equation}\label{SDE-x}
d{\bf x}_t=\sum_{i=1}^{d}X_i({\bf x}_t)\circ dB_{t}^{i} +V({\bf
x}_t)\ dt
\end{equation}
up to $\tau$ if for all $f\in C^{\infty}({\bf N})$,
\[
f({\bf x}_t)=f({\bf x}_0)+\int_{0}^{t} \sum_{i=1}^{d} X_i f({\bf
x}_s)\circ dB_{s}^{i}+\int_{0}^{t} Vf({\bf
x}_s)\ ds, \ 0\leq t<\tau.
\]
 Call a second order
differential operator $\mathbf{A}$ the generator of ${\bf x}$ if
\[
f({\bf x}_t)-f({\bf x}_0)-\int_{0}^{t}\mathbf{A} f({\bf x}_s)\ ds, \
0\leq t<\tau,
\]
is a local martingale for all $f\in C^{\infty}({\bf N})$. It is
known (cf.
 \cite{Hs}) that (\ref{SDE-x}) has a unique solution with a
 H\"{o}rmander type second order elliptic operator generator
 \[
\mathbf{A}=\sum_{i=1}^{d}X_i^2 + V.
 \]
If  $X_1, \cdots, X_d, V$ are such that the corresponding
$\mathbf{A}$ is the Laplace operator on ${\bf N}$, then the solution
of the SDE (\ref{SDE-x}) generates the Brownian motion on ${\bf N}$.
However, there is no general way of obtaining such a collection of
vector fields on a general Riemannian manifold.

To obtain  the Brownian motion $({\bf x}_t)_{t\in \R_+}$ on ${\bf
N}$, we  adopt the Eells-Elworthy-Malliavin approach (cf.
\cite{El}) by constructing a canonical diffusion on the frame bundle $O(\bf N)$.
 Let $TO(\bf N) $ be the tangent space of $O(\bf N)$. For $x\in \bf N$
and ${\bf {w}} \in O_x(\bf N)$, an element $\bf u \in T_{\bf w} O(\bf N) $ is
{\it {vertical }} if its projection on $T_x\bf N$ vanishes. The canonical connection associated with the metric defines a {\it {horizontal }} complement,  identified with $T_x\bf N.$ For a vector $v\in T_x \bf N$, ${\bf H}_v$,   the horizontal lift of $v$ to $T_{\bf w} O(\bf N)$,  describes the infinitesimal parallel transport of the frame $\bf w$ in the direction of $v.$ 

Suppose ${\bf N}$ has dimension $m$. Let
$B_t=(B_{t}^{1}, \cdots, B_{t}^{m})$ be an $m$-dimensional Brownian
motion  on a probability space $(\Theta, \mathcal{F}, \mathcal{F}_t,
\Bbb Q)$ with generator $\Delta$. Let $\{e_{i}\}$ be the standard orthonormal basis on
$\R^{m}$. Then, we consider  the  canonical diffusion on the
orthonormal bundle $O({\bf N})$ given by the solution ${\bf w}_t$ of
the Stratonovich SDE
\begin{eqnarray*} && d
{\bf w}_t=\sum_{i=1}^{m}{\bf H}_i({\bf w}_t)\circ dB_{t}^{i},\\
&& {\bf w}_0={\bf w},\end{eqnarray*}  where  ${\bf H}_i({\bf w}_t)$  is the horizontal lift of ${\bf w}_te_i$ to ${\bf w}_t$. The Brownian motion ${\bf
x}=({\bf x}_t)_{t\in \R_+}$  can be obtained as the projection on
${\bf N}$ of ${\bf w}_t$ for any choice of ${\bf w}_0$ which
projects to ${{\bf x}}_0$.  We can regard ${\bf x}(\cdot)$ as a
measurable map from $\Theta$ to $C_{{\bf x}_0}(\Bbb R_+, {\bf N})$,
the space of continuous functions $\rho$ from $\Bbb R_+$ to ${\bf
N}$ with $\rho(0)={\bf x}_0$. So
$$\P:=\Q({\bf x}^{-1})$$
gives the probability distribution of the Brownian motion paths in
$\Om_+$.  For any $\tau\in \Bbb R_+$,  let $C_{{\bf x}_0}([0, \tau],
{\bf N})$ denote the space of continuous functions $\rho$ from $[0,
\tau]$ to ${\bf N}$ with $\rho(0)={{\bf x}}_0$. Then ${\bf x}$ also
induces a measurable map ${\bf x}_{[0, \tau]}:\ \Theta \to C_{{\bf
x}_0}([0, \tau], {\bf N})$ which sends $\underline{\om}$ to $({\bf
x}_t(\underline{\om}))_{t\in [0, \tau]}$.  We see that
\[
\P_{\tau}:=\Q ({\bf x}_{[0, \tau]}^{-1})
\]
describes the distribution probability of the Brownian motion paths
on ${\bf N}$ up to time $\tau$.

More generally, we can obtain in the same way, and on the same probability space, a diffusion with generator $\D + V_1$, where $V_1$ is  a  bounded $C^1$ vector field on ${\bf N}$. We
denote by $\overline{V}_1$ the horizontal lift of $V_1$ in $O({\bf
N})$. Consider the Stratonovich SDE on $O (\bf N)$
\begin{eqnarray*} && d
{\bf u}_t=\sum_{i=1}^{m}{\bf H}_i({\bf u}_t)\circ dB_{t}^{i}+ \overline{V}_1({\bf u}_t)\ dt,\label{horizontal lift}\\
&& {\bf u}_0={\bf u}.\notag\end{eqnarray*} Then,  the diffusion process ${\bf y}=({{\bf y}}_t)_{t\in
\Bbb R_+}$ on ${\bf N}$ with infinitesimal generator $\Delta_{{\bf
N}}+V_1$  can be obtained as the
projection  on ${\bf N}$ of the  solution ${\bf u}_t$ for any choice
of ${\bf u}_0$ which projects to ${\bf y}_0$. We call ${\bf u}_t$
the  horizontal lift of ${\bf y}_t$.
 Let $\P^1$ be the distribution
of ${\bf y}$ in $C_{{\bf y}_0}(\Bbb R_+, \bf N)$ and let
$\P_{\tau}^1$ ($\tau\in \Bbb R_+$) be the distribution of $({\bf
y}_t(\underline{\om}))_{t\in [0, \tau]}$ in $C_{{\bf y}_0}([0,
\tau], {\bf N})$, respectively. Then
\[
\P^1=\Q({\bf y}^{-1}),\ \ \ \P^1_{\tau}=\Q({\bf y}_{[0,
\tau]}^{-1}).
\]

We now express the relation between $\P^1_\tau $ and $\P_\tau$,  as described by the Girsanov-Cameron-Martin formula.
Let ${\rm M}^1_t$ be the random process on $\Bbb R$ satisfying ${\rm
M}_{0}^1=1$ and the Stratonovich SDE
\[
d{\rm M}^1_{t}={\rm M}^1_{t}\langle\frac{1}{2}V_1({\bf x}_t), {\bf w}_t\circ
dB_t\rangle_{{\bf x}_t}-{\rm M}^1_{t}\left(\|\frac{1}{2}V_1({\bf x}_t)\|^2+{\rm
Div}\left(\frac{1}{2}V_1({\bf x}_t)\right)\right).
\]
Then
\[
{\rm M}^1_{t}=\exp\left\{\int_{0}^{t}\langle \frac{1}{2}V_1({\bf
x}_s(\u{\om})), {\bf w}_s(\u{\om})\circ dB_s(\u{\om})\rangle_{{\bf
x}_s}-\int_{0}^{t}\left(\|\frac{1}{2}V_1({\bf x}_s(\u{\om}))\|^2+{\rm Div}
\big(\frac{1}{2}V_1({\bf x}_s(\u{\om}))\big)\right)\ ds\right\}.
\]
In the more familiar Ito's stochastic integral form, we have
\[
d{\rm M}^1_{t}= \frac{1}{2}{\rm M}^1_{t}\langle V_1({\bf x}_t), {\bf w}_t
dB_t\rangle_{{\bf x}_t}
\]
and
\begin{equation}\label{Girsanov-martingale-new-coefficients}
{\rm M}^1_{t}=\exp\left\{\frac{1}{2}\int_{0}^{t}\langle V_1({\bf
x}_s(\u{\om})), {\bf w}_s(\u{\om}) dB_s(\u{\om})\rangle_{{\bf
x}_s}-\frac{1}{4}\int_{0}^{t}\|V_1({\bf x}_s(\u{\om}))\|^2\ ds\right\}.
\end{equation}
Since each $\E_{\Q}\left(\exp\{\frac{1}{4}\int_{0}^{t}\|V_1({\bf
x}_s(\u{\om}))\|^2\ ds\}\right)$ is finite, we have by Novikov
(\cite{N}),    that ${\rm M}^1_{t}, t\geq 0$, is a continuous
$(\mathcal{F}_t)$-martingale, i.e.,
\[
\E_{\Q}\left({\rm M}^1_t\right)=1\ \ \mbox{for every}\ t\geq 0,
\]
where $\E_{\Q}$ is the expectation of a random variable with respect
to $\Q$. For $\tau\in \Bbb R_+$, let $\Q^1_{\tau}$ be the probability
on $\Theta$, which is absolutely continuous with respect to $\Q$
with
\[
\frac{d{\Q^1_{\tau}}}{d\Q}(\u{\om})={\rm M}^1_{\tau}(\u{\om}).
\]
Note that ${\rm M}^1_{\tau}$ is a martingale, so that the projection
of $\Q^1_{\tau}$ on the coordinates up to $\tau' <\tau $ is given by
the same formula. A  version of the Girsanov theorem (cf.
\cite[Theorem 11B]{El}) says that $(({\bf y}_t)_{t\in [0, \tau]},
\Q)$ is isonomous to $(({\bf x}_t)_{t\in [0, \tau]}, \Q^1_\tau)$ in
the sense that for any finite numbers $\tau_1, \cdots, \tau_s\in [0,
\tau]$,
\begin{equation}\label{Girsanov-1}
\left(\Q({\bf y}_{\tau_1}^{-1}), \cdots, \Q({\bf
y}_{\tau_s}^{-1})\right)=\left(\Q^1_{\tau}({\bf
x}_{\tau_1}^{-1}),\cdots, \Q^1_{\tau}({\bf x}_{\tau_s}^{-1})\right).
\end{equation}
(The coefficients in (\ref{Girsanov-martingale-new-coefficients}) differ from the ones in  \cite{El} because $B_t$ has generator $\Delta$.)
Let $\Q^1$ be the probability on $\Theta$ associated with
$\{\Q^1_{\tau}\}_{\tau\in \Bbb R_+}$. Then (\ref{Girsanov-1})
intuitively means that by changing the measure $\Q$ on $\Theta$ to
$\Q^1$, ${\bf x}$ has the same distribution as $({\bf y}, \Q)$. As a
consequence, we have $\P^1_{\tau}=\Q^1_{\tau}({\bf x}^{-1})$ for all
$\tau\in \Bbb R_+$ and hence
\begin{equation*}
\frac{d \P_{\tau}^{1}}{d\P_{\tau}}\left({\bf x}_{[0,
\tau]}\right)=\E_{\Q}\left({\rm M}^1_{\tau}\big | \mathcal{F}({\bf
x}_{[0, \tau]})\right), \ \rm{a.s.},
\end{equation*}
where $\E_{\Q}\left(\cdot \ |\ \cdot\right)$ is the conditional
expectation with respect to $\Q$ and  $\mathcal{F}({\bf x}_{[0,
\tau]})$ is the smallest $\sigma$-algebra on $\Theta$ for which the
map ${\bf x}_{[0, \tau]}$ is measurable.

Let $V_2$ be another bounded $C^1$ vector field on ${\bf N}$. Consider the
diffusion process ${\bf z}=({\bf z}_t)_{t\in \Bbb R_+}$ on ${\bf N}$
with the same initial point as ${\bf y}$, but with infinitesimal
generator $\Delta_{{\bf N}}+V_1+V_2$. Let $\P^2$ be the distribution
of  ${\bf z}$ in the space of continuous paths on ${\bf N}$  and let $\P^2_{\tau}   (\tau\in \Bbb R_+)$
be the distribution of $({\bf z}_t(\underline{\om}))_{t\in [0,
\tau]}$. The Girsanov-Cameron-Martin formula on manifolds (cf.
\cite[Theorem 11C]{El}) says that $\P_{\tau}^2$ is absolutely
continuous with respect to $\P_{\tau}^1$ with
\begin{equation}\label{Girsanov-2}
\frac{d \P_{\tau}^{2}}{d\P_{\tau}^1}({\bf y}_{[0,
\tau]})=\E_{\Q}\left({\rm M}_{\tau}^{2}\big |\mathcal{F}({\bf
y}_{[0, \tau]})\right),\ \rm{a.s.},
\end{equation}
where
\[
{\rm M}^2_{\tau}(\u{\om})=\exp\left\{\frac{1}{2}\int_{0}^{\tau}\langle V_2({\bf
y}_s(\u{\om})), {\bf u}_s(\u{\om}) dB_s(\u{\om})\rangle_{{\bf
y}_s}-\frac{1}{4}\int_{0}^{\tau}\|V_2({\bf y}_s(\u{\om}))\|^2\ ds\right\}
\]
and  $\mathcal{F}({\bf y}_{[0, \tau]})$ is the smallest
$\sigma$-algebra on $\Theta$ for which the map ${\bf y}_{[0, \tau]}$
is measurable.

\section{Regularity of the linear drift and the stochastic entropy for $\Delta+Y$}

Consider a one-parameter family of variations
$\{\LL^{\lambda}=\Delta+Y+Z^{\l}:\ |\lambda|<1\}$ of $\LL$ with
$Z^0=0$ and $Z^{\lambda}$ twice differentiable in $\l$ so that
$\sup_{\l\in (-1, 1)}\max\{\|\frac{dZ^{\l}}{d\l}\|, \|\frac{d^2
Z^{\l}}{d\l^2}\|\}$ is finite.  Assume each $\LL^{\lambda}$ is
subordinate to the stable foliation, $Y+Z^{\l}$  is  such
that $(Y+Z^{\l})^*$, the dual of $(Y+Z^{\l})$ in the cotangent bundle  to the stable foliation over  $SM$,
satisfies $d(Y+Z^{\l})^*=0$ leafwisely and  $\mbox{pr}(-\langle \overline{X}, Y+Z^{\l}\rangle)>0$. Then each
$\LL^{\l}$ has a unique harmonic measure. Hence the linear drift for
$\LL^{\l}$,
 denoted  $\overline{\ell}_{\l}:=\ell_{\LL^{\l}}$, and the stochastic entropy
for $\LL^{\l}$, denoted  $\overline{h}_{\l}:=h_{\LL^{\l}}$, are well-defined. In this section, we  show the differentiability  of
$\overline{\ell}_{\l}$ and $\overline{h}_{\l}$ in $\l$ at $0$ (Theorem \ref{differential-ell} and Theorem \ref{differential-h}).

Consider the diffusion process of the stable foliation of $S\M$ corresponding to $\LL^{\l}$ ($\l\in (-1, 1)$). Let
$B_t=(B_{t}^{1}, \cdots, B_{t}^{m})$ be an $m$-dimensional Brownian
motion  on a probability space $(\Theta, \mathcal{F}, \mathcal{F}_t,
\Bbb Q)$ with generator $\Delta$.  For each ${\bf v}=(x, \xi)\in S\M$, $W^s({\bf v})$ can be
identified with $\M\times \{\xi\}$, or simply $\M$. So for each
$\l\in (-1, 1)$, there exists the diffusion process ${\bf y}_{\bf
v}^{\l}=({\bf y}^{\l}_{{\bf v}, t})_{t\in \Bbb R_+}$ on $W^s({\bf
v})$ starting from ${\bf v}$ with infinitesimal generator
$\LL^{\l}_{\bf v}$.  Each ${\bf y}_{\bf v}^{\l}$ induces a measurable map
from $\Theta$ to $C_{{\bf v}}(\Bbb R_+, W^s({\bf v}))\subset \Om_+$
and $\overline\P^{\l}_{{\bf v}}:=\Q({({\bf y}_{{\bf v}}^{\l})}^{-1})$ gives
the distribution probability of ${\bf y}_{\bf v}^{\l}$ in $C_{{\bf
v}}(\Bbb R_+, W^s({\bf v}))$. For any $\tau\in \Bbb R_+$, let
$\overline \P_{{\bf v}, \tau}^{\l}$ be the distribution of $({\bf
y}^{\l}_{{\bf v}, t})_{t\in [0, \tau]}$ in $C_{{\bf v}}([0, \tau],
W^s({\bf v}))$. We have by the Girsanov-Cameron-Martin formula on
manifolds (\ref{Girsanov-2}) that $\overline\P^{\l}_{{\bf v}, \tau}$
is absolutely continuous with respect to $\overline\P^{0}_{{\bf v},
\tau}$ with
\begin{equation}\label{Gir-P}
\frac{d \overline\P_{{\bf v}, \tau}^{\l}}{d\overline\P_{{\bf v},
\tau}^0}({\bf y}_{\bf v, [0, \tau]}^0 )=\E_{\Q}\left(\overline{\rm
M}^{\l}_{\tau}\big | \mathcal{F}\big({\bf y}_{\bf v, [0,
\tau]}^0\big)\right),\ \ \rm{a.s.},
\end{equation}
where \[ \overline{\rm
M}_{\tau}^{\l}(\u{\om})=\exp\left\{\frac{1}{2}\int_{0}^{\tau}\langle
Z^{\l}({\bf y}_{{\bf v}, s}^0(\u{\om})), {\bf u}_{{\bf v},
s}^0(\u{\om}) dB_s(\u{\om})\rangle_{{\bf y}_{{\bf v},
s}^0}-\frac{1}{4}\int_{0}^{\tau}\|Z^{\l}({\bf y}_{{\bf v}, s}^0(\u{\om}))\|^2\
ds\right\},
\]
${\bf u}_{{\bf v}, t}^0$ is the  horizontal lift of  ${\bf
y}^{0}_{{\bf v}, t}$ to   $O(W^s({\bf v}))$  and
$\mathcal{F}({\bf y}_{\bf v, [0, \tau]}^0)$ is the smallest
$\sigma$-algebra on $\Theta$ for which the map ${\bf y}_{\bf v, [0,
\tau]}^0$ is measurable.

 For each $\l\in (-1, 1)$, let ${\bf m}^{\l}$ be the unique
$\LL^{\l}$-harmonic measure and $\wt{\bf m}^{\l}$ be its  $G$-invariant  extension to  $S\M$. 
We see that $\overline
\P^{\l}=\int \overline\P_{{\bf v}}^{\l}\ d{\wt {\bf m}}^{\l}({\bf
v})$ is the shift invariant measure on ${\wt{ \Om}}_+$  corresponding to
$\wt {\bf m}^{\l}$  and we restrict $\overline
\P^{\l} $ to $\Om _+$.   Consider the  space $\overline{\Theta}=SM\times
\Theta$ with product $\sigma$-algebra and  probability
$\overline\Q^{\l}$, $d\overline\Q^{\l}({\bf v},
\un{\om})=d\Q(\un{\om})\times d{\bf m}^{\l}({\bf v})$.  Let
${\bf y}_t^{\l}:\ SM\times \Theta\to S\M$ be  such that
\[
{\bf y}_t^{\l}({\bf v}, \un{\om})={\bf y}^{\l}_{{\bf v},
t}(\un{\om}), \ {\rm{for}}\ ({\bf v}, \un{\om})\in SM\times \Theta.
\]
Then ${\bf y}^{\l}=({\bf y}^{\l}_t)_{t\in \Bbb R_+}$ defines a
random process on the probability space $(\overline\Theta,
\overline\Q^{\l})$ with images in the space of continuous paths on
the stable leaves of $S\M$. %\footnote{Should we say that we have a fixed fundamental domain for $\om (0)$?}\footnote{Yes, but where? See the blue line page 10....}

Simply write ${\bf y}_t={\bf y}^0_t$ and let ${\bf u}_t$ be such
that ${\bf u}_t({\bf v}, \un{\om})={\bf u}_{{\bf v}, t}^0(\un\om)$
for $({\bf v}, \un\om)\in \overline\Theta$. Denote by
$(Z^{\l})'_{0}:=(dZ^{\l}/d\l)|_{\l=0}$. We consider  three random
variables on $(\overline\Theta,\overline{\Bbb Q}^0)$:
\begin{eqnarray*}
{\bf M}_{t}^{0}&:=& \frac{1}{2}\int_{0}^{t}\langle (Z^{\l})'_{0}({\bf y}_{s}),
{\bf u}_{s}
dB_s \rangle_{{\bf y}_{s}},\\
{\bf Z}_{\ell, t}^0&:=& \left[
d_{\W}({\bf y}_0, {\bf y}_t)-t\ell_{\LL^0}\right],\\
{\bf Z}_{h, t}^0&:=&  -\left[{\bf 1}_{\{d({\bf y}_0, {\bf y}_t)\geq
1\}}\cdot\ln \bfg({\bf y}_0, {\bf y}_t)+th_{\LL^0}\right],
\end{eqnarray*}
where ${\bf 1}_{B}$ is the characteristic function of the event $B$.
We will prove the following two Propositions separately in Sections \ref{sec-4.2} and \ref{sec-4.3}.

\begin{prop}\label{Covariance-1} The laws of the  random vectors $({\bf Z}_{\ell, t}^0/\sqrt{t}, {\bf
M}_{t}^{0}/\sqrt{t})$ under $\overline\Q^0$ converge  in distribution as $t$
tends to $+\infty$ to a  bivariate centered normal  law  with some covariance
matrix $\Sigma_{\ell}$. The covariance matrices of $({\bf Z}_{\ell,
t}^0/\sqrt{t}, {\bf M}_{t}^{0}/\sqrt{t})$ under $\overline\Q^0$
converge to $\Sigma_{\ell}$.
\end{prop}

\begin{prop}\label{Covariance-2} The laws  of the random vectors $({\bf Z}_{h, t}^0/\sqrt{t}, {\bf
M}_{t}^{0}/\sqrt{t})$ under $\overline\Q^0$ converge in distribution as $t$
tends to $+\infty$ to a bivariate centered normal law with some covariance
matrix $\Sigma_{h}$. The covariance matrices of $({\bf Z}_{h,
t}^0/\sqrt{t}, {\bf M}_{t}^{0}/\sqrt{t})$ under $\overline\Q^0$
converge to $\Sigma_{h}$.
\end{prop}

\subsection{The differential of the linear drift}\label{sec-4.2}  For any $\l\in (-1, 1)$, let  $\overline{\ell}_{\l}$ be the linear drift of $\LL^{\l}$.  The main result of this subsection is the
following.
\begin{theo}\label{differential-ell}The function
$\l\mapsto \overline{\ell}_{\l}$   is  differentiable at $0$ with%\footnote{Not cited. I erase the citation number of the equation! Same for the  one after (\ref{pass to limit-ell}).}
\begin{equation*}%\label{equ-differential-ell}
\frac{d\overline{\ell}_{\l}}{d\l}\Big|_{\l=0}=\lim\limits_{t\to+\infty}\frac{1}{t}\E_{\overline\Q^0}({\bf
Z}_{\ell, t}^{0}{\bf M}_t^0).
\end{equation*}
\end{theo}

We fix a fundamental domain $M_0$ of $\M$ and identify $\Om_+$ with the lift of its elements in $\wt{\Omega}_+$ starting from $M_0$. In the following two  subsections,   we restrict the probabilities on $\wt{\Om}_+$  to $\Omega_+$.  
For any $\tau\in \Bbb R_+$,  recall that  $\overline \P_{{\bf v},
\tau}^{\l}$ is  the distribution of $({\bf y}^{\l}_{{\bf v},
t})_{t\in [0, \tau]}$ in $C_{{\bf v}}([0, \tau], W^s({\bf v}))$. By
an abuse of notation, we can also  regard   $\overline \P_{{\bf v},
\tau}^{\l}$ as a measure on $\Om_+$ whose value only depends on
$(\om(t))_{t\in [0, \tau]}$ for $\om=(\om(t))_{t\in \Bbb R_+}\in
\Om_+$.  Let $\overline \P^{\l}_t=\int \overline{\P}_{{\bf v}, t}^{\l}\ d{{\bf
m}}^{\l}({\bf v})$. Then
\[
\overline{\ell}_{\l}=\lim\limits_{t\to
+\infty}\frac{1}{t}\E_{\overline\P^{\l}_{t}}\left(d_{\W}(\om(0),
\om(t))\right).
\]

We will prove Theorem \ref{differential-ell} in two steps.
Firstly, using negative curvature, we find a
finite number $D_{\ell}$ such that for  all $\l\in [-\delta_1,
\delta_1]$ (where $\delta_1$ is from Lemma \ref{shadow lem}) and all
$t>0$,
\begin{equation}\label{sub-add-ell}
\big|\E_{\overline\P^{\l}}\left( d_{\W}(\om(0),
\om(t))\right)-t\overline\ell_{\l}\big|\leq D_{\ell}.\end{equation}
In particular, for $t = \l ^{-2}$, 
\[ \big|\l \E_{\overline\P^{\l}}\left( d_{\W}(\om(0),
\om(\l^{-2}))\right)-\frac{1}{\l} \overline\ell_{\l}\big|\leq \l D_{\ell}.\]
Thanks to (\ref{sub-add-ell}), the study of  $\displaystyle \frac{d\overline{\ell}_{\l}}{d\l}\Big|_{\l=0}=\lim\limits_{\l \to 0 } \frac{1}{\l} \big|\overline{\ell}_{\l} - \overline{\ell}_{0} \big | $ reduces to the study of \[ \lim\limits_{\l \to 0 } \big|\l \E_{\overline\P^{\l}}\left( d_{\W}(\om(0),
\om(\l^{-2}))\right)-\frac{1}{\l} \overline\ell_{0}\big|. \] 
Setting  $\l=\pm 1/\sqrt{t}$,  the second step is to show 
\begin{equation}\label{pass to limit-ell}
\lim_{t\to +\infty}
{\rm(I)}_{\ell}^t=\lim_{t\to +\infty} \E_{\overline\Q^0}\left(\frac{1}{\sqrt{t}} {\bf
Z}_{\ell, t}^0\cdot \overline{\rm
M}_{t}^{\l}\right),
\end{equation}
where 
\begin{equation*}%\label{dfn-I-ell-t} 
{\rm(I)}_{\ell}^t : = \E_{\overline\P^{\l}_t}\left(\frac{1}{\sqrt{t}}\left( d_{\W}(\om(0),
\om(t))-t\overline\ell_{0}\right)\right). \end{equation*}
Using notations from the previous subsection,  we know $d\overline{\P}^{\l}_t /d\overline\P^{0}_t$ is given by $\overline{\rm
M}_{t}^{\l}$. Each $\overline{\P}^{\l}_t$ is a random perturbation of the distribution $\overline{\P}^{0}_t$ on the path space and at the scale of $\l=1/\sqrt{t}$,   $\overline{\rm
M}_{t}^{\l}$  converge  in distribution to  $e^{{\bf
M}^0-\frac{1}{2}\E_{\overline\Q^0}(({\bf M}^{0})^2)}$ as $t$ goes to infinity. Consequently, 
$\frac{1}{\sqrt{t}} {\bf
Z}_{\ell, t}^0\cdot\overline{\rm
M}_{t}^{\l}$  converge  in distribution to  ${\bf Z}_{\ell}^{0}e^{{\bf
M}^0-\frac{1}{2}\E_{\overline\Q^0}(({\bf M}^{0})^2)}$, which can be identified with $\lim_{t\to+\infty}(1/t)\E_{\overline\Q^0}({\bf
Z}_{\ell, t}^{0}{\bf M}_t^0)$ using Proposition \ref{Covariance-1} (see Lemma \ref{ell-Y-2}).

 We therefore follow the above discussion and prove (\ref {sub-add-ell}) and (\ref{pass to limit-ell}).  Let us first show that there is a 
finite number $D_{\ell}$ such that for  all $\l\in [-\delta_1,
\delta_1]$  and all
$t>0$,
\begin{equation*}
\big|\E_{\overline\P^{\l}}\left( d_{\W}(\om(0),
\om(t))\right)-t\overline\ell_{\l}\big|\leq D_{\ell}.\end{equation*}
Since  the $\LL^{\l}$-diffusion has leafwise infinitesimal generator $\LL^{\l}_{\bf v}$ and $\P^{\l}$ is stationary,  we have
\begin{eqnarray*}
\E_{\overline\P^{\l}}\left(
b_{{\om}(0)}({\om}(t))\right)&=&\E_{\overline\P^{\l}}\left(\int_{0}^{t}\frac{\pp}{\pp s}b_{{\om}(0)}({\om}(s))\ ds\right)\\
  &=&\E_{\overline\P^{\l}}\left(\int_{0}^{t}(\LL_{{\om}(0)}^{\l}b_{{\om}(0)})({\om}(s))\ ds\right)\\
  &=&  t \int_{\scriptscriptstyle{M_0\times \pp\M}}\LL^{\l}_{\bf v}b_{{\bf v}}\ d\wt{{\bf m}}^{\l}\\
  &=& t\overline{\ell}_{\l}.
\end{eqnarray*}
So, proving  (\ref{sub-add-ell}) reduces to showing  that for all $\l\in[-\delta_1, \delta_1]$ and all $t>0$,
\begin{equation}\label{equivalent-sub-add-ell}
\E_{\overline\P^{\l}}\left(\left| d_{\W}({\om}(0), {\om}(t))-
b_{{\om}(0)}({\om}(t))\right|\right)<D_{\ell},
\end{equation}
which intuitively means that  for all $\l, \bfv $, the $\LL^{\l}_\bfv$-diffusion orbits ${\om}(t)$ does not accumulates to the point $\xi \in \pp \M$ such that  $\om (0) = \bfv  = (x,\xi )$. For $\om\in \Om_+$, we still denote $\om$ its projection to $\M$.  Then the leafwise distance $d_{\W}({\om}(0), {\om}(t))$ in (\ref{equivalent-sub-add-ell}) is  just $d({\om}(0), {\om}(t))$.

 We first take a look at the distribution of  ${\om}(\infty):=\lim_{s\to +\infty}{\om}(s)$  on the boundary.  Let
$x\in \M$ be a reference point and let $\iota>0$ be a positive
number.  Define
\[
d_x^{\iota}(\zeta, \eta):=\exp\left(-\iota(\zeta|\eta)_x\right), \
\forall \zeta, \eta\in \pp\M,
\]
where $(\zeta|\eta)_x$ is defined as in  (\ref{Gromov-product-def-sec3}).
If $\iota_0$ is small,  each  $d_x^{\iota}(\cdot, \cdot)$ ($x\in \M,
\iota\in (0, \iota_0)$) defines a distance on $\pp\M$ (\cite{GH}), the so-called
 \emph{$\iota$-Busemann distance}, which is related
to the Busemann functions since
\[
b_{{\bf v}}(y)=\lim\limits_{\zeta, \eta\to
\xi}\left((\zeta|\eta)_y-(\zeta|\eta)_x\right),\ \mbox{for any} \
{\bf v}=(x, \xi)\in S\M,  y\in \M.
\]
The following shadow lemma (\cite[Lemma 2.14]{Moh}, see also \cite{PPS})  says that the
$\LL^{\l}$-harmonic measure has a positive dimension on the boundary
in a uniform way.

\begin{lem}\label{shadow lem}There are $D_1, \delta_1,  \alpha_1, \iota_1>0$ such that for all $\l\in [-\delta_1,
\delta_1]$,  all ${\bf v}\in SM$ and all $\zeta\in \pp\M$, $t>0$,
\[
\overline\P_{{\bf v}}^{\l}\left(d_x^{\iota_1}\left(\zeta,
{\om}(\infty)\right)\leq e^{-t}\right)\leq D_1 e^{-\alpha_1 t},
\]
where  %\footnote{Delete? if we put it in page 30.}\lin{${\om}(\infty):=\lim_{s\to +\infty}{\om}(s)$ and} 
we identify 
${\om}(s)$ with its projection on $\M$.
\end{lem}

As a consequence, we see that for $\overline\P^{\l}$-almost all
orbits $\om\in \Omega_+$,  the distance between ${\om}(s)$ and
$\g_{{\om}(0), {\om}(\infty)}$, the geodesic connecting
${\om}(0)$ and ${\om}(\infty)$, is bounded in the following sense.

\begin{lem}\label{travel-along-geodesic} There exists a finite number $D_2$ such that for all $\l\in [-\delta_1,
\delta_1]$ (where $\delta_1$ is as in Lemma \ref{shadow lem}) and
$s\in \Bbb R_+$,
\begin{equation*}
\E_{\overline\P^{\l}}\left( d({\om}(s), \g_{{\om}(0),
{\om}(\infty)})\right)<D_2.
\end{equation*}
\end{lem}
\begin{proof}
  Extend $\overline\P^{\l}$ to a shift invariant probability measure
$\breve\P^{\l}$ on the set of trajectories from $\Bbb R$ to $SM$,
by
\[
\breve\P^{\l}=\int_{\scriptscriptstyle SM}\overline\P_{{\bf
v}}^{\l}\otimes (\overline\P')_{{\bf v}}^{\l}\ d{\bf m}^{\l}(\bf v),
\]
where $(\overline\P')_{{\bf v}}^{\l}$ is the probability describing
the reversed $\LL_{{\bf v}}^{\l}$-diffusion starting from ${\bf v}$.
Then we have by invariance of $\breve\P^{\l}$ that
\begin{eqnarray}
\E_{\overline\P^{\l}}\left( d({\om}(s), \g_{{\om}(0),
{\om}(\infty)})\right)&=& \E_{\breve \P^{\l}}\left(
d({\om}(0), \g_{{\om}(-s), {\om}(\infty)})\right)\notag\\
&=& \int \left( d(x, \g_{{\om}(-s), {\om}(\infty)})\right)
d\overline \P^\l_{\bf v}  (\wt {\om}) d (\overline \P')^\l _{\bf v}
(\wt {\om}(-s)) d{\bf m}^{\l}({\bf v}).\label{distance-om-s}
\end{eqnarray}
Fix ${\om} (-s ) = z $ at distance $D$ from $x$, and let $\zeta
\in \pp\M $ be  $\lim_{t\to +\infty}\gamma_{x, z}(t)$. We estimate
\[  \int d(x, \g_{z,
{\om}(\infty)})\  d\overline \P^\l_{\bf v}  (\wt {\om})= \int
_0^{+\infty} \overline \P^\l_{\bf v}  ( d(x, \g_{z,
{\om}(\infty)})
>t )\ dt. \] For $t \geq D $, it is clear that $ \overline
\P^\l_{\bf v} ( d(x, \g_{z, {\om}(\infty)}) >t ) = 0 $. For $t  <
D$, if $d(x, \g_{z, {\om}(\infty)}) >t $, then $
d_x^{\iota_1}\left(\zeta,
{\om}(\infty)\right) \leq C e^{-\iota_1 t}$ for some constant $C$ %(cf. \cite{K2})
and hence we have by Lemma \ref{shadow lem} that
\[  \overline \P^\l_{\bf v}  ( d(x, \g_{z, {\om}(\infty)}) >t ) \leq C D_1 e^{-\a_1 \iota_1 t} .\]
Therefore,
\[  \int  d(x, \g_{z,
{\om}(\infty)}) \  d\overline \P^\l_{\bf v}  (\wt {\om}) \leq
\int_1^D C D_1 e^{-\a_1 \iota_1 t}\  dt +1\leq \frac{C
D_1}{\a_1\iota_1}e^{-\alpha_1\iota_1}+1:= D_2.\] Using
(\ref{distance-om-s}), we obtain that  $\E_{\overline\P^{\l}}\left(
d({\om}(s), \g_{{\om}(0), {\om}(\infty)})\right)$ is
bounded by $D_2$ as well.
\end{proof}

Now,  using Lemmas \ref{shadow lem} and
\ref{travel-along-geodesic}, we prove in Lemma \ref{cov-bound-ell}  that there is a bounded square
integrable difference between $d_{\W}({\om}(0), {\om}(s))$ and
$b_{{\om}(0)}({\om}(s))$ for all $s$ (cf. \cite[Lemma 3.4]{Ma}).   This Lemma  \ref{cov-bound-ell}  implies  (\ref{equivalent-sub-add-ell}) and  therefore concludes the proof of (\ref{sub-add-ell}).

\begin{lem}\label{cov-bound-ell}
  There exists a finite number $D_3$ such that for all $\l\in [-\delta_1,
\delta_1]$ (where $\delta_1$ is as in Lemma \ref{shadow lem}) and
$s\in \Bbb R_+$,
\[
\E_{\overline\P^{\l}}\left(\left| d_{\W}({\om}(0), {\om}(s))-
b_{{\om}(0)}({\om}(s))\right|^2\right)<D_3.
\]
\end{lem}

\begin{proof} It is clear that
\begin{eqnarray*}
\E_{\overline\P^{\l}}\left(\left| d_{\W}({\om}(0), {\om}(s))-
b_{{\om}(0)}({\om}(s))\right|^2\right)= 4\int
\left({\om}(s)\big| \xi\right)_x^2 \ d\overline \P^\l_{\bf v}
({\om}) d{\bf m}^{\l}({\bf v}),
\end{eqnarray*}
where $\om(0)={\bf v}=(x, \xi)$  and $\left({\om}(s)\big| \xi\right)_x:=\lim_{y\to
\xi}\left({\om}(s)\big| y\right)_x$ (see (\ref{Gromov-product-def-sec3}) for the definition of $(z|y)_x$ for $x, y, z\in \M$).
So, it suffices to estimate
\begin{equation*}
\int_{0}^{+\infty} \overline\P^{\l}_{{\bf v}}( ({\om}(s)\big|
\xi)_x^2>t)\ dt=\int_{0}^{+\infty} \overline\P^{\l}_{{\bf v}}(
({\om}(s)\big| \xi)_x>\sqrt{t})\ dt.
\end{equation*}
 For each $t>0$, divide the
event  $\{\om\in \Om_+:\ ({\om}(s)\big| \xi)_x>\sqrt{t}\}$ into
two sub-events
\begin{eqnarray*}
A_1(t)&:=&\{\om\in \Om_+:\ ({\om}(s)\big| \xi)_x>\sqrt{t},\
({\om}(s)\big| {\om}(\infty))_x\geq \frac{1}{4}\sqrt{t}\},\\
A_2(t)&:=&\{ \om\in \Om_+:\ ({\om}(s)\big| \xi)_x>\sqrt{t}, \
({\om}(s)\big| {\om}(\infty))_x <\frac{1}{4}\sqrt{t} \}.
\end{eqnarray*}
We estimate $\overline\P^{\l}_{{\bf v}}(A_i(t))$, $i=1, 2$,
successively.  Since $M$ is a closed connected negatively curved
Riemannian manifold, its universal cover  $\M$ is Gromov hyperbolic in the sense that
there exists $\delta>0$ such that for any $x_1, x_2, x_3\in \M$,
\[
(x_1 | x_2)_x\geq \min\{(x_1  | x_3)_x,\ (x_2  | x_3)_x\}-\delta.
\]
So on each $A_1(t)$, where  $t>64\delta^2$, we have
\[
(\xi |{\om}(\infty))_x\geq \frac{1}{8}\sqrt{t}.
\]
Hence, by Lemma \ref{shadow lem},
\begin{eqnarray*}
\overline\P^{\l}_{{\bf v}}(A_1(t))\leq \overline\P^{\l}_{{\bf v}}((\xi |{\om}(\infty))_x\geq
\frac{1}{8}\sqrt{t})=\overline\P^{\l}_{{\bf
v}}(d_x^{\iota_1}({\om}(\infty),
\xi)<e^{-\frac{1}{8}\iota_1\sqrt{t}})\leq D_1
e^{-\frac{1}{8}\iota_1\alpha_1\sqrt{t}},
\end{eqnarray*}
where the last quantity is integrable with respect to $t$,
independent of $s$. For $\om\in A_2(t)$,
\[
d_{\W}({\om}(0), {\om}(s))\geq ({\om}(s)\big| \xi)_x>\sqrt{t}.
\]
On the other hand, the  point $y(s)$ on $\gamma_{{\om}(0),
{\om}(\infty)}$ closest to  ${\om}(s)$ satisfies
\[
({\om}(s)\big|y(s))_x
\leq ({\om}(s)\big| {\om}(\infty))_x
<\frac{1}{4}\sqrt{t}.
\]
So we must have
\[
d({\om}(s), \gamma_{{\om}(0),
{\om}(\infty)})>\frac{1}{2}\sqrt{t}.
\]
Hence,
\[
\int_{0}^{\infty} \overline\P^{\l}_{{\bf v}}(A_2(t))\ dt\leq  \int \overline\P^{\l}_{{\bf v}}\left(d({\om}(s),
\gamma_{{\om}(0), {\om}(\infty)})>\frac{1}{2}\sqrt{t}\right)\
dt\ d{\bf m}^{\l}({\bf v}),
\]
which,  by the same argument as the one used in the proof of Lemma \ref{travel-along-geodesic}, is  bounded from above by some constant independent of $s$.
\end{proof}

To show $\lim_{t\to +\infty}
{\rm(I)}_{\ell}^t =\lim_{t\to+\infty}(1/t)\E_{\overline\Q^0}({\bf
Z}_{\ell, t}^{0}{\bf M}_t^0)$, we first prove Proposition \ref{Covariance-1}.

\begin{proof}[Proof of Proposition \ref{Covariance-1}] Let $({\bf Z}_{t}^0)_{t\in \Bbb R_+}$,  $u_0$ be as  in Proposition \ref{M-0-M-1}.   The process $({\bf Z}_{t}^0)_{t\in \Bbb R_+}$ is a centered martingale with stationary increments and its law under $\overline\P^0$ is the same as the law of $(\overline{\bf Z}_{t}^{0})_{t\in \Bbb R_+}$ under
$\overline\Q^0$, where  $(\overline{\bf Z}_{t}^{0})_{t\in \Bbb R_+}$ on $(\overline\Theta, \overline\Q^0)$  is  given by
\[
\overline{\bf Z}_{t}^{0}({\bf v}, \un\om)=-b_{{\bf
v}}\left({\bf y}_{{\bf v}, t}(\un\om)\right)+t\overline{\ell}_0+
u_0 ({\bf y}_{{\bf v},
t}(\un\om)))-u_0 ({\bf v}).
  \]

The pair $(-\overline{\bf Z}_{t}^0, {\bf
M}_{t}^{0})$ is  a centered martingale on $(\overline\Theta, \overline\Q^0)$ with
stationary increments.  To show $(-\overline{\bf Z}_{t}^0/\sqrt{t}, {\bf
M}_{t}^{0}/\sqrt{t})$ converge in distribution to a bivariate centered normal vector, it suffices to show  for any  $(a, b)\in \Bbb R^2$,  the combination  $-a\overline{\bf Z}_{t}^0/\sqrt{t}+b{\bf
M}_{t}^{0}/\sqrt{t}$  converge to a centered normal distribution.  The martingales  $\overline{\bf Z}_{t}^{0}$  and  ${\bf
M}_{t}^{0}$ on $(\overline\Theta, \overline\Q^0)$ have integrable increasing process functions  $2\|\overline{X}+\nabla u_0\|^2$ and $\|(Z^{\l})'_{0}\|^2$, respectively.
Using Schwarz inequality, we conclude  that  $-a\overline{\bf Z}_{t}^0+b{\bf
M}_{t}^{0}$ also has an integrable increasing process function.  %Now using the fact that the $\LL^0$-diffusion system on $(\overline\Theta, \overline\Q^0)$ is  weakly mixing, we see that (\ref{check-CLT for martingale}) holds true for $-a\overline{\bf Z}_{t}^0+b{\bf
%M}_{t}^{0}$.
By Lemma \ref{CLT_martingale},  $-a\overline{\bf Z}_{t}^0/\sqrt{t}+b{\bf
M}_{t}^{0}/\sqrt{t}$  converge in distribution in $\overline\Q^0$ to a centered normal law  with variance $\Sigma_{\ell}[a,
b]=(a, b)\Sigma_{\ell}(a, b)^{T}$ for some matrix $\Sigma_{\ell}$.    Since  both
$\overline{\bf Z}_{t}^0$ and ${\bf M}_{t}^0$ have stationary
increments, we also have
\[
\Sigma_{\ell}[a, b]=\frac{1}{t}\Bbb E_{\overline\Q^0}\left[(-a\overline{\bf Z}_t^0+b{\bf
M}_t^0)^2\right],\ \mbox{for all}\ t\in \Bbb
R_+.
\]
The condition (\ref{check-CLT for martingale}) in Lemma \ref{CLT_martingale} is satisfied since the increasing process $\<-a\overline{\bf Z}_\cdot^0+b{\bf
M}_\cdot^0, -a\overline{\bf Z}_\cdot^0+b{\bf
M}_\cdot^0 \>_n $ is a Birkhoff sum of a square integrable function over a mixing system (Proposition \ref{mixing}). This shows Proposition \ref{Covariance-1} for the pair  $(-\overline{\bf
Z}_{t}^0/\sqrt{t}, {\bf M}_{t}^{0}/\sqrt{t} )$ instead of the pair  $({\bf Z}_{\ell, t}^0/\sqrt{t}, {\bf
M}_{t}^{0}/\sqrt{t}).$

Recall that  $\overline\P_{{\bf v}}^0$-a.e. $\om\in
\Om_+$ is such that 
$b_{{\bf v}}({\om}(t))-d_{\W}(\om(0), {\om}(t))$
converges to a finite number.   Moreover, we have by Lemma
\ref{cov-bound-ell} that
\[
\sup_{t} \E_{\overline\P^{\l}}(\big|{\bf Z}_{\ell,
t}^0+\overline{\bf Z}_{t}^0 \big|^2)<+\infty
\]
and hence
\[
\ \ \ \ \ \ \ \ \ \ \ \E_{\overline\P^{\l}}(\frac{1}{t}\big|{\bf Z}_{\ell,
t}^0+\overline{\bf Z}_{t}^0 \big|^2)\to 0,\ \rm{as}\ t\to +\infty.
\]
Consequently,  $({\bf Z}_{\ell, t}^0/\sqrt{t},
{\bf M}_{t}^{0}/\sqrt{t})$  has the same limit normal law  as  $(-\overline{\bf
Z}_{t}^0/\sqrt{t}, {\bf M}_{t}^{0}/\sqrt{t} )$ and its covariance matrix  under $\overline\Q^0$
converges to $\Sigma_{\ell}$ as $t$ goes to infinity.
\end{proof}

We  state  one  more  lemma  from \cite{Bi}
on the limit of the expectations of a class of random
variables
 on a common probability space which converge  in distribution.

 \begin{lem}(cf. \cite[Theorem 25.12]{Bi})\label{convergence-lemma}
   If the random variables $X_t$ ($t\in \Bbb R$) on a common probability space converge to $ X$ in
   distribution, and there exists some $q>1$ such that
   $\sup_{t}\E_{\nu}\left(|X_t|^{q}\right)<+\infty$, then $X$ is
   integrable and
   \[
\lim\limits_{t\to
+\infty}\E_{\nu}\left(X_t\right)=\E_{\nu}\left(X\right).
   \]
 \end{lem}

By the above discussion, Theorem \ref{differential-ell} follows from
\begin{lem}\label{ell-Y-2}
 We have  $\lim_{t \to +\infty }{\rm(I)}_{\ell}^t =\lim_{t\to+\infty}(1/t)\E_{\overline\Q^0}({\bf Z}_{\ell, t}^{0}{\bf
M}_t^0)$.
\end{lem}
%By Proposition \ref{CLT}, the distribution of  $\left(d_{\W}(\om(0),
%\om(t))-t\overline\ell_0\right)/\sqrt{t}$ under $\overline\P^0$  is
%asymptotic to a centered normal distribution. Hence
%\begin{eqnarray*} {\rm(II)}_{\ell}
%&=&\lim_{t\to +\infty}
%\E_{\overline\P^{\l}_t}\left(\frac{1}{\sqrt{t}}\left( d_{\W}(\om(0),
%\om(t))-t\overline\ell_{0}\right)\right)-\lim_{t\to
%+\infty}\E_{\overline\P^{0}}\left(\frac{1}{\sqrt{t}}\left(
%d_{\W}(\om(0),
%\om(t))-t\overline\ell_{0}\right)\right)\\
%&=&  \lim_{t\to +\infty}
%\E_{\overline\P^{\l}_t}\left(\frac{1}{\sqrt{t}}\left( d_{\W}(\om(0),
%\om(t))-t\overline\ell_{0}\right)\right)\\
%&=:& \rm{(III)}_{\ell}.
%\end{eqnarray*}
\begin{proof}
Let  ${\bf y}=({\bf y}_t)_{t\in \Bbb R_+}=({\bf y}_{{\bf v},
t})_{{\bf v}\in SM, t\in \Bbb R_+}$  be the diffusion process on
$(\overline\Theta, \overline\Q^{0})$ corresponding to $\LL^{0}$.
We know from Section \ref{4.1-Girsanov}   that  $\overline\P^{\l}_{{\bf
v}, t}$ is absolutely continuous with respect to
$\overline\P^{0}_{{\bf v}, t}$ with
\begin{equation*}
\frac{d \overline\P_{{\bf v}, t}^{\l}}{d\overline\P_{{\bf v},
t}^0}({\bf y}_{\bf v, [0, t]})=\E_{\Q}\left(\overline{\rm
M}^{\l}_{t}\big | \mathcal{F}({\bf y}_{\bf v, [0, t]})\right),
\end{equation*}
where \[ \overline{\rm
M}_{t}^{\l}(\u{\om})=\exp\left\{\frac{1}{2}\int_{0}^{t}\langle Z^{\l}({\bf
y}_{{\bf v}, s}(\u{\om})), {\bf u}_{{\bf v}, s}(\u{\om})
dB_s(\u{\om})\rangle_{{\bf y}_{{\bf v},
s}(\u{\om})}-\frac{1}{4}\int_{0}^{t}\|Z^{\l}({\bf y}_{{\bf v},
s}(\u{\om}))\|^2\ ds\right\}.
\]
Consequently  we have
\begin{eqnarray*}
 \lim_{t \to +\infty}\rm{(I)}_{\ell}^t&=&\lim_{t\to +\infty} \E_{\overline\P^0}\left(\frac{1}{\sqrt{t}}\left( d_{\W}(\om(0),
\om(t))-t\overline\ell_{0}\right)\frac{d\overline\P_{\om(0), t}^{\l}}{d\overline\P_{\om(0), t}^{0}}\right)\\
& =&\lim_{t\to +\infty}
\E_{\overline\Q^0}\left(\frac{1}{\sqrt{t}}\left( d_{\W}({\bf y}_0,
{\bf y}_t)-t\overline\ell_{0}\right)\cdot
e^{{\rm{(II)}}_{\ell}^{t}}\right),
\end{eqnarray*}
where \[ {\rm{(II)}}_{\ell}^{t}=\frac{1}{2}\int_{0}^{t}\langle Z^{\l}({\bf y}_{
s}({\bf v}, \un{\om})), \ {\bf u}_{s}({\bf v}, \un{\om})d
B_{s}\rangle_{{\bf y}_{ s}({\bf v},
\un{\om})}-\frac{1}{4}\int_{0}^{t}\|Z^{\l}({\bf y}_{s}({\bf v}, \un{\om}))\|^2\
ds.
\]
Let $\overline{Z}^{\l}$ be such that $Z^{\l}=\l
(Z^{\l})'_{0}+\l^2\overline{Z}^{\l}$.
 We calculate for $\l=1/\sqrt{t}$   that
\begin{eqnarray*}
{\rm{(II)}}_{\ell}^{t}&=& \frac{1}{2\sqrt{t}}\int_{0}^{t}\langle (Z^{\l})'_0({\bf y}_{s}), \ {\bf u}_{s}d B_{s}\rangle_{{\bf y}_{s}}-\frac{1}{4t}\int_{0}^{t}\|(Z^{\l})'_0({\bf y}_{s})\|^2\ ds\\
&& +\frac{1}{2t}\int_{0}^{t}\langle \overline{Z}^{\l}({\bf y}_{s}), \ {\bf u}_{s}d B_{s}\rangle_{{\bf y}_{s}}\\
&& -\frac{1}{2t^{\frac{3}{2}}}\int_{0}^{t}\langle (Z^{\l})'_0({\bf y}_{s}), \overline{Z}^{\l}({\bf y}_{s})\rangle_{{\bf y}_{s}}\ ds-\frac{1}{4t^2}\int_{0}^{t} \|\overline{Z}^{\l} ({\bf y}_{s})\|^2\ ds\\
&=:&\frac{1}{\sqrt{t}}{\bf
M}_{t}^0-\frac{1}{2t}\langle {\bf
M}_t^0, {\bf
M}_t^0\rangle_t+  \rm{(III)}_{\ell}^{t}+
\rm{(IV)}_{\ell}^{t},
\end{eqnarray*}
where both $\rm{(III)}_{\ell}^{t}$ and $\rm{(IV)}_{\ell}^{t}$ converge
almost surely to zero  as $t$ goes to infinity.   Therefore, by Proposition  \ref{Covariance-1},  the variables  $\frac{1}{\sqrt{t}} {\bf
Z}_{\ell, t}^0\cdot\overline{\rm
M}_{t}^{\l}$  converge  in distribution to  ${\bf Z}_{\ell}^{0}e^{{\bf
M}^0-\frac{1}{2}\E_{\overline\Q^0}(({\bf M}^{0})^2)},$ where $( {\bf Z}_{\ell}^{0}, {\bf M}^{0}) $ is a bivariate normal  variable with covariance matrix $\Sigma _\ell .$

Indeed, to  justify
\begin{equation*}
\lim_{t \to +\infty } \rm{(I)}_{\ell}^t=\lim_{t\to +\infty} \E_{\overline\Q^0}\left(\frac{1}{\sqrt{t}} {\bf
Z}_{\ell, t}^0\cdot \overline{\rm
M}_{t}^{\l}\right) =
\E_{\overline\Q^0}\left({\bf Z}_{\ell}^{0}e^{{\bf
M}^0-\frac{1}{2}\E_{\overline\Q^0}(({\bf M}^{0})^2)}\right),
\end{equation*}
we have by Lemma \ref{convergence-lemma} that it suffices to show
for $q=\frac{3}{2}$,
\[
\sup\limits_{t}\E_{\overline\Q^0}\left(\left|\frac{1}{\sqrt{t}} {\bf
Z}_{\ell, t}^0\cdot
\overline{\rm
M}_{t}^{\l}\right|^{q}\right)<+\infty.
\]
By H\"{o}lder's inequality, we calculate that
\begin{eqnarray*}
\left(\E_{\overline\Q^0}\left(\left|\frac{1}{\sqrt{t}} {\bf
Z}_{\ell, t}^0\cdot
\overline{\rm
M}_{t}^{\l}\right|^{\frac{3}{2}}\right)\right)^4&\leq&
\left(\E_{\overline\Q^0}\left(\frac{1}{t}\left| {\bf Z}_{\ell,
t}^0\right|^2\right)\right)^{3}\cdot \E_{\overline\Q^0}\left(e^{6{\rm{(IV)}}_{\ell}^{t}}\right)\\
&=:& \rm{(V)}_{\ell}^{t}\cdot \rm{(VI)}_{\ell}^{t},
\end{eqnarray*}
where $\rm{(V)}_{\ell}^{t}$ is uniformly bounded in $t$ by
Proposition \ref{Covariance-1}. For $\rm{(VI)}_{\ell}^{t}$,  we use the Girsanov-Cameron-Martin formula to conclude that
\begin{eqnarray*}
\rm{(VI)}_{\ell}^{t}&=& \E_{\overline\Q^0}\left(\exp
\left(3\int_{0}^{t}\langle Z^{\l}({\bf y}_{ s}({\bf v}, \un{\om})),
\ {\bf u}_{s}({\bf v}, \un{\om})d B_{s}\rangle_{{\bf y}_{ s}({\bf
v}, \un{\om})}-\frac{3}{2}\int_{0}^{t}\|Z^{\l}({\bf y}_{s}({\bf v},
\un{\om}))\|^2\ ds\right)\right)\\
&\leq& \E_{\widetilde{\Q}}\left(\exp\left(\frac{15}{2}\int_{0}^{t}\|Z^{\l}({\bf
y}_{s}({\bf v}, \un{\om}))\|^2\ ds\right)\right)
\end{eqnarray*}
for some probability measure $\widetilde{\Q}$ on $\overline\Theta$.
Using again $Z^{\l}=\l (Z^{\l})'_{0}+\l^2\overline{Z}^{\l}$ and that
$\l=1/\sqrt{t}$, we see that
\[
\int_{0}^{t}\|Z^{\l}({\bf y}_{s}({\bf v}, \un{\om}))\|^2\ ds\leq
\frac{2}{t}\int_{0}^{t} \|(Z^{\l})'_{0}\|^2\
ds+\frac{2}{t^2}\int_{0}^{t} \|\overline{Z}^{\l}\|^2\ ds,
\]
where the quantities on the right hand side of the inequality are
uniformly bounded in $t$. So $\rm{(VI)}_{\ell}^{t}$ is uniformly
bounded in $t\geq 1$ as well. Now, (\ref{pass to limit-ell}) holds. The calculation is the same with $\l = -1/\sqrt t.$ 

Finally, since $({\bf Z}_{\ell}^0,{\bf M}^{0})$ has a
bivariate normal
distribution, we have\[
\E_{\overline\Q^0}\left({\bf Z}_{\ell}^{0}e^{{\bf
M}^0-\frac{1}{2}\E_{\overline\Q^0}(({\bf
M}^{0})^2)}\right)=\E_{\overline\Q^0}({\bf Z}_{\ell}^{0}{\bf M}^0),
\]\footnote{We leave the proof of the equality as an exercise.  Let the couple $(x,y) $ have a bivariate centered normal distribution. By diagonalizing the covariance matrix, we may assume that \begin{eqnarray*} x &=& \cos \theta X - \sin \theta Y \\ y &=& \sin \theta X + \cos \theta Y, \end{eqnarray*} where $X$ and $Y$ are independent centered normal distributions with variance $\sigma ^2 $ and $\tau ^2 $ respectively. Then by independence, all $\E xy, \E y^2 $ and $\E xe^y$ are easy to compute and one finds $\E xy = \E (x e^{y - \E y^2 /2})$.} which is  $\lim_{t\to+\infty}(1/t)\E_{\overline\Q^0}({\bf Z}_{\ell,
t}^{0}{\bf M}_t^0)$ by  Proposition \ref{Covariance-1}.
\end{proof}

\subsection{The differential of the stochastic entropy}\label{sec-4.3}  For any $\l\in (-1, 1)$, let  $\overline{h}_{\l}$ be the entropy of $\LL^{\l}$.  In this subsection, we establish the following formula for $(d\overline{h}_{\l}/d\l)|_{\l=0}$.

\begin{theo}\label{differential-h}The function
$\l\mapsto \overline{h}_{\l}$   is  differentiable at $0$ with
\[
\frac{d\overline{h}_{\l}}{d\l}\Big|_{\l=0}=\lim\limits_{t\to+\infty}\frac{1}{t}\E_{\overline\Q^0}({\bf
Z}_{h, t}^{0}{\bf M}_t^0).
\]
\end{theo}

Since  $c_i, 0\leq i\leq 8,$ and $\alpha_2, \alpha_3$ of Lemmas \ref{GM-classical Harnack}--\ref{triangle-separated by cones} depend
only on the geometry of $\M$ and  the coefficients of $\LL$, we  may 
assume the constants are such that  these lemmas hold true for every couple
$\LL^{\l}, {\bf G}^{\l}$ with $\l\in (-1, 1)$.

For each  $\l \in (-1, 1)$, by Lemma \ref{almost-triangle-inequ} and the Subadditive Ergodic Theorem we obtain a constant $\overline
    h_{\l, 0}$ such that for  $\overline\P^{\l}$-a.e. $\om\in \Om_+$,
\begin{equation*}\lim\limits_{t\to
+\infty}-\frac{1}{t} \ln G_{{\bf v}}^{0}({\om}(0),
{\om}(t))=\overline h_{\l, 0}, \ \mbox{where}\ {\bf v}=\om(0).
\end{equation*}
For $\overline\P^{\l}$-a.e. $\om\in \Om_+$, since ${\om}(t)$ converges to a point in $\pp\M$
as $t$ tends to infinity, we also have
\begin{equation}\label{h_lambda_0-definition}
\lim\limits_{t\to +\infty}\frac{1}{t} d_{G_{{\bf
v}}^{0}}({\om}(0), {\om}(t))=\overline h_{\l, 0}. 
\end{equation}
The equation (\ref{h_lambda_0-definition})  continues to hold if we  replace  the pathwise limit by its expectation.   Actually, since $\overline{\P}^{\lambda}$ is a probability on $\Om_+$, using (\ref{d-G-d-greater-1}),  we see
that
\[
\E_{\overline\P^{\l}}\left(\sup_{t\in [0, 1]} [d_{G_{{\bf
v}}^{0}}({\om}(0), {\om}(t))]^2\right)\leq
2\alpha_2^2\E_{\overline\P^{\l}}\left(\sup_{t\in
[0,1]}[d_{\W}(\om(0), \om(t))]^2\right) +8(\ln c_2)^2<+\infty. 
\]
So, using
(\ref{almost-triangle-inequ}), we have by the Subadditive Ergodic
Theorem  that  for $\l \in (-1, 1)$,  
\[\lim\limits_{t\to
+\infty}\frac{1}{t}\E_{\overline\P^{\l}}\left( d_{G_{{\bf
v}}^{0}}({\om}(0), {\om}(t))\right)=\overline h_{\l, 0}.\]
 %By Proposition \ref{prop-h-in p-g},  $\overline
%h_{\l}=\overline h_{\l, \l}$ and $\overline h_{0}=\overline h_{0,0}$.

The main strategy to prove Theorem \ref{differential-h} is to
split $(\overline h_{\l}-\overline h_0)/\l$ into two terms:
\begin{eqnarray*}
\frac{1}{\l}(\overline h_{\l}-\overline h_0)=\frac{1}{\l}(\overline
h_{\l}-\overline h_{\l, 0})+\frac{1}{\l}(\overline h_{\l,
0}-\overline h_{0})=:{\rm{(I)}}_{h}^{\l} +
{\rm{(II)}}_{h}^{\l},\label{h-lambda-I-II}
\end{eqnarray*}
then show $\lim_{\l\to 0}{\rm{(I)}}_{h}^{\l}=0$ and  $\lim_{\l\to
0}{\rm{(II)}}_{h}^{\l}=\lim_{t\to+\infty}(1/t)\E_{\overline\Q^0}({\bf
Z}_{h, t}^{0}{\bf M}_t^0)$
successively. Since $d_{G_{{\bf
v}}^{0}}$  behaves in the same way as a distance function, the terms $\overline h_{\l, 0}$ and $\overline h_{0}$ are the `linear drifts' of the diffusions with respect these `distances' in distributions $\overline \P^{\l}$ and $\overline \P^{0}$, respectively. Hence $\lim_{\l\to 0}{\rm{(II)}}_{h}^{\l}$ can be evaluated  by following the evaluation of  $(d\overline{\ell}_{\l}/ d\l)|_{\l=0}$  in Section \ref{sec-4.2}. The new term  ${\rm{(I)}}_{h}^{\l}$  represents the contribution of the of change of Green `metric' between $G_{\bfv}^0 $ and $G_{\bfv}^\l$. It turns out that this contribution is of order $\l^2$ for $C^1$ drift change of $\LL^{0}$. Consequently, we have the following.

\begin{lem}
  $\lim_{\l\to 0}{\rm{(I)}}_{h}^{\l}= 0$.\end{lem}
  
  \begin{proof}
  For each $\lambda\in (-1, 1)$, recall that by Proposition \ref{prop-h-in p-g} we have for $\overline{\P}^{\lambda}$-a.e. $\om\in \Omega_+$,  $\om(0)=:{\bf v}$, 
  \begin{eqnarray}
  \overline{h}_{\lambda}&=&\lim\limits_{t\to +\infty}-\frac{1}{t}\ln p_{{\bf v}}^{\l}(t, \om(0), \om(t))\notag\\
  &=& \lim\limits_{t\to +\infty}-\frac{1}{t}\int \left(\ln p_{\bf v}^{\l}(t, x, y)\right)  p_{\bf v}^{\l}(t, x, y)\ dy. \label{h-lambda-lambda-integral formula}
   \end{eqnarray}
   Similarly, by the same proof as for Proposition   \ref{prop-h-in p-g}, we have that 
   \begin{equation}\label{h-lambda-0-inf-s}
   \overline h_{\l, 0}=\inf_{s>0}\{\overline h_{\l, 0}(s)\},
   \end{equation}
  where  for $\overline{\P}^{\lambda}$-a.e. $\om\in \Omega_+$,  $\om(0)=:{\bf v}$, 
    \begin{eqnarray}
  \overline{h}_{\lambda, 0}(s)&:=&\lim\limits_{t\to +\infty}-\frac{1}{t}\ln p_{{\bf v}}^{0}(st, \om(0), \om(t))\notag\\
  &=& \lim\limits_{t\to +\infty}-\frac{1}{t}\int \left(\ln p_{\bf v}^{0}(st, x, y)\right)  p_{\bf v}^{\l}(t, x, y)\ dy.\label{h-lambda-0-integral formula}
   \end{eqnarray}
Since we are considering  the pathwise limit on $p^0_{\bf v}$ with respect to $\overline{\P}^{\l}$, the infimum in (\ref{h-lambda-0-inf-s}) is not necessarily obtained at $s=1$.  But  we still have $\limsup_{\l\to 0+}{\rm{(I)}}_{h}^{\l}\leq 0$ since
\begin{eqnarray*}
\overline{h}_{\l}-\overline{h}_{\l, 0}&=&\sup_{s>0}\left\{\lim\limits_{t\to +\infty}-\frac{1}{t}\int \left(\ln \frac{p_{\bf v}^{\l}(t, x, y)}{p_{\bf v}^0(st, x, y)}\right)\ p_{\bf v}^{\l}(t, x, y)\ dy\right\}\\
&\leq&\sup_{s>0}\left\{\lim\limits_{t\to +\infty}\frac{1}{t}\int  \frac{p_{\bf v}^0(st, x, y)}{p_{\bf v}^{\l}(t, x, y)}\ p_{\bf v}^{\l}(t, x, y)\ dy\right\}\notag\\
&=& 0,\notag
\end{eqnarray*}
where we use $-\ln a \leq a^{-1}-1$ for $a>0$  to derive the second inequality.

To show  $\liminf_{\l\to 0+}{\rm{(I)}}_{h}^{\l}\geq 0$,  we observe that  by (\ref{h-lambda-lambda-integral formula}), (\ref{h-lambda-0-inf-s}) and (\ref{h-lambda-0-integral formula}), 
\begin{equation}\label{h-lambda-0-lambda-lb}\overline
h_{\l}-\overline h_{\l, 0}\geq \overline
h_{\l}-\overline h_{\l, 0}(1)=\lim\limits_{t\to +\infty}-\frac{1}{t}\int \left(\ln \frac{p_{\bf v}^{\l}(t, x, y)}{p_{\bf v}^0(t, x, y)}\right)\ p_{\bf v}^{\l}(t, x, y)\ dy.\end{equation}
We proceed to estimate $\ln (p_{\bf v}^{\l}(t, x, y)/p_{\bf v}^0(t, x, y))$ using  the Girsanov-Cameron-Martin formula in  Section  \ref{4.1-Girsanov}.  For  ${\bf v}, {\bf w}\in S\M$,  let $\Omega_{{\bf v}, {\bf w}, t}$ be the collection of $\om\in \Om_+$ such that  $\om(0)={\bf v}, \om(t)={\bf w}$. Since the space $\Om_+$ is separable, the measure $\P^{\l}$ disintegrates into a class of conditional probabilities $\{\P_{{\bf v, w}, t}^{\l}\}_{{\bf v}, {\bf w}\in S\M}$ on  $\Omega_{{\bf v}, {\bf w}, t}$'s  such that
\begin{equation}\label{disintegration of p-t-x-y}
\E_{\P_{{\bf v, w}, t}^{\l}}\left(\frac{d\P_{t}^{0}}{d\P_{t}^{\l}}\right)=\frac{ {\bf p}^{0}(t, {\bf v}, {\bf w})}{ {\bf p}^{\l}(t, {\bf v}, {\bf w})}.
\end{equation}
Letting ${\bf v}=(x, \xi), {\bf w}=(y, \xi)$ in (\ref{disintegration of p-t-x-y}), we obtain
\begin{equation}\label{ln-p-0-lambda}
\ln \frac{p^0_{\bf v}(t, x, y)}{p^{\l}_{\bf v}(t, x, y)}=\ln\left(\E_{\P_{{\bf v, w}, t}^{\l}}\left(\frac{d\P_{{\bf v}, t}^{0}}{d\P_{{\bf v}, t}^{\l}}\right)\right)\geq  \E_{\P_{{\bf v, w}, t}^{\l}}\left(\ln\left(\frac{d\P_{{\bf v}, t}^{0}}{d\P_{{\bf v}, t}^{\l}}\right)\right).
\end{equation}
Recall that
\begin{equation*}
\frac{d \overline\P_{{\bf v}, t}^{0}}{d\overline\P_{{\bf v},
t}^{\l}}({\bf y}^{\l}_{{\bf v}, [0, t]})=\E_{\Q^{\l}}\left(\overline{\rm
M}^{\l}_{t}\big | \mathcal{F}({\bf y}^{\l}_{{\bf v}, [0, t]})\right),
\end{equation*}
where ${\bf y}^{\l}=({\bf y}^{\l}_{{\bf v},
t})_{{\bf v}\in SM, t\in \Bbb R_+}$ is the diffusion process on
$(\overline\Theta, \overline\Q^{\l})$ corresponding to $\LL^{\l}$ and \[ \overline{\rm
M}_{t}^{\l}(\u{\om})=\exp\left\{-\frac{1}{2}\int_{0}^{t}\langle Z^{\l}({\bf
y}^{\l}_{{\bf v}, s}(\u{\om})), {\bf u}^{\l}_{{\bf v}, s}(\u{\om})
dB_s(\u{\om})\rangle_{{\bf y}^{\l}_{{\bf v},
s}(\u{\om})}-\frac{1}{4}\int_{0}^{t}\|-Z^{\l}({\bf y}^{\l}_{{\bf v},
s}(\u{\om}))\|^2\ ds\right\}.
\]
So we can further deduce from (\ref{ln-p-0-lambda}) that
%\footnote{Or like this? Maybe like this!}
%\begin{small}
\begin{eqnarray*}%\hspace*{-0.1cm}
& &\ln \frac{p^0_{\bf v}(t, x, y)}{p^{\l}_{\bf v}(t, x, y)}\\ && \geq  \E_{\P_{{\bf v, w}, t}^{\l}}\left(\E_{\Q^{\l}}\left( (-\frac{1}{2}\int_{0}^{t}\langle Z^{\l}({\bf
y}^{\l}_{{\bf v}, s}), {\bf u}^{\l}_{{\bf v}, s}
dB_s\rangle_{{\bf y}^{\l}_{{\bf v},
s}}-\frac{1}{4}\int_{0}^{t}\|Z^{\l}({\bf y}^{\l}_{{\bf v},
s})\|^2\ ds )\big| \mathcal{F}({\bf y}^{\l}_{{\bf v}, [0, t]})\right)\right)\notag\\
&&= -\E_{\P_{{\bf v, w}, t}^{\l}}\left(\E_{\Q^{\l}}\left( (\frac{1}{4}\int_{0}^{t}\|Z^{\l}({\bf y}^{\l}_{{\bf v},
s})\|^2\ ds )\big | \mathcal{F}({\bf y}^{\l}_{{\bf v}, [0, t]})\right)\right)\; 
\geq -\frac{1}{4}(\l C)^2 t,
\end{eqnarray*}%\end{small}
where the equality holds true since $\int_{0}^{t}\langle Z^{\l}({\bf
y}^{\l}_{{\bf v}, s}), {\bf u}^{\l}_{{\bf v}, s}
dB_s\rangle_{{\bf y}^{\l}_{{\bf v},
s}}$ is a centered martingale  and $C$ is some constant which bounds the norm of $dZ^{\l}/d\l$.  Reporting this in (\ref{h-lambda-0-lambda-lb}) gives
\[
\liminf_{\l\to 0+}{\rm{(I)}}_{h}^{\l}=\liminf_{\l\to 0+}\frac{1}{\l}\left(\overline{h}_{\l}-\overline{h}_{\l, 0}\right)\geq -\frac{1}{4}\limsup_{\l\to 0+}(\l C^{2}) =0.
\]
We prove in the same way (with switched  arguments) that  $\lim_{\l\to 0-}{\rm{(I)}}_{h}^{\l}= 0$.
   \end{proof}

The analysis of ${\rm{(II)}}_{h}^{\l}$ is analogous to that was used
for $(d\overline{\ell}_{\l}/d\l)|_{\l=0}$.
% $\displaystyle \frac{d\overline{\ell}_{\l}}{d\l}\Big|_{\l=0}$. % Let $\l=\pm1/\sqrt{t}$  and write
%\begin{eqnarray}\label{III-h-t}
%\ \ \ \ \ \ \ \ \
%{\rm{(III)}}_{h}^{t}:=\left(\E_{\overline\P^{\l}_t}(\frac{1}{\sqrt{t}}d_{G_{{\bf
%v}}^{0}}({\om}(0),
%{\om}(t)))-\E_{\overline\P_t^0}(\frac{1}{\sqrt{t}}d_{G_{{\bf
%v}}^{0}}({\om}(0), {\om}(t)))\right).
%\end{eqnarray}
%We claim  that if $\lim_{t\to+\infty}{\rm{(III)}}_{h}^{t}$ exists,  so does  $\lim_{\l\to 0}{\rm{(II)}}_{h}^{\l}$ and the limits are equal.   It suffices to
We first  find a finite number  $D_{h}$ such
that for $\l\in[-\delta_1, \delta_1]$ (where
$\delta_1$ is from Lemma \ref{shadow lem})  and all $t\in \Bbb R_+$,\begin{equation}\label{sub-add-h-new}\big|\E_{\overline\P^{\l}}\left(d_{G_{{\bf v}}^{0}}({\om}(0),
{\om}(t))\right)-t\overline h_{\l, 0}\big|\leq
D_{h}.\end{equation}
Indeed, using again the fact that the  $\LL^{\l}$-diffusion has leafwise infinitesimal generator $\LL^{\l}_{\bf v}$ and $\P^{\l}$ is stationary,  we have
\begin{eqnarray*}
\E_{\overline\P^{\l}}\left(
-\ln k_{{\bf
v}}^{0}({\om}(t), \xi)\right)&=&\E_{\overline\P^{\l}}\left(-\int_{0}^{t}\frac{\pp}{\pp s}(\ln k_{{\bf
v}}^{0}({\om}(s), \xi))\ ds\right)\\
  &=&\E_{\overline\P^{\l}}\left(-\int_{0}^{t}\LL^{\l}_{\bf v}(\ln k_{{\bf
v}}^{0})({\om}(s))\ ds\right)\\
  &=&  -t \int_{\scriptscriptstyle{M_0\times \pp\M}}\LL^{\l}_{\bf v}(\ln k_{{\bf
v}}^{0})\ d\wt{{\bf m}}^{\l}\\
  &=& t\overline{h}_{\l, 0}.
\end{eqnarray*}
So (\ref{sub-add-h-new}) will be a simple consequence of the following lemma.

\begin{lem}\label{cov-bound-h}
  There exists a finite number  $\wt{D}_3$ such that for all $\l\in [-\delta_1,
\delta_1]$  and
$t\in \Bbb R_+$,
\[
\E_{\overline\P^{\l}}\left(\left|d_{G_{{\bf
v}}^{0}}({\om}(0), {\om}(t))+\ln k_{{\bf
v}}^{0}({\om}(t), \xi)\right|^2\right)<\wt{D}_3.
\]
\end{lem}
\begin{proof} For ${\bf v}=(x, \xi)\in S\M$,  $\om \in \Omega_+$ starting from ${\bf v}$, $t\geq 0$,  we continue to denote $\om$ its projection to $\M$.  Let  $z_t({\om})$ be the point on the geodesic ray $\g_{{\om}(t), \xi}$
closest to $x$.  We will divide $\Omega_+$ into four events
$A'_i(t)$, $1\leq i\leq 4$, and show there exists a finite
$\wt{D}'_3$ such that
\[
{\bf I}_i:=\E_{\overline\P^{\l}}\left(\left|d_{G_{{\bf
v}}^{0}}({\om}(0), {\om}(t))+\ln k_{{\bf
v}}^{0}({\om}(t), \xi)\right|^2\cdot {\bf
1}_{A'_i(t)}\right)\leq \wt{D}'_3.
\]

Let $A'_1(t)$ be the event that $d({\om}(0), {\om}(t))>1$ and
$d({\om}(0), z_t({\om}))\leq 1$. For $\om\in A'_1(t)$,   using
Harnack's inequality  (\ref{LY-harnack}) and  Lemma
\ref{triangle-along geodesic},  we easily specify the constant ratios involved in (\ref{Green-Matin-kernel-function}) and obtain ${\bf I}_1\leq (\ln
(c_2c_4^2c_7))^2$.

Let $A'_2(t)$ be the collection of $\om$ such that both
$d({\om}(0), {\om}(t))$ and $d({\om}(0), z_t({\om}))$
are greater than $1$ and $z_t({\om})\not={\om}(t)$. For such
$\om$, we first have by Lemma \ref{triangle-separated by cones} that
\begin{equation}\label{equ-1-lem-4.17}
\left|d_{G_{{\bf v}}^{0}}({\om}(0), {\om}(t))-d_{G_{{\bf
v}}^{0}}({\om}(0), z_t({\om}))-d_{G_{{\bf
v}}^{0}}({\om}(t), z_t({\om}))\right|\leq -\ln c_8.
\end{equation}
For  $d_{G_{{\bf
v}}^{0}}({\om}(t), z_t({\om}))$,  it is true by
Lemma \ref{triangle-along geodesic}  that
\[
\left|d_{G_{{\bf
v}}^{0}}({\om}(t), z_t({\om}))+\ln G_{\bf v}^0(y, {\om}(t))-\ln G_{\bf v}^0(y, z_t({\om}))\right|\leq -\ln c_7,
\]
where  $y$ is an arbitrary point on $\g_{z_t({\om}), \xi}$  far away from $z_t({\om})$.  Then we can use  Lemma
\ref{triangle-separated by cones}  to replace  $\ln G_{\bf v}^0(y, z_t({\om}))$ by $ \ln G_{\bf v}^0(y, {\om}(0))-\ln G_{\bf v}^0(z_t({\om}),  {\om}(0))$, which, by letting $y$ tend to $\xi$,  gives
\[
\left|d_{G_{{\bf v}}^{0}}({\om}(t), z_t({\om}))+\ln
k_{{\bf v}}^{0}({\om}(t), \xi)\right|\leq -\ln
(c_7c_8)+\left|\ln {G}_{{\bf v}}^{0}({\om}(0),
z_t({\om}))\right|.
\]
This, together with (\ref{equ-1-lem-4.17}), further  implies \begin{eqnarray*}
  \left|d_{G_{{\bf
v}}^{0}}({\om}(0), {\om}(t))+\ln k_{{\bf
v}}^{0}({\om}(t), \xi)\right|&\leq& -\ln (c_2
c_7c_8^2)+2\left|\ln {G}_{{\bf
v}}^{0}({\om}(0), z_t({\om}))\right|\\
&\leq& -\ln (c_2^5 c_7c_8^2)+2\alpha_2 d({\om}(0),
z_t({\om})).
\end{eqnarray*}
Since $\M$ is $\delta$-Gromov hyperbolic for some $\delta>0$,
it is true (cf. \cite[Proposition 2.1]{K2})  that
\[
d(x, \g_{y, z})\leq (y|z)_x+4\delta,\ \mbox{for any}\   x, y,
z\in \M.
\]
Consequently, we have \begin{equation}\label{delta-hyper}
d({\om}(0), z_t({\om}))\leq ({\om}(t)|\xi)_{{\om}(0)}+4\delta=
\frac{1}{2}\left|d({\om}(0), {\om}(t))-b_{{\bf
v}}({\om}(t))\right|+4\delta.
\end{equation}
Using  Lemma \ref{cov-bound-ell}, we finally obtain
\[
{\bf I}_2\leq 2\left(8\alpha_2\delta-\ln (c_2^5c_7c_8^2)\right)^2+2\alpha_2^2 D_3.
\]

Let $A'_3(t)$ be the collection of $\om$ such that $d({\om}(0),
{\om}(t))>1$ and $z_t({\om})={\om}(t)$. Let
$\g'_{{\om}(t), \xi}$ be the two sided extension of the geodesic
$\g_{{\om}(t), \xi}$ and let $z'_t({\om})\in \g'_{{\om}(t),
\xi}$ be the point closet to ${\om}(0)$. Then
$z'_t({\om})\preceq z_t({\om})$ on $\g'_{{\om}(t), \xi}$.
For $\om\in A'_3(t)$, using (\ref{LY-harnack})  if
$d(z'_t({\om}), {\om}(t))<1$ (or using Lemma
\ref{triangle-separated by cones}, otherwise),   we see that
\begin{eqnarray*}
d_{G_{{\bf
v}}^{0}}({\om}(0), {\om}(t))&\leq& d_{G_{{\bf
v}}^{0}}({\om}(0), z'_t({\om}))+d_{G_{{\bf
v}}^{0}}(z'_t({\om}), {\om}(t))-\ln (c_4c_8)\\
&\leq& \alpha_2 \left(d({\om}(0),
z'_t({\om}))+d(z'_t({\om}), {\om}(t))\right)-\ln (c_2^4 c_4c_8)\\
&\leq& 3\alpha_2 d({\om}(0), \g_{{\om}(t), \xi})-\ln (c_2^4 c_4c_8)\\
&\leq& \frac{3}{2}\alpha_2\left|d({\om}(0), {\om}(t))-b_{{\bf
v}}({\om}(t))\right|+12\alpha_2\delta-\ln (c_2^4 c_4c_8),
\end{eqnarray*}
where we use (\ref{delta-hyper}) to derive the last inequality.
Choose $y\in  \g_{{\om}(t), \xi}$ with $d({\om}(0), y)$ and
$d({\om}(t), y)$ are greater than 1.  Similarly,  using Lemma
\ref{triangle-separated by cones}, and then Lemma
\ref{triangle-along geodesic},  we have
\begin{eqnarray*}
  \left|\ln\frac{G_{{\bf
v}}^{0}({\om}(t), y)}{G_{{\bf v}}^{0}({\om}(0),
y)}\right|&=&\left|d_{G_{{\bf v}}^{0}}({\om}(0),
y)-d_{G_{{\bf v}}^{0}}({\om}(t), y)\right|\\
&\leq& -\ln c_8+ \left|d_{G_{{\bf v}}^{0}}({\om}(0),
z'_t({\om}))+d_{G_{{\bf v}}^{0}}(z'_t({\om}),
y)-d_{G_{{\bf v}}^{0}}({\om}(t), y)\right|\\
&\leq& -\ln (c_7c_8)+d_{G_{{\bf v}}^{0}}({\om}(0),
z'_t({\om}))+d_{G_{{\bf v}}^{0}}(z'_t({\om}),
{\om}(t))\\
&\leq&\frac{3}{2}\alpha_2\left|d({\om}(0), {\om}(t))-b_{{\bf
v}}({\om}(t))\right|+12\alpha_2\delta-\ln (c_2^4c_7c_8).
\end{eqnarray*}
Letting $y$ tend to $\xi$, we obtain
\[
\left|\ln k_{{\bf v}}^{0}({\om}(t), \xi)\right|\leq
\frac{3}{2}\alpha_2\left|d({\om}(0), {\om}(t))-b_{{\bf
v}}({\om}(t))\right|+12\alpha_2\delta-\ln (c_2^4c_7c_8).
\]
Thus, using Lemma \ref{cov-bound-ell} again, we obtain
\begin{eqnarray*}
  {\bf I}_3&\leq& \E_{\overline\P^{\l}}\left(\left(3\alpha_2\left|d({\om}(0), {\om}(t))-b_{{\bf
v}}({\om}(t))\right|+24\alpha_2\delta-\ln (c_2^8c_4c_7c_8^2)\right)^2\right)\\
&\leq& 18\alpha_2^2 D_3+2\left(24\alpha_2\delta-\ln
(c_2^8c_4c_7c_8^2)\right)^2.
\end{eqnarray*}

Finally, let $A'_4(t)$ be the event that $d({\om}(0),
{\om}(t))\leq 1$. Then ${\bf I}_4\leq (-\ln(c_2c_4))^2$ by the
classical Harnack inequality (\ref{LY-harnack}).
\end{proof}

As before, this reduces  the  proof of Theorem \ref{differential-h} to showing  \[\lim_{t\to
+\infty}{\rm{(III)}}_{h}^{t}=\lim_{t\to+\infty}(1/t)\E_{\overline\Q^0}({\bf
Z}_{h, t}^{0}{\bf M}_t^0),\] where  \begin{equation}\label{dfn-III-h-t} 
{\rm{(III)}}_{h}^{t}:= \E_{\overline\P^{\l}_t}\left(\frac{1}{\sqrt{t}}\left(d_{G_{{\bf
v}}^{0}}({\om}(0),
{\om}(t))-t \overline h_0\right)\right). \end{equation} The proof is completely parallel to the computation of  $ \lim_{t \to +\infty}\rm{(I)}_{\ell}^t$.   We  prove Proposition  \ref{Covariance-2} first.

\begin{proof}[Proof of Proposition \ref{Covariance-2}] Let $({\bf Z}_{t}^1)_{t\in \Bbb R_+}$,  $u_1$ be as  in Proposition \ref{M-0-M-1}.   The process $({\bf Z}_{t}^1)_{t\in \Bbb R_+}$ is a centered martingale with stationary increments and its law under $\overline\P^0$ is the same as the law of $(\overline{\bf Z}_{t}^{1})_{t\in \Bbb R_+}$ under
$\overline\Q^0$, where  $(\overline{\bf Z}_{t}^{1})_{t\in \Bbb R_+}$ on $(\overline\Theta, \overline\Q^0)$  is  given by %\footnote{I think I have clarified the state of $k_{\bf v}$.}
 \[
\overline{\bf Z}_{t}^{1}({\bf v}, \un\om)=\ln k_{{\bf
v}}\left({\bf y}_{{\bf v}, t}(\un\om), \xi\right)+t
\overline{h}_0+ u_1 \left({\bf y}_{{\bf v},
t}(\un\om)\right)-u_1 \left({\bf v}\right).
  \]

The pair $(-\overline{\bf Z}_{t}^1, {\bf
M}_{t}^{0})$ is  a centered martingale on $(\overline\Theta, \overline\Q^0)$ with
stationary increments and integrable  increasing process function.  As before, it follows that for $(a, b)\in \Bbb R^2$,  $-a\overline{\bf Z}_{t}^1/\sqrt{t}+b{\bf
M}_{t}^{0}/\sqrt{t}$  converge in distribution in $\overline\Q^0$ to a centered normal law  with variance $\Sigma_{h}[a,
b]=(a, b)\Sigma_{h}(a, b)^{T}$ for some matrix $\Sigma_{h}$.   Therefore, $(-\overline{\bf Z}_{t}^1/\sqrt{t}, {\bf
M}_{t}^{0}/\sqrt{t})$ converge in distribution to a centered normal vector with covariance $\Sigma_{h}$.  Since  both
$\overline{\bf Z}_{t}^1$ and ${\bf M}_{t}^0$ have stationary
increments, we also have
\[
\Sigma_{h}[a, b]=\frac{1}{t}\Bbb E_{\overline\Q^0}\left[(-a\overline{\bf Z}_t^1+b{\bf
M}_t^0)^2\right],\ \mbox{for all}\ t\in \Bbb
R_+.
\]
This shows Proposition \ref{Covariance-2} for the pair  $(-\overline{\bf
Z}_{t}^1/\sqrt{t}, {\bf M}_{t}^{0}/\sqrt{t} )$ instead of the pair  $({\bf Z}_{\ell, t}^1/\sqrt{t}, {\bf
M}_{t}^{0}/\sqrt{t}).$

Recall that for  $\overline\P^0$-a.e.  orbits  $\om\in \Om_+$ with $\om(0)=: {\bf v}$,
  ${\om}$,  the projection of $\om$ to $\M$, is such that
\[ \limsup\limits_{t\to +\infty}|\ln G_{{\bf v}}(x,
{\om}(t))-\ln k_{{\bf v}}({\om}(t), \xi)|<+\infty.
\]
We have by Lemma \ref{cov-bound-h} that
\[
\sup_{t} \E_{\overline\P^{\l}}(\big|{\bf Z}_{h, t}^0+\overline{\bf
Z}_{t}^1 \big|^2)<+\infty .
\]
Therefore,
\[
\ \ \ \ \ \ \ \ \ \E_{\overline\P^{\l}}(\frac{1}{t}\big|{\bf Z}_{h, t}^0+\overline{\bf
Z}_{t}^1 \big|^2)\to 0,\ {\rm{as}}\ t\to +\infty.
\]
Consequently,  $({\bf Z}_{h, t}^0/\sqrt{t}, {\bf
M}_{t}^{0}/\sqrt{t})$  has the same limit normal law  as  $(-\overline{\bf Z}_{t}^1/\sqrt{t},
{\bf M}_{t}^{0}/\sqrt{t})$ and its covariance matrix  under $\overline\Q^0$
converges to $\Sigma_{h}$ as $t$ goes to infinity.
\end{proof}

Finally, Theorem \ref{differential-h} follows from
\begin{lem} $\lim _{t \to +\infty } {\rm{(III)}}_{h}^{t} = \lim_{t\to+\infty}(1/t)\E_{\overline\Q^0}({\bf
Z}_{h, t}^{0}{\bf M}_t^0)$,
where  ${\rm{(III)}}_{h}^{t}$ is defined in  (\ref{dfn-III-h-t}). \end{lem}
\begin{proof}  %Since for $\overline\P^0$-a.e. $\om$,
%${\om}(t)$ tends to a  boundary point as $t$ goes to infinity,
%so  $\left(d_{G_{{\bf v}}^{0}}({\om}(0), {\om}(t))-t \overline
%h_0\right)/\sqrt{t}$ and   $-\left(\ln \bfg^0(\om(0), \om(t))+t
%\overline h_0\right)/\sqrt{t}$ have the same asymptotic
%distribution,  which is  a centered normal distribution by
%Proposition \ref{CLT}.  Thus,
%\begin{eqnarray*}
%&&\lim_{t\to +\infty}{\rm{(III)}}_{h}^{t}\\&&=
%\lim_{t\to +\infty}
%\E_{\overline\P^{\l}_t}\left(\frac{1}{\sqrt{t}}\left(d_{G_{{\bf
%v}}^{0}}({\om}(0), {\om}(t))-t \overline
%h_0\right)\right)-\lim_{t\to
%+\infty}\E_{\overline\P^{0}}\left(\frac{1}{\sqrt{t}}\left(d_{G_{{\bf
%v}}^{0}}({\om}(0),
%{\om}(t))-t \overline h_0\right)\right)\\
%&&=\lim_{t\to +\infty}
%\E_{\overline\P^{\l}_t}\left(\frac{1}{\sqrt{t}}\left(d_{G_{{\bf
%v}}^{0}}({\om}(0),
%{\om}(t))-t \overline h_0\right)\right)\\
%&&=: {\rm{(IV)}}_{h},
%\end{eqnarray*}
%providing the last limit exists. 
Let  ${\bf y}=({\bf y}_t)_{t\in
\Bbb R_+}=({\bf y}_{{\bf v}, t})_{{\bf v}\in SM, t\in \Bbb R_+}$  be
the diffusion process on $(\overline\Theta, \overline\Q^{\l})$
corresponding to $\LL^{\l}$ defined in  Section  \ref{4.1-Girsanov}.
Using the Girsanov-Cameron-Martin formula  for
$d\overline\P^{\l}_{{\bf v}, t}/d\overline\P^{0}_{{\bf v}, t}$ (see
(\ref{Gir-P})), we have
\begin{eqnarray*}
\lim_{t\to+ \infty }{\rm{(III)}}_{h}^t
 &=&\lim_{t\to +\infty} \E_{\overline\P^0}\left(\frac{1}{\sqrt{t}}\left(d_{G_{{\bf
v}}^{0}}({\om}(0),
{\om}(t))-t \overline h_0\right)\frac{d\overline\P_{\om(0), t}^{\l}}{d\overline\P_{\om(0), t}^{0}}\right)\\
& =&  \lim_{t\to +\infty}
\E_{\overline\Q^0}\left(\frac{1}{\sqrt{t}}\left(d_{G_{{\bf
v}}^{0}}({\bf y}_0({\bf v}, \un{\om}), {\bf y}_t({\bf v},
\un{\om}))-t \overline h_0\right)\cdot
\overline{\rm
M}_{t}^{\l}(\u{\om})\right)\\
&=& \lim_{t\to +\infty}
\E_{\overline\Q^0}\left(\frac{1}{\sqrt{t}} {\bf Z}_{h,
t}^0\cdot
\overline{\rm
M}_{t}^{\l}\right),\end{eqnarray*}
where we identify ${\bf y}_t({\bf v}, \un{\om}))\in \M\times
\{\xi\}$ with its projection point on $\M$. As before, by  Proposition  \ref{Covariance-2},  the variables  $\frac{1}{\sqrt{t}} {\bf
Z}_{h, t}^0\cdot \overline{\rm
M}_{t}^{\l}$  converge  in distribution to  ${\bf Z}_{h}^{0}e^{{\bf
M}^0-\frac{1}{2}\E_{\overline\Q^0}(({\bf M}^{0})^2)}$,  where $( {\bf Z}_{h}^{0}, {\bf M}^{0}) $ is a bivariate centered normal variable with covariance matrix $\Sigma _h .$

  Again, we have by Proposition \ref{Covariance-2} and the same reasoning as in the proof of Lemma \ref{ell-Y-2} that
 \[
\sup\limits_{t}\E_{\overline\Q^0}\left(\left|\frac{1}{\sqrt{t}} {\bf
Z}_{h, t}^0\cdot
\overline{\rm
M}_{t}^{\l}\right|^{\frac{3}{2}}\right)<+\infty.
\]
It follows from  Lemma \ref{convergence-lemma} that
\begin{eqnarray*}\label{pass to limit-h}
\lim_{t\to +\infty}{\rm{(III)}}_{h}^{t}=\lim_{t\to +\infty} \E_{\overline\Q^0}\left(\frac{1}{\sqrt{t}} {\bf
Z}_{h, t}^0\cdot\overline{\rm
M}_{t}^{\l}\right) =
\E_{\overline\Q^0}\left({\bf Z}_{h}^{0}e^{{\bf
M}^0-\frac{1}{2}\E_{\overline\Q^0}(({\bf M}^{0})^2)}\right).
\end{eqnarray*}
Finally, using the fact that  $({\bf Z}_{h}^0,{\bf M}^{0})$ has a bivariate normal
distribution,  we have again \[ \E_{\overline\Q^0}\left({\bf
Z}_{h}^{0}e^{{\bf M}^0-\frac{1}{2}\E_{\overline\Q^0}(({\bf
M}^{0})^2)}\right)=\E_{\overline\Q^0}({\bf Z}_{h}^{0}{\bf M}^0),
\]
which is  $\lim_{t\to+\infty}(1/t)\E_{\overline\Q^0}({\bf Z}_{h,
t}^{0}{\bf M}_t^0)$ by  Proposition \ref{Covariance-2}.
\end{proof}

\section{Infinitesimal Morse correspondence}

In this section, we study the limit (\ref{Intro-X-lambda}) and give an expression for the derivative of the geodesic spray when the metric varies in $\Re(M)$. 

Let $(M, g)$ be a negatively curved closed connected $m$-dimensional
Riemannian manifold as before. Let $\pp\M$ be the geometric boundary
of the universal cover space $(\wt{M}, \wt{g})$.  We can identify
$\wt{M}\times \pp\M$ with $S\M_{\wt{g}}$, the unit tangent bundle of
$\M$ in metric $\wt{g}$, by sending $(x, \xi)$  to  the unit tangent vector of the
$\wt{g}$-geodesic starting at $x$ pointing at $\xi$.

Let $\lambda\in (-1, 1) \mapsto g^{\lambda}$ be a one-parameter
family of $C^3$ metrics on $M$ of negative curvature with $g^0=g$.
Denote by $\wt{g}^{\l}$ the $G$-invariant extension of $g^{\l}$ to $\M$. For each $\l$,
the geometric boundary of $(\M, \wt{g}^{\l})$, denoted
$\pp\M_{\wt{g}^{\l}}$, can be identified with $\pp\M$ since the
identity isomorphism from $G=\pi_1(M)$ to itself induces a
homeomorphism between $\pp\M_{\wt{g}^{\l}}$ and $\pp\M$. So
each $(x, \xi)\in \M\times \pp\M$  is also associated with  the $\wt{g}^{\l}$-geodesic spray $\overline{X}_{\wt{g}^{\l}}(x, \xi)$,  the horizontal vector in $TT\M$ which projects to 
the unit tangent vector of the $\wt{g}^{\l}$-geodesic starting at
$x$ pointing towards $\xi$.  Our  very  first step to study the
differentiability of the  linear drift under a one-parameter family of conformal changes $g^{\l}$ of $g$  is to understand the differentiable dependence of  the geodesic sprays 
$\overline{X}_{\wt{g}^{\l}}(x, \xi)$ on  the parameter $\l$.

For each $g^{\l}$, there exist  \emph{$(g, g^{\l})$-Morse correspondence} (\cite{A,Gr,Mor}), the  homeomorphisms from $SM_{g}$ to
$SM_{g^{\l}}$  sending a $g$ geodesic on $M$ to a $g^\l $ geodesic on $M$.  The
$(g, g^{\l})$-Morse correspondence is not unique, but any two such
maps only differ by shifts in the geodesic flow directions
(i.e., if $F_1, F_2$ are two $(g, g^{\l})$-Morse correspondence
maps,  then there exists a real valued function $t(\cdot)$ on $SM_{g}$
such  that  $F_1^{-1}\circ F_2(v)={
\Phi}_{t(v)}(v)$ for $v\in SM_{g}$), where ${\Phi}$ is the geodesic flow map on
$SM_g$ (\cite{A,Gr,Mor}, see \cite[Theorem 1.1]{FF}).

Let us construct a $(g, g^{\l})$-Morse correspondence map  by lifting the
systems to their universal cover spaces as in \cite{Gr}.  For an
oriented geodesic $\gamma$ in $(\M, \wt{g})$, denote by
$\pp^+(\gamma)\in \pp\M_{\wt{g}}$ and $\pp^-(\gamma)\in
\pp\M_{\wt{g}}$ the asymptotic classes of its positive and negative
directions. The map $\gamma\mapsto (\pp^+(\gamma), \pp^-(\gamma))\in
\pp\M_{\wt{g}}\times \pp\M_{\wt{g}}$ establishes a homeomorphism
between the set of all oriented geodesics in $(\M, \wt{g})$ and
$\pp^2(\M_{\wt{g}})=(\pp\M_{\wt{g}}\times \pp\M_{\wt{g}})\backslash
\{(\xi, \xi):\ \xi\in \pp \M_{\wt{g}}\}$. So the natural
homeomorphism ${\rm D}^{\l}:\ \pp^2(\M_{\wt{g}})\to
\pp^2(\M_{\wt{g}^{\l}})$ induced from the identity isomorphism from
$G$ to itself can be viewed as a homeomorphism between the sets of
oriented geodesics in $(\M, {\wt{g}})$ and $(\M, {\wt{g}^{\l}})$.
Realize points from $S\M_{\wt{g}}$ by pairs $(\gamma, y)$, where
$\gamma$ is an oriented geodesic and $y\in \gamma$, and define a map
$\wt{F}^{\l}:\ S\M_{\wt{g}}\to S\M_{\wt{g}^{\l}}$  by sending
$(\gamma, y)\in S\M_{\wt{g}}$ to
\[
\wt{F}^{\l}(\gamma, y)=({\rm D}^{\l}(\gamma), y'),
\]
where $y'$ is  the intersection point of  ${\rm D}^{\l}(\gamma)$ and the  hypersurface $\{\exp_{\wt{g}} Y:\
Y\bot {\rm v}\}$, where ${\rm v}$   is the vector in $S_y\M_{\wt{g}}$  pointing at  $\pp^+(\gamma)$. The map $\wt{F}^{\l}$ is a homeomorphism since both
$g$ and $g^{\l}$ are of negative curvature. Returning to $SM_{g}$
and $SM_{g^{\l}}$, we obtain a  map
$F^{\l}$.  Given  any sufficiently small $\epsilon$, if $g^{\l}$ is in a
sufficiently small $C^3$-neighborhood of $g$, then $F^{\l}$ is the
only $(g, g^{\l})$-Morse correspondence map such that the footpoint of
$F^{\l}({v})$ belongs to the hypersurface of points  $\{\exp_g Y:\
Y\bot {v}, \|Y\|_{g}<\epsilon\}$.

Regard $SM_{g^{\l}}$ as a subset of $TM$ and let $\pi^{\l}: \
SM_{g^{\l}}\to SM_g$ be the projection map sending $v$ to
$v/\|v\|_g$. The map $\pi^{\l}$ records  the direction information of
the vectors of $SM_{g^{\l}}$ in $SM_g$. Let $F^{\l}: SM_g\to
SM_{g^{\l}}$ be the $(g, g^{\l})$-Morse correspondence map obtained
as above. We obtain a  one-parameter family of homeomorphisms
$\pi^{\l}\circ F^{\l}$ from $SM_g$ to $SM_g$. By using the implicit
function theorem, de la Llave-Marco-Moriy\'{o}n \cite[Theorem A.1]{LMM}
improved the differentiable dependence of $\pi^{\l}\circ F^{\l}$ on
the parameter $\l$.

\begin{theo}\label{FF-thm 2.1}(cf. \cite[Theorem 2.1]{FF})
  There exists a $C^3$ neighborhood of $g$ so that for any $C^3$ one-parameter
family of $C^3$ metrics  $\lambda\in (-1, 1)\mapsto g^{\lambda}$ in
it with $g^0=g$, the map $\l\mapsto \pi^{\l}\circ F^{\l}$ is $C^3$
with values in the Banach manifold of continuous maps $SM_g\to
SM_g$. The tangent to the curve $\pi^{\l}\circ F^{\l}$ is a  continuous
vector field $\Xi_{\l}$ on $SM_g$. \end{theo}

 Following  Fathi-Flaminio
\cite{FF}, we will call  $\Xi:=\Xi_0$ in Theorem \ref{FF-thm 2.1}
the \emph{infinitesimal Morse correspondence at $g$ for the curve
$g^{\l}$}. It was shown in \cite{FF} that the vector field $\Xi$ only
depends on $g$ and the differential of $g^{\l}$ in $\l$ at $0$. More precisely, the horizontal and the vertical components of $\Xi$ are described by:

\begin{theo}(\cite[Proposition
2.7]{FF})\label{FF-Morse-correspondence}
  Let $\Xi$ be the infinitesimal Morse correspondence at $g$ for the
  curve $g^{\l}$ and let $\Xi_{\g}$ be the restriction of  horizontal component  of $\Xi$ to  a unit speed $g$-geodesic $\g$. Then $\Xi_{\g}$ is the
  unique bounded solution of the equation
  \begin{equation}\label{Morse-correspondence}
\nabla_{\dot{\g}}^2\Xi_{\g}+ {\rm R}(\Xi_{\g},
\dot{\g})\dot{\g}+\Gamma_{\dot{\g}}\dot{\g}-\langle
\Gamma_{\dot{\g}}\dot{\g}, \dot{\g}\rangle \dot{\g}=0
  \end{equation}
  satisfying $\langle\Xi_{\g}, \dot{\g}\rangle=0$ along $\g$, where $\dot\g(t)=\frac{d}{dt}\g(t)$, $\nabla$ and ${\rm R}$ are the Levi-Civita connection and curvature tensor of metric $g$,
   $\nabla^{\l}$ is the Levi-Civita connection of the metric $g^{\l}$ and $\Gamma=\partial_{\l}\nabla^{\l}|_{\l=0}$.  The vertical
  component of $\Xi$ in $T(SM_g)$ is given by
  $\nabla_{\dot{\g}}\Xi_{\g}$.
\end{theo}

We will still denote by  $\Xi$ the $G$-invariant extension to $T(S\M_{\widetilde{g}})$ of
the infinitesimal Morse-correspondence at $g$ for the curve
$g^{\l}$.
For any geodesic $\gamma$ in  $(\widetilde{M},\widetilde{g})$, let $N(\g) $ be the normal bundle of $\g$: $$ N(\g )= \cup _{t \in \R}N_t (\g),  {\textrm { where }} N_t(\g) \; = \; ( \dot \g (t))^\perp \; = \; \{ E \in T_{\g(t)} \widetilde{M}: \< E, \dot \g (t)\> = 0 \}.$$
The one-parameter family of vectors along $\g$ arising in equation (\ref{Morse-correspondence})
\begin{equation}\label{Upsilon}
\Upsilon (t):=\left(\Gamma_{\dot{\g}}\dot{\g}-\langle
\Gamma_{\dot{\g}}\dot{\g}, \dot{\g}\rangle \dot{\g}\right)(\g(t)), \ t\in \Bbb R,
\end{equation}
is such that $\Upsilon (t)$ belongs to $N_t(\g)$ for all $t$. The restriction of the infinitesimal Morse correspondence  to $\gamma$ is $(\Xi_{\g}, \nabla_{\dot{\g}}\Xi_{\g})$,
with both $\Xi_{\g}$ and  $\nabla_{\dot{\g}}\Xi_{\g}$  belonging to $N(\g)$ as well.  In the following, we  will  specify $\Xi_{\g}$ and  $\nabla_{\dot{\g}}\Xi_{\g}$ using  $\Upsilon$ and a special coordinate system of  $N_t(\g)$'s arising from the stable and unstable Jacobi  fields along $\g$.

Let  ${\bf v} =(x, {\rm v})$  be a point in $T\M$. Recall from Subsection 2.3 the definition (\ref{Jacobi  field-definition}) of Jacobi Fields, Jacobi tensors and, for $\bfv \in S\M$, of the stable and unstable tensors along $\g _\bfv$ denoted $S_\bfv $ and $U_\bfv $. 
 For each ${\bf{v}}\in S\M$, the vectors $(Y, S'_{\bf v}(0)Y)$, $Y\in N_0(\g)$, (or  $(Y, U'_{\bf v}(0)Y)$)   generate $TW_{\bf v}^{ss}$ (or $TW_{\bf v}^{su}$). As a consequence of the Anosov property of the geodesic flow on $S\M$,  the operator $(U'_{{\bf v}}(0)-S'_{{\bf v}}(0))$ is positive and symmetric (see \cite {Bo}). %\lin{(this property is first obtained by  Eberlein \cite{Eb} for the compact manifold and is extended by Bolton \cite{Bo} to  a manifold without conjugate points whose geodesic flows are Anosov).}\footnote{Yes, I think it is \cite{Bo}. A nice paper. It is not in \cite{Bo}; but I agree, it is  a nice paper, I did not know it. }\footnote{Do you mean P. Eberlein, When is a geodesic flow of Anosov type? He didn't relate the positivity of  $(U'_{{\bf v}}(0)-S'_{{\bf v}}(0))$ to Anosov property. Actually, I can't find it written in \cite{Eb}. But the properties that are equivalent to  Anosov flow were shown in \cite{Bo} to be equivalent to the positivity of the $(U'_{{\bf v}}(0)-S'_{{\bf v}}(0))$. Anyway, you can refer to Knieper, New results on noncompact harmonic manifolds page 10.  There we mentioned it the way above.}
  Hence we can choose vectors $\vec{x}_1, \cdots, \vec{x}_{m-1}$ to form a basis of $N_0(\g_{{\bf v}})$ so that
\begin{equation}\label{U-S-delta}
\langle  (U'_{{\bf v}}(0)-S'_{{\bf v}}(0))\vec{x}_i, \vec{x}_j\rangle=\delta_{ij}.
\end{equation}
Let $J_1, \cdots, J_{2m-2}$ be the Jacobi  fields with
\[	
(J_i(0), J'_i(0))=\left\{
  \begin{array}{ll}
(\vec{x}_i,  S'_{{\bf v}}(0)\vec{x}_i), & \hbox{if}\ i\in \{1, \cdots, m-1\}; \\
   (\vec{x}_{i+1-m},  U'_{{\bf v}}(0)\vec{x}_{i+1-m}), & \hbox{if}\ i\in \{m, \cdots, 2m-2\}.
     \end{array}
\right.
\]
Since  the Wronskian of two Jacobi  fields remains constant along  geodesics, we have
\begin{equation}\label{Wronskian}
W(J_i, J_j)=\left\{
  \begin{array}{ll}
0, & \hbox{if}\ i, j\in \{1, \cdots, m-1\}\ \hbox{or}\  i, j\in \{m, \cdots, 2m-2\}; \\
 -\delta_{i, j+1-m}, & \hbox{if}\ i\in \{1, \cdots, m-1\} \ \ \hbox{and}\  j\in\{m, \cdots, 2m-2\}.
     \end{array}
\right.
\end{equation}
Equivalently,  if we write  ${\bf J}_{\rm s}$ for the matrix with column vectors $(J_1, \cdots, J_{m-1})$  and  ${\bf J}_{\rm u}$  for the matrix with column vectors $(J_m, \cdots, J_{2m-2})$,  then  (\ref{Wronskian}) gives
\begin{equation}\label{Wronskian-Matrix form}
{\bf J}^{*}_{w}{\bf J}'_{w}=({\bf J}'_w)^*{\bf J}_w,\ w={\rm s}\ \hbox{or}\  {\rm u},\ \hbox{and}\  \ {\bf J}^*_{\rm u}{\bf J}'_{\rm s}-({\bf J}'_{\rm u})^* {\bf J}_{\rm s}=-\mbox{Id}.
\end{equation}

The collection $(J_i(t), J_i'(t)), i = 1, \cdots, m-1$,  (or  $(J_i(t), J_i'(t)), i = m, \cdots, 2m-2$)   generate $TW_{\dot \g_{\bfv }(t)}^{ss}$ (or $TW_{\dot \g_{\bfv }(t)}^{su}$).
%The collection  $J_1(t), \cdots, J_{2m-2}(t)$  provides a basis for each $N_{t}(\g)\times N_t(\g)$.
 Consequently,   any $V(t)=(V_1(t), V_2(t))\in TT\M$ along $\g$ with $V_i(t)\in N_t(\g)$, $i=1, 2$, can be expressed as
 \[ V_1(t) = \sum _{i=1 }^{m-1} a_i(t) (J_i(t), J_i'(t)), \quad V_2(t) =  \sum _{i=m }^{2m-2} b_{i-m+1}(t) (J_i(t), J_i'(t)) ,\]
 where $(a_i(t),  b_i(t)),  i = 1 \cdots m-1, $ are $2m-2$ real numbers. Writing them as two column vectors $\vec a(t) , \vec  b(t) $, we write any such  $V(t) $ as \[ V(t) = ({\bf J}_{\rm s}(t)\vec{a}(t), {\bf J}'_{\rm s}(t)\vec{a}(t))+({\bf J}_{\rm u}(t)\vec{b}(t), {\bf J}'_{\rm u}(t)\vec{b}(t)).\] %with $\vec{a}(t), \vec{b}(t)$ being two $\Bbb R^{m-1}$ vector variables in $t$.
 To specify the infinitesimal Morse correspondence $\Xi$  at $g$ for the
  curve $g^{\l}$, it suffices to find the coefficients $\vec{a}(t), \vec{b}(t)$ for the restriction of $\Xi$  along any $\widetilde{g}$-geodesic $\g$.

\begin{prop}\label{a-b-prop}  Let $\Xi$ be the infinitesimal Morse correspondence at $g$ for  a $C^3$ one-parameter
family of $C^3$ metrics  $g^{\l}$ with $g^{0}=g$.  Then the restriction of $\Xi$ to a $\widetilde{g}$-geodesic $\g$ is  $({\bf J}_{\rm s}(t)\vec{a}(t), {\bf J}'_{\rm s}(t)\vec{a}(t))+({\bf J}_{\rm u}(t)\vec{b}(t), {\bf J}'_{\rm u}(t)\vec{b}(t))$ with
\begin{equation}\label{a-b-t-formula}
\vec{a}(t)=\int_{-\infty}^{t} {\bf J}^*_{\rm u}(s)\Upsilon(s)\ ds,\ \ \vec{b}(t)=\int^{+\infty}_{t} {\bf J}^*_{\rm s}(s)\Upsilon(s)\ ds,
\end{equation}
where $\Upsilon (s) $ is given by (\ref{Upsilon}).
\end{prop}
\begin{proof} By the construction of Morse correspondence,  for any $\widetilde{g}$-geodesic $\g$, the value of $\Xi$ along $\g$, denoted $\Xi(\g)$,  belongs to $N(\g)\times N(\g)$. So,  there are  $\vec{a}(t), \vec{b}(t), t\in \R$, such that
\begin{equation*}\label{a-b-Wron-1}
\Xi(\g)=({\bf J}_{\rm s}(t)\vec{a}(t), {\bf J}'_{\rm s}(t)\vec{a}(t))+({\bf J}_{\rm u}(t)\vec{b}(t), {\bf J}'_{\rm u}(t)\vec{b}(t)).\end{equation*}

 The horizontal part $\Xi_\g$ of  $\Xi(\g)$ is  ${\bf J}_{\rm s}(t)\vec{a}(t)+ {\bf J}_{\rm u}(t)\vec{b}(t)$. On the other hand, the vertical part of  $\Xi(\g)$ is  ${\bf J}'_{\rm s}(t)\vec{a}(t)+ {\bf J}'_{\rm u}(t)\vec{b}(t)$, which, by  Theorem \ref{FF-Morse-correspondence}, is also
\[
\nabla_{\dot{\g}}\Xi_{\g}={\bf J}'_{\rm s}(t)\vec{a}(t)+ {\bf J}'_{\rm u}(t)\vec{b}(t)+{\bf J}_{\rm s}(t)\vec{a}'(t)+ {\bf J}_{\rm u}(t)\vec{b}'(t).
\]
So we must have
\begin{equation}\label{a-b-Wron-1}{\bf J}_{\rm s}(t)\vec{a}'(t)+ {\bf J}_{\rm u}(t)\vec{b}'(t)=0.\end{equation}
Differentiating $\nabla_{\dot{\g}}\Xi_{\g}={\bf J}'_{\rm s}(t)\vec{a}(t)+ {\bf J}'_{\rm u}(t)\vec{b}(t)$ along $\g$ and reporting it in (\ref{Morse-correspondence}), we obtain
\[
{\bf J}'_{\rm s}(t)\vec{a}'(t)+ {\bf J}'_{\rm u}(t)\vec{b}'(t)+{\bf J}''_{\rm s}(t)\vec{a}(t)+ {\bf J}''_{\rm u}(t)\vec{b}(t)+{\rm R}(t){\bf J}_{\rm s}(t)\vec{a}(t)+ {\rm R}(t){\bf J}_{\rm u}(t)\vec{b}(t)=-\Upsilon (t),
\]
which simplifies to
\begin{equation}\label{a-b-Wron-2}
{\bf J}'_{\rm s}(t)\vec{a}'(t)+ {\bf J}'_{\rm u}(t)\vec{b}'(t)=-\Upsilon(t)
\end{equation}
by the defining property of Jacobi  fields. Using (\ref{Wronskian-Matrix form}),  we solve $\vec{a}', \vec{b}'$ from (\ref{a-b-Wron-1}),   (\ref{a-b-Wron-2}) with
\begin{equation}\label{a'-b'}
\vec{a}'={\bf J}_{\rm u}^{*}\Upsilon, \ \vec{b}'=-{\bf J}_{\rm s}^{*}\Upsilon.
\end{equation}
Note that ${\bf J}_{\rm u}(-\infty)={\bf J}_{\rm s}(+\infty)=0$.  Finally, we  recover $\vec{a}(t), \vec{b}(t)$ from (\ref{a'-b'}) by integration.
\end{proof}

  For any $s\in \Bbb R$, let $({\rm{K}}_{s}, {\rm{K}}_{s}')$  be the unique Jacobi  field along a $\widetilde{g}$-geodesic $\g$ such that \[
{{\rm K}}_{s}'(s)=\Upsilon (s)\ \ \mbox{and}\ \  {{\rm K}}_{s}(s)=0.
\]
Then
\[
\left({{\rm K}}_s(0), {{\rm K}}'_s(0)\right)=(D{ \bf \Phi}_{s})^{-1}\left(0,   \Upsilon(s)\right).
\]
We further express $\Xi$ using ${\rm{K}}_{s}$'s  by specifying  the value of $\Xi(\g(0))$ for any $\widetilde{g}$-geodesic $\g$.

\begin{prop}\label{a-b-better in K} Let $\Xi$ be the infinitesimal Morse correspondence at $g$ for  a $C^3$ one-parameter
family of $C^3$ metrics  $g^{\l}$ with $g^{0}=g$.  Then for the $\widetilde{g}$-geodesic $\g$ with $\dot\g(0)={\bf v}$:
\begin{eqnarray*}
\Xi_{\g}(0)&=& (U'_{\bf v}(0)-S'_{\bf v}(0))^{-1}\big[\int^{0}_{-\infty}({\rm K}'_{s}(0)-U'_{\bf v}(0){\rm K}_{s}(0))\ ds\\
&&\ \ \ \ \ \ \ \ \ \  \ \ \ \ \ \ \ \ \ \ \ \ \  \ +\int_{0}^{+\infty}({\rm K}'_{s}(0)-S'_{\bf v}(0){\rm K}_{s}(0))\ ds \big],\\
(\nabla_{\dot{\g}}\Xi_{\g}) (0)&=& S'_{\bf v}(0) (U'_{\bf v}(0)-S'_{\bf v}(0))^{-1}\int^{0}_{-\infty}({\rm K}'_{s}(0)-U'_{\bf v}(0){\rm K}_{s}(0))\ ds\\
&&+U'_{\bf v}(0) (U'_{\bf v}(0)-S'_{\bf v}(0))^{-1}\int_{0}^{+\infty}({\rm K}'_{s}(0)-S'_{\bf v}(0){\rm K}_{s}(0))\ ds.
\end{eqnarray*}
\end{prop}
\begin{proof} By Proposition \ref{a-b-prop},  for any $\widetilde{g}$-geodesic $\g$,
\[\Xi(\g(0))=\left(\Xi_{\g}(0), (\nabla_{\dot{\g}}\Xi_{\g}) (0)\right)=\left({\bf J}_{\rm s}(0)\vec{a}(0)+{\bf J}_{\rm u}(0)\vec{b}(0), {\bf J}'_{\rm s}(0)\vec{a}(0)+{\bf J}'_{\rm u}(0)\vec{b}(0)\right),\]
where  $\vec{a}(0), \vec{b}(0)$ are given by (\ref{a-b-t-formula}). We first express $\vec{a}(0)$ using  ${\rm{K}}_{s}$'s.  Let $s\leq 0$.  The Wronskian between ${\rm K}_s$ and  any unstable Jacobi fields are preserved along the geodesics and  must have the same value at $\g(s)$ and $\g(0)$.  This gives
\[
{\bf J}^*_{\rm u}(s) \Upsilon (s)= {\bf J}^*_{\rm u}(0) {\rm K}'_s(0)-({\bf J}_{\rm u}')^*(0){\rm K}_s(0).
\]
Consequently,
\begin{eqnarray*}
({\bf J}^*_{\rm u})^{-1}(0){\bf J}^*_{\rm u}(s) \Upsilon (s)={\rm K}'_s(0)-({\bf J}^*_{\rm u})^{-1}(0)({\bf J}_{\rm u}')^*(0){\rm K}_s(0)= {\rm K}'_s(0)-U'_{\bf v}(0){\rm K}_s(0),
\end{eqnarray*}
where we use the fact that ${\bf J}_{\rm u}'(0)=U'_{\bf v}(0){\bf J}_{\rm u}(0)$ for the second equality.  So we have
\[
\vec{a}(0)={\bf J}^*_{\rm u}(0)\int^{0}_{-\infty}({\rm K}'_{s}(0)-U'_{\bf v}(0){\rm K}_{s}(0))\ ds.
\]
Similarly, for any $s\geq 0$, a comparison of  the Wronskian between ${\rm K}_s$ and any stable Jacobi fields at time $s$ and $0$ gives
\[
{\bf J}^*_{\rm s}(s) \Upsilon (s)= {\bf J}^*_{\rm s}(0) {\rm K}'_s(0)-({\bf J}_{\rm s}')^*(0){\rm K}_s(0).
\]
As a consequence, we have
\begin{eqnarray*}
({\bf J}^*_{\rm s})^{-1}(0){\bf J}^*_{\rm s}(s) \Upsilon (s)={\rm K}'_s(0)-({\bf J}^*_{\rm s})^{-1}(0)({\bf J}_{\rm s}')^*(0){\rm K}_s(0)= {\rm K}'_s(0)-S'_{\bf v}(0){\rm K}_s(0),
\end{eqnarray*}
which gives
\[
\vec{b}(0)={\bf J}^*_{\rm s}(0)\int_{0}^{+\infty}({\rm K}'_{s}(0)-S'_{\bf v}(0){\rm K}_{s}(0))\ ds.
\]
The formula for $\Xi(\g(0))$  follows by using ${\bf J}_{\rm s}(0)={\bf J}_{\rm u}(0)$ and ${\bf J}_{\rm u}(0){\bf J}^*_{\rm u}(0)=(U'_{\bf v}(0)-S'_{\bf v}(0))^{-1}$.
\end{proof}

A dynamical point of view of  the integrability of the integrals  in Proposition  \ref{a-b-better in K} is that $({\rm K}'_{s}(0)-U'_{\bf v}(0){\rm K}_{s}(0))$ $(s\leq 0)$ is the stable  vertical part of $(D{\bf \Phi}_{s})^{-1}\left(0,   \Upsilon(s)\right)$   and hence decays exponentially when $s$ goes to $-\infty$,  while $({\rm K}'_{s}(0)-S'_{\bf v}(0){\rm K}_{s}(0))$ $(s\geq 0)$ is the unstable  vertical part of $(D{\bf \Phi}_{s})^{-1}\left(0,   \Upsilon(s)\right)$ and thus decays exponentially when $s$ goes to $+\infty$.

For any curve $\l\in (-1, 1)\mapsto \mathcal{C}_{\l}\in {\bf N}$ (or  $\mathcal{C}^{\l}\in {\bf N}$)  on 
some Riemannian manifold ${\bf N}$, we write
$(\mathcal{C}_{\l})'_0:=(d\mathcal{C}_{\l}/d\l)|_{\l=0}$ (or $(\mathcal{C}^{\l})'_0:=(d\mathcal{C}^{\l}/d\l)|_{\l=0}$) whenever the
differential exists.   We can put a formula concerning $\left(\overline{X}_{\wt{g}^{\l}}\right)'_0$ for any $C^3$  curve $g^{\l}$ in $\Re(M)$ with $g^0=g$.

\begin{prop}\label{lem-W} Let $(M, g)$ be a negatively curved closed connected $m$-dimensional
Riemannian manifold.   Then  for any $C^3$ one-parameter family of $C^3$ metrics $\lambda\in
(-1, 1)\mapsto g^{\lambda}$ in it with $g^0=g$, the map $\l\mapsto
\overline{X}_{\wt{g}^{\l}}(x, \xi)$ is differentiable at $\l=0$ for each ${\bf
v}=(x, \xi)$ with
\[
\left(\overline{X}_{\wt{g}^{\l}}\right)'_0(x, \xi)=\left(0,
\left(\|\overline{X}_{\wt{g}^{\l}}\|_{\wt{g}}\right)'_0({\bf v}){\bf
v}+\int_{0}^{+\infty}({\rm K}'_{s}(0)-S'_{\bf v}(0){\rm K}_{s}(0))\ ds\right).
\]
\end{prop}
\begin{proof} Express the homeomorphism $\wt{F}^{\l}$ as a map from $\M\times
\pp\M$ to $\M\times \pp\M_{\wt{g}^{\l}}$ with
\begin{equation*}
\wt{F}^{\l}(x, \xi)=(f^{\l}_{\xi}(x), \xi), \ \forall (x, \xi)\in
S\M,
\end{equation*}
where  $f^{\l}_{\xi}$ records the change of footpoint of the
$(g, g^{\l})$-Morse correspondence $\wt{F}^{\l}$.  We have
\begin{eqnarray*}
 && \frac{1}{\l}\left(\overline{X}_{\wt{g}^{\l}}(x,
\xi)-\overline{X}_{\wt{g}}(x,
\xi)\right)\\
&&=\frac{1}{\l}\left(\overline{X}_{\wt{g}^{\l}}(x,
\xi)-\frac{\overline{X}_{\wt{g}^{\l}}(x,
\xi)}{\|\overline{X}_{\wt{g}^{\l}}(x,
\xi)\|_{\wt{g}}}\right)+\frac{1}{\l}\left(\frac{\overline{X}_{\wt{g}^{\l}}(x,
\xi)}{\|\overline{X}_{\wt{g}^{\l}}(x,
\xi)\|_{\wt{g}}}-\overline{X}_{\wt{g}}(x,
\xi)\right)\\
&&=:(a)_{\l}+(b)_{\l}.
\end{eqnarray*}
When $\l$ tends to zero,  $(a)_{\l}$ tends to $(0,
\left(\|\overline{X}_{\wt{g}^{\l}}\|\right)'_0({\bf v}){\bf v})$.  For $(b)_{\l}$, we can transport $\overline{X}_{\wt{g}}(x,
\xi)$ to $\displaystyle \frac{\overline{X}_{\wt{g}^{\l}}(x,
\xi)}{\|\overline{X}_{\wt{g}^{\l}}(x,
\xi)\|_{\wt{g}}}$ along two pieces of curves: the first is to follow the footpoint of the  inverse of the $(g^{\l},g)$-Morse correspondence  from $\overline{X}_{\wt{g}}(x,
\xi)$ to $\overline{X}_{\wt{g}}((f_{\xi}^{\l})^{-1}(x),
\xi)$ with the constraint that the vector remains  within $TW^{s} (x,\xi )$;   the second is to use the  $(g^{\l}, g)$-Morse correspondence  from $\overline{X}_{\wt{g}}((f_{\xi}^{\l})^{-1}(x),
\xi)$ to $\displaystyle \frac{\overline{X}_{\wt{g}^{\l}}(x,
\xi)}{\|\overline{X}_{\wt{g}^{\l}}(x,
\xi)\|_{\wt{g}}}.$   By Theorem \ref{FF-thm 2.1} and Theorem
\ref{FF-Morse-correspondence},
  the second curve   is $C^1$ and the derivative is 
$\left(\Xi_{\g_{{\bf v}}}(0), \nabla_{\dot{\g}_{{\bf v}}}\Xi_{\g_{{\bf
v}}}(0)\right)$, which is  also $({\bf J}_{\rm s}(0)\vec{a}(0), {\bf J}'_{\rm s}(0)\vec{a}(0))+({\bf J}_{\rm u}(0)\vec{b}(0), {\bf J}'_{\rm u}(0)\vec{b}(0))$ with $\vec{a}(0), \vec{b}(0)$ from Proposition \ref{a-b-prop}.  The horizontal projection of the first  curve is the reverse of the second one; so it is also $C^1$ and the horizontal part of the derivative is $-\Xi_{\g_{{\bf v}}}(0)$.  Since it belongs to  $TW^{s } (x,\xi ) $   which is a graph over the horizontal plane, the vertical part is also $C^1$ and the derivative is given by $ S'_{{\bf v}}(0)(-\Xi_{\g_{{\bf
v}}}(0)).$  So,
\[
\lim\limits_{\l\to 0}(b)_{\l}=\left( 0, (\nabla_{\dot{\g}_{{\bf v}}}\Xi_{\g_{{\bf
v}}})(0)- S'_{{\bf v}}(0)\Xi_{\g_{{\bf
v}}}(0) \right)= \left(0, (U'_{{\bf v}}(0)- S'_{{\bf v}}(0)){\bf J}_{\rm u}(0)\vec{b}(0) \right),
\]
which, by  our choice  of  ${\bf J}_{\rm u}(0)={\bf J}_{\rm s}(0)$ and the defining property of ${\bf J}_{\rm u}(0)$ in (\ref{U-S-delta}),  is
\[
\left(0, ({\bf J}_{\rm s}^{*})^{-1}(0)\vec{b}(0)\right) =\left(0, \int_{0}^{+\infty}({\rm K}'_{s}(0)-S'_{\bf v}(0){\rm K}_{s}(0))\ ds \right) .
\]
\end{proof}
\begin{cor}\label{X-lam-Lem}Let $(M, g)$ be a negatively curved closed connected
Riemannian manifold and let  $\lambda\in (-1, 1)\mapsto g^{\lambda}\in \Re(M)$ be a $C^3$
curve of $C^3$ conformal changes of the metric $g^0=g$.  The map $\l\mapsto
\overline{X}_{\wt{g}^{\l}}(x, \xi)$ is differentiable for each ${\bf
v}=(x, \xi)$ with
\[
\left(\overline{X}_{\wt{g}^{\l}}\right)'_0(x, \xi)=\left(0,
-\vf \circ \p \ {\bf
v}+\int_{0}^{+\infty}({\rm K}'_{s}(0)-S'_{\bf v}(0){\rm K}_{s}(0))\ ds\right),
\]
where $\vf:\ M\to \Bbb R$ is such that $g^{\l}=e^{2\l \vf + O(\l^2)}g$, $\p$ denotes the projections  $\p : SM \to M$ and $\p : S\M \to \M$, and  $\left({{\rm K}}_s(0), {{\rm K}}'_s(0)\right)=(D{\bf \Phi}_{s})^{-1}\left(0,   \Upsilon(s)\right)$ with  $\Upsilon=-\nabla \vf+\langle\nabla \vf,
\dot\g_{\bf v}\rangle\dot\g_{\bf v}$.
\end{cor}

\begin{proof} Let $\l\in (-1, 1)\mapsto \vf^{\l}$  be such
that $g^{\l}=e^{2\vf^{\l}}g$.  Clearly,
$\|\overline{X}_{\wt{g}^{\l}}\|_{\wt{g}}=e^{-\vf^{\l}\circ \p }$ and hence
$\left(\|\overline{X}_{\wt{g}^{\l}}\|_{\wt{g}}\right)'_0({\bf
v} ){\bf v}=-\vf \circ \p \ {\bf v}$.  Write $\langle\cdot, \cdot\rangle_{\l}$ for the $\wt{g}^{\l}$-inner product and let $\nabla^{\l}$ denote the associated  Levi-Civita connection (we simply write $\langle\cdot, \cdot\rangle$ and $\nabla$ when $\l=0$). Each  $\nabla^{\l}$ is torsion free and preserves the metric  inner product.  Using these two properties, we obtain Koszul's formula, which says for any smooth vector fields $X, Y, Z$ on $\M$, 
\begin{equation}\label{Koszul}
2\langle\nabla^{\l}_X Y, Z\rangle_{\l}=X\langle Y, Z\rangle_{\l}+ Y\langle X, Z\rangle_{\l}-Z\langle X, Y\rangle_{\l}+\langle [X, Y], Z\rangle_{\l}-\langle [X, Z], Y\rangle_{\l}-\langle [Y, Z], X\rangle_{\l}.
\end{equation}
Note that $\wt{g}^{\l}=e^{2\vf^{\l}\circ \p}\wt{g}$, which means $\langle\cdot, \cdot\rangle_{\l}=e^{2\vf^{\l}\circ \p}\langle\cdot, \cdot\rangle$.  So, if we multiply both sides of (\ref{Koszul}) with $e^{-2\varphi^{\l}\circ \p}$ and compare it with the expression (\ref{Koszul}) for $\nabla$,   we obtain
\begin{eqnarray*}
2\langle\nabla^{\l}_X Y, Z\rangle&=& e^{-2\varphi^{\l}\circ \p}\left((D_X e^{2\varphi^{\l}\circ \p}) \langle Y, Z\rangle+ (D_Y e^{2\varphi^{\l}\circ \p})\langle X, Z\rangle-(D_Ze^{2\varphi^{\l}\circ \p})\langle X, Y\rangle\right)\\
&&+X\langle Y, Z\rangle+ Y\langle X, Z\rangle-Z\langle X, Y\rangle+\langle [X, Y], Z\rangle-\langle [X, Z], Y\rangle-\langle [Y, Z], X\rangle\\
&=& 2 (D_X \varphi^{\l}\circ \p) \langle Y, Z\rangle+2 (D_Y \varphi^{\l}\circ \p) \langle X, Z\rangle-2 (D_Z \varphi^{\l}\circ \p) \langle X, Y\rangle+2\langle\nabla_X Y, Z\rangle.\end{eqnarray*}
Since $Z$ is arbitrary, this implies
\[
\nabla^{\l}_X Y-\nabla_X Y=(D_X\vf^{\l}\circ \p)Y+(D_Y\vf^{\l}\circ \p)X-\langle X , Y\rangle\nabla \vf^{\l}\circ \p
\]
for any two smooth vector fields $X, Y$ on  $\M$. As a consequence, we have
\[
\Gamma_{X} Y= (D_X\vf\circ \p)Y+(D_Y\vf\circ \p)X-\langle X , Y\rangle\nabla \vf\circ \p.
\]
In particular,  $\Gamma_{\dot{\g}}\dot{\g}=2\langle\nabla\vf \circ \p , \dot{\g}\rangle\dot{\g}-\nabla\vf \circ \p $ and the equation (\ref{Morse-correspondence}) reduces to %\footnote{Hmm... The Koszul formula is stated on $\M$, as it should be, and the application is on $S\M$ because, in our convention, all our $\nabla, D, \langle, \rangle$, and so on are on the leaves of $\W$ with the metric coming from $(\M, \wt g). $ So these few lines are correct, but should we explain that?} \footnote{Yes, maybe. I need to think about it. We transfer everywhere...}
\begin{equation*}
\nabla_{\dot{\g}}^2\Xi_{\g}+ {\rm R}(\Xi_{\g},
\dot{\g})\dot{\g}-\nabla \vf \circ \p +\langle\nabla \vf \circ \p ,
\dot\g\rangle\dot\g=0.\end{equation*}
The formula for $\left(\overline{X}_{\wt{g}^{\l}}\right)'_0(x, \xi)$  follows immediately by Proposition \ref{lem-W}.
\end{proof}

\section{Proof of the main theorems}

Let $\lambda\in (-1, 1)\mapsto g^{\lambda}\in \Re(M)$ be a $C^3$
curve of  $C^3$ conformal changes of the metric $g^0=g$.  We  simply use the superscript $\l$
$(\l\not=0)$ for $\overline X, {\bf m}, \wt{\bf m}, k_{\bf v},$ $ \P$ to
indicate that the metric used is $g^{\l}$, for instance, ${\bf
m}^{\l}$ is the harmonic measure for the laminated Laplacian in
metric $g^{\l}$.  The corresponding quantities for $g$ will appear  without superscripts.  Let $\l\in (-1, 1)\mapsto \vf^{\l}$  be such
that $g^{\l}=e^{2\vf^{\l}}g$. For each  $\l$, we have \[
\Delta^{\lambda}=e^{-2\vf^{\lambda}}\left(\Delta+(m-2)\nabla\vf^{\lambda}\right)=:
e^{-2\vf^{\lambda}} L_{\lambda}.
\]
Let $\widehat{\LL}^{\l}:=\Delta+Z^{\l}$ with $Z^{\l} = (m-2)\nabla\vf^{\lambda} \circ \p.$
Leafwisely, $Z^\l$ is the dual of the closed form $(m-2)d \vf ^\l\circ \p.$  Moreover, the pressure of the function $- \langle \overline {X}, Z^0 \rangle  = 0 $ is positive. Therefore,  there exists $\d>0$ such that for $|\l | < \d$, the pressure of the function $- \langle \overline{X}, Z^\l \rangle $ is still positive, so that the results of Section 3 apply to $\widehat{\LL}^\l$ for $\l \in (-\d, +\d).$
Note that
  $\widehat \ell _\l $ and $\widehat{h}_{\l}$ defined in  Section 
\ref{Sec-intro} are just the linear drift and the stochastic entropy
for the operator $\widehat{\LL}^{\l}$ with respect to metric $g$.  Let $\ell_{\l}$ and $h_{\l}$ be the linear drift
and entropy for $(M, g^{\l})$ as were defined in  Section
\ref{Sec-intro}. From the results in  Sections  3 and 4, the following limits considered in  Section 
\ref{Sec-intro} exist:
\begin{eqnarray*}
(d\ell_{\lambda}/d\lambda)|_{\lambda=0}&=&\lim\limits_{\l\to
0}\frac{1}{\l}(\ell_{\l}-\widehat{\ell}_{\l})+\ \lim\limits_{\l\to
0}\frac{1}{\l}(\widehat{\ell}_{\l}-\ell_0)\  =: {\rm{\bf
(I)}}_{\ell}+{\rm{\bf (II)}}_{\ell},\\
(dh_{\lambda}/d\lambda)|_{\lambda=0}&=&\lim\limits_{\l\to
0}\frac{1}{\l}(h_{\l}-\widehat{h}_{\l})+\lim\limits_{\l\to
0}\frac{1}{\l}(\widehat{h}_{\l}-h_0)=:{\rm{\bf (I)}}_{h}+{\rm{\bf
(II)}}_{h}.
\end{eqnarray*}
This shows the differentiability in $\l$ at $0$ of $\l \mapsto \ell_\l $ and $\l \mapsto h_\l $ (Theorem \ref{main-thm}). In this section, we give more details and formulas for the derivative. Namely, we prove the following Theorem

\begin{theo}\label{Main-formulas} Let $(M, g)$ be a negatively curved compact connected $m$-dimensional Riemannian
manifold and let $\lambda\in (-1, 1)\mapsto g^{\lambda}=e^{2\vf^{\l}}g\in \Re(M)$%\footnote{Since we also use $\vf^{\l}$ in the proof, only mentioning its first order derivative $\vf$   is not enough.  }
  be a $C^3$ curve of $C^3$  conformal changes of the metric $g^0=g$ with
constant volume. Let $\vf $ be such that
$g^{\l}=e^{2\l \vf + O(\l^2)}g$. With the above notations, the
following holds true.
\begin{itemize}
  \item[i)] The function
$\l\mapsto {\ell}_{\l}$   is  differentiable at $0$ with%\footnote{I add the subscript $\bf v$ to $k$!}
\begin{eqnarray}
&&(\ell_{\l})'_0=\int_{\scriptscriptstyle{M_0\times
\pp\M}}\langle\vf \circ \p \ \overline{X}+\int_{0}^{+\infty}({\rm K}'_{s}(0)-S'_{(x, \xi)}(0){\rm K}_{s}(0))\ ds,
\nabla\ln k_{\bf v}\rangle\ d\wt{{\bf
 m}}\notag\\
    &&\ \ \ \ \ \ \ \ \ +(m-2)\int_{{\scriptscriptstyle{M_0\times \pp\M}}} \vf \circ \p\ \langle \nabla u_0+\overline{X}, \nabla
\ln k_{\bf v}  \rangle  \ d{\bf {\wt m}},\label{lineardriftderivative}
\end{eqnarray}
where $\left({{\rm K}}_s(0), {{\rm K}}'_s(0)\right)=(D{\bf \Phi}_{s})^{-1}\left(0,   \Upsilon(s)\right)$ with  $\Upsilon=-\nabla \vf+\langle\nabla \vf,
\dot\g\rangle\dot\g$ along the $\widetilde{g}$-geodesic $\g$ with $\dot\g(0)=(x, \xi)$ and  $u_0$ is the
function defined before Proposition \ref{M-0-M-1}.
  \item[ii)]The function
$\l\mapsto {h}_{\l}$   is  differentiable at $0$ with
\begin{equation}\label{entropyderivative} (h_{\l})'_0= (m-2) \int_{{\scriptscriptstyle{SM}}} \vf \circ \p \ \langle \nabla (u_1+\ln k_{\bf v}), \nabla
\ln k_{\bf v}  \rangle\  d{\bf {m}}, \end{equation}
 where $u_1$ is the
function defined before Proposition \ref{M-0-M-1}.
\end{itemize}
\end{theo}
\begin{proof}

Observe firstly that since $g^\l $ has constant volume, $m \int \vf \ d{\rm{Vol}} = \left({\rm{Vol}} (M, g^\l )\right)'_0 = 0 $ and therefore
\begin{equation}\label{volume} \int _{SM}  \vf \circ \p \, d{\bf m} \; = \; 0. \end{equation}

We derive the formula for $(h_{\l})'_0$ first.  Let $\widehat{\bf m}^{\l}$ be the  $G$-invariant extension to $S\M$ of the
harmonic measure corresponding to $\widehat{\LL}^{\l}$ with respect
to  metric $g$. Then $d\widehat{\bf m}^{\l}=e^{-2\vf^{\l} \circ \p}d\wt{{\bf
m}}^{\l}$, where $\vf^{\l} $ also denotes its $G$-invariant extension to $\M$.  Moreover, since there is only a time change between the
leafwise diffusion processes with  infinitesimal operators $\widehat{\LL}^{\lambda}$ and $\Delta^{\lambda}$,  the leafwise Martin kernel functions of the two operators are the same.  (Indeed, because $\widehat{\LL}^{\lambda}$  only differs from $\Delta^{\lambda}$  by multiplication by a positive function, the leafwise positive harmonic functions  of the two generators are the same. In particular, the minimal  leafwise positive harmonic functions normalized at $x=\p({\bf v})$ are the same for  $\widehat{\LL}^\l $ and $\D^\l $.  It is known (\cite[Theorem 3]{An}) that the leafwise Martin kernel functions  $k_{\bf v}^{\l}(\cdot, \xi)$ of  $\widehat{\LL}^{\lambda}$  (or $\Delta^{\lambda}$) can be characterized as minimal  leafwise positive  $\widehat{\LL}^{\lambda}$ (or $\Delta^{\lambda}$)-harmonic functions normalized at $x$  such that $k_{\bf v}^{\l}(y, \xi)$ goes to zero when $y$ tends to a point in the boundary different from $\xi$.   Thus, the two  Martin kernel functions coincide.) Using Proposition \ref{formulas-l-h-Y-h}, we obtain
%\footnote{I add leafwise everywhere in this paragraph so that the reader should know $k_{\bf v}$ is on $S\M$.}
\begin{equation}\label{widehat-h}
\widehat{h}_{\l}=\int \|\nabla^{0}\ln k_{\bf v}^{\l}(x,\xi)\|_{0}^2\
  d\widehat{{\bf m}}^{\l}=\int e^{-2\vf^{\l}\circ \p}\|\nabla\ln k_{\bf v}^{\l}(x, \xi)\|^2 \ d\wt{{\bf
 m}}^{\l},
\end{equation}
whereas here, and hereafter, the integrals with respect to $\widehat{{\bf m}}^{\l}$ and  $\wt{{\bf
 m}}^{\l}$ are always taken on $M_0\times \partial \M$   and we will omit the
subscript of $\int_{\scriptscriptstyle M_0\times\pp\M}$ whenever
there is no ambiguity.  As before,   $k_{\bfv}^{\l} (\cdot,  \eta) $ should be understood as  a function on $W^s (\bfv)$ for all $\eta$. In particular, for $\eta = \xi$. Then its gradient (for the lifted metric from $\M$ to $W^s (\bfv)$)  is a tangent vector to $W^s (\bfv)$. We also know   $k_{\bf v}^{\l}(y, \eta)=k_{\eta}^{\lambda}(y)$, where $k_{\eta}^{\lambda}$ is the Martin kernel function on $\M$ for the $\wt{g}^{\lambda}$-Laplacian.  Of our special interest is $k_{\bf v}^{\l}(\cdot, \xi)$, which we will abbreviate as $k_{\bf v}^{\l}$ in the following context.

For $(h_{\l})'_0$, we have
\begin{eqnarray*}
(h_{\l})'_0=\lim\limits_{\l\to
0}\frac{1}{\l}(h_{\l}-\widehat{h}_{\l})+\lim\limits_{\l\to
0}\frac{1}{\l}(\widehat{h}_{\l}-h_0)=:{\rm{\bf (I)}}_{h}+{\rm{\bf
(II)}}_{h},
\end{eqnarray*}
if both limits exist.  It is easy to see ${\rm{\bf (I)}}_{h}=0$
since by Proposition \ref{formulas-l-h-Y-h} and (\ref{widehat-h}),
\begin{eqnarray*}
 {\rm{\bf (I)}}_{h}&=&  \lim\limits_{\l\to 0}\frac{1}{\l}\left(\int \|\nabla^{\l}\ln k_{\bfv}^{\l}(x,\xi)\|_{\l}^2\
  d\wt{{\bf m}}^{\l}-\int \|\nabla^{0}\ln k_{\bfv}^{\l}(x, \xi)\|_{0}^2\
  d\widehat{{\bf m}}^{\l}\right)\\
 &=& \lim\limits_{\l\to 0}\int\frac{1}{\l}(e^{-2\vf^{\l}\circ \p}-e^{-2\vf^{\l}\circ \p})\|\nabla\ln k_{\bfv}^{\l}(x, \xi)\|^2 \ d\wt{{\bf
 m}}^{\l},
%\\ &=& 0,
\end{eqnarray*}
where we use 
\[\nabla^{\l} \ln k^{\l}_{\bfv}(x, \xi) = e^{-2 \vf^{\l} \circ \p}
\nabla \ln k^{\l}_{\bfv}(x, \xi) \; {\mbox {and}} \; \| \nabla^{\l} \ln
k^{\l}_{\bfv}(x, \xi) \|_{\l}^2 = e^{-2 \vf^{\l}\circ \p} \| \nabla \ln
k^{\l}_{\bfv}(x, \xi)\|^2 .\]
Thus,
\begin{equation}\label{zeroentropy} {\rm{\bf (I)}}_{h} \;= \; 0 .\end{equation}

For ${\rm{\bf (II)}}_{h}$, we have by Theorem
\ref{differential-h} that it equals to
 $\lim_{t\to +\infty}(1/t)\E_{\overline\Q}({\bf Z}_{h, t}{\bf
 M}_t)$.
Recall that ${\bf{x}}_t $ belongs to $W^s ({\bf {x}}_0)$. The process
\begin{eqnarray}\label{Z-champ}
{\bf \wt{Z}}_{t}^{1} &=& f_1({\bf x}_t)-f_1({\bf x}_0)-\int_{0}^{t}
(\Delta f_1)({\bf x}_s)\ ds,
\end{eqnarray}
where $f_1=-\ln k_{\bfv} -u_1$ and $\bfv = {\bf {x}}_0$  and the function
$u_1$ is such that
\begin{equation}\label{uone:def}
\D u_1  = \|\nabla \ln k_{\bf v}\|^2 - h_{0}
\end{equation}
 is a martingale with increasing process $\ 2\|\nabla \ln
k_{\bf v}+\nabla u_1\|^2  ({\bf x}_t)\ dt.$ 
It is true by Proposition  \ref{Covariance-2}  that
\[\lim\limits_{t\to +\infty}\frac{1}{t}\E_{\overline\Q}({\bf Z}_{h, t}{\bf M}_t)=\lim\limits_{t\to +\infty}\frac{1}{t}\E_{\overline\Q}({\bf \wt{Z}}_{t}^1{\bf M}_t),\]
where
\[
 {\bf M}_{t}=\frac{1}{2} \int_{0}^{t}\langle (Z^{\l})'_{0}({\bf x}_{s}),
{\bf w}_{s} dB_s \rangle_{{\bf x}_{s}}.
\]
Note that  $(Z^{\l})'_{0}$, the $G$-invariant extension  of $(m-2)\nabla \vf  \circ \p $, is a gradient field.  So, if we write
$\psi=\frac{1}{2}(m-2)\vf \circ \p $,  we have by Ito's formula that
\begin{equation}\label{M-champ}
 {\bf M}_{t}=\psi({\bf x}_t)-\psi({\bf x}_0)-\int_{0}^{t}(\Delta \psi)({\bf x}_s)\ ds
\end{equation}
is a martingale with increasing process $2\|\nabla \psi\|^2$. Using
(\ref{Z-champ}), (\ref{M-champ}) and a straightforward computation
using integration by parts formula for $(a {\bf \wt{Z}}_t^1+b{\bf
M}_t)^2$, $a, b=0$ or $1$,  we
obtain
\[{\bf \wt{Z}}_{t}^1{\bf M}_t=2 \int_{0}^{t}\langle \nabla f_1, \nabla \psi\rangle ({\bf x}_s)\ ds\]
and hence%\footnote{As you can imagine, I add subscript $\bf v$ everywhere.}
\begin{eqnarray*}
\lim\limits_{t\to +\infty}\frac{1}{t}\E_{\overline\Q}({\bf
\wt{Z}}_{t}^1{\bf M}_t)= 2 \int \langle\nabla f_1, \nabla
\psi\rangle\ \ d{\bf \wt m}= -2\int\langle\nabla \ln k_{\bf v},  \nabla
\psi\rangle\ d{\bf \wt m} - 2 \int \langle \nabla u_1, \nabla \psi
\rangle\ d\bf{\wt m}.
\end{eqnarray*}
    Here,%\footnote{You already lift $\psi$ to $S\M$, but we take integral on $SM$, so I emphasize that we identify $SM$ with $M_0\times \partial \M$ and $\bf m$ is identified with the restriction of $\wt{\bf m}$ on $M_0\times \partial \M$. I add bigskip before Next...}
\[
-2\int\langle\nabla \ln k_{\bf v}, \nabla \psi\rangle\ d{\bf \wt
m}=2\int_{\scriptscriptstyle SM} {\rm Div}(\nabla \psi)\ d{\bf
m}=(m-2)\int_{\scriptscriptstyle SM} \Delta (\vf \circ \p) \  d{\bf m}=0,
\]
where  the first equality is the integration by parts  formula and $\bf m$ is identified with the restriction of $\widetilde{\bf m}$ to $M_0\times \partial \M$,   and
the last one holds because ${\bf m}$ is $\Delta$-harmonic. We
finally obtain
\begin{equation*}%\label{formula-for h}
(h_{\l})'_0= - (m-2)  \int _{\scriptscriptstyle SM}
 \langle \nabla u_1, \nabla \vf \circ \p \rangle\ d{\bf{m}}.
\end{equation*}
Observe that:
  \begin{eqnarray}  2 \langle \nabla u_1, \nabla \vf \circ \p \rangle & = & \D (u_1 \vf \circ \p) - \D (u_1) \vf \circ \p - u_1 \D \vf \circ \p \notag\\
&=&  \D (u_1 \vf \circ \p)  - \vf  \circ \p \|\nabla \ln k_{\bf v}\|^2 + h_{0} \vf \circ \p- u_1 \D
\vf \circ \p,\label{h-calculation of inner product}
\end{eqnarray}
where we use the defining property (\ref{uone:def}) of $u_1$. When
we take the integral of (\ref{h-calculation of inner product}) with respect to ${\bf m}$,
the first term vanishes because $\bf m$ is $\D$ harmonic, the second term
gives $-\int \vf \circ \p  \|\nabla \ln k_{\bf v}\|^2 \ d{\bf m}$, the third
term vanishes by (\ref{volume}). Finally for the last
term, by using the integration by parts formula:
\begin{equation}\label{int-by-path} \int _{\scriptscriptstyle SM} u \D v \ d{\bf m} = \int _{\scriptscriptstyle SM} v \D
u\ d{\bf m} +2 \int_{\scriptscriptstyle SM} v \langle \nabla u,
\nabla \ln k_{\bf v}  \rangle \ d{\bf m} , \end{equation} we have
\begin{eqnarray*}
\int _{\scriptscriptstyle SM} u_1\Delta \vf \circ \p \  d{\bf  m}&=&\int
_{\scriptscriptstyle SM} \vf \circ \p\left( \D u_1 + 2  \langle
\nabla u_1, \nabla \ln k_{\bf v}  \rangle \right)\  d{\bf  m}\\
&=&\int _{\scriptscriptstyle SM} \vf \circ \p\left(\|\nabla \ln k_{\bf v}\|^2 + 2
\langle \nabla u_1, \nabla \ln k_{\bf v} \rangle\right)\ d{\bf m}.
\end{eqnarray*}

\bigskip
Next, we derive the formula for $(\ell_{\l})'_0$. Clearly,
\begin{eqnarray*}
(\ell_{\l})'_0=\lim\limits_{\l\to
0}\frac{1}{\l}(\ell_{\l}-\widehat{\ell}_{\l})+\lim\limits_{\l\to
0}\frac{1}{\l}(\widehat{\ell}_{\l}-\ell_0)=: {\rm{\bf
(I)}}_{\ell}+{\rm{\bf (II)}}_{\ell},
\end{eqnarray*}
if  both limits exist. Here the $\widehat{\ell}_{\l}$ defined in
the introduction  is just the linear drift for the operator
$\widehat{\LL}^{\l}$ with respect to metric $g$. The $\rm{\bf
(II)}_{\ell}$ term  can be analyzed similarly as above for ${\rm{\bf
(II)}}_{h}$.  Indeed, by Theorem \ref{differential-ell}, $\rm{\bf
(II)}_{\ell}=\lim_{t\to +\infty}(1/t)\E_{\overline\Q}({\bf Z}_{\ell,
t}{\bf M}_t)$. The process
\begin{eqnarray}\label{Z-champ-zero}
{\bf \wt{Z}}_{t}^{0} &=& f_0({\bf x}_t)-f_0({\bf x}_0)-\int_{0}^{t}
(\Delta f_0)({\bf x}_s)\ ds,
\end{eqnarray}
where $f_0=b_{{\bf v}}-u_0$  and the
function $u_0$ is such that
\begin{equation}\label{uzero:def}
\D u_0  = -{\rm{Div}}(\overline{X})-\ell_{0}
\end{equation}
is a  martingale  with increasing process
$2\|\overline{X}+\nabla u_0\|^2({\bf x}_t)\ dt.$
It is true by Proposition \ref{Covariance-1} that
\[\lim\limits_{t\to +\infty}\frac{1}{t}\E_{\overline\Q}({\bf Z}_{\ell, t}{\bf M}_t)=\lim\limits_{t\to +\infty}\frac{1}{t}\E_{\overline\Q}({\bf \wt{Z}}_{t}^0{\bf M}_t),\]
where ${\bf M}_t$, by (\ref{M-champ}), is a martingale with
increasing process $2\|\nabla \psi\|^2$. So using (\ref{M-champ}),
(\ref{Z-champ-zero})
 and a straightforward computation using integration
by parts formula for $(a {\bf \wt{Z}}_t^0+b{\bf M}_t)^2$, $a, b=0$
or $1$,  we obtain
\[{\bf \wt{Z}}_{t}^0{\bf M}_t=2 \int_{0}^{t}\langle \nabla f_0, \nabla \psi\rangle ({\bf x}_s)\ ds\]
and hence (recall that $\nabla b _\bfv = -\overline X (\bfv)$, see (\ref{Buse-geo}))%\footnote{two lines or one line, smaller size?}
\begin{eqnarray*}
\lim\limits_{t\to +\infty}\frac{1}{t}\E_{\overline\Q}({\bf
\wt{Z}}_{t}^0{\bf M}_t)&=& 2 \int \langle\nabla f_0, \nabla
\psi\rangle\ \ d{\bf \wt{m}}\\
&=& -(m-2)\left(\int\langle \overline{X},
\nabla (\vf \circ \p)\rangle\ d{\bf \wt{m}} +  \int\langle \nabla u_0, \nabla
(\vf \circ \p)\rangle\ d\bf{\wt{m}}\right).
\end{eqnarray*}
Using the formula ${{\rm{Div}}}(\vf \circ \p\ \overline{X})=\vf \circ \p\
{{\rm{Div}}}\overline{X}+\langle\nabla(\vf \circ \p), \overline{X}\rangle$, we obtain
%\footnote{It is also in small size. If you do not like it, change it!}
%\begin{small}
\begin{eqnarray*}
\int\langle \overline{X},
\nabla (\vf \circ \p)\rangle\ d{\bf \wt{m}}&=& \int\left({{\rm{Div}}}(\vf \circ \p\ \overline{X})-\vf \circ \p\
{{\rm{Div}}}\overline{X}\right)\ d{\bf \wt{m}}\\ &=&-\int\vf \circ \p\left(\langle \overline{X}, \nabla\ln k_{\bf v} \rangle+
{{\rm{Div}}}\overline{X}\right)\ d{\bf \wt{m}},
\end{eqnarray*}
%\end{small}
where we used the foliated integration by parts formula $\int {\rm{Div}} Y \ d{\bf \wt{m}} = - \int \< Y, \nabla \ln k_{\bf v} \> \ d{\bf \wt{m}}.$  Observe that:
  \begin{eqnarray*}  2 \langle \nabla u_0, \nabla (\vf \circ \p) \rangle & = & \D (u_0 \ \vf \circ \p) - \D (u_0) \vf \circ \p - u_0 \D (\vf \circ \p) \\
&=&  \D (u_0 \ \vf \circ \p )  + \vf \circ \p \  {\rm{Div}}(\overline{X}) + \ell_{0} \vf \circ \p
- u_0 \D (\vf \circ \p),
\end{eqnarray*}
where we use the defining property (\ref{uzero:def}) of $u_0$.  When
we report in the integration $2\int\langle \nabla u_0, \nabla
(\vf \circ \p)\rangle\ d\bf{\wt{m}}$, the first term
vanishes because $\bf m$ is $\D$ harmonic, the second term is $-\int \vf \circ \p \Delta u_0\ d\bf{\wt{m}}$ by (\ref{uzero:def}) and  the third term vanishes by (\ref{volume}).  Again,  using the integration by parts
formula (\ref{int-by-path}) for $\int u_0\Delta (\vf \circ \p)\  d{\bf \wt m}$, we have
\[
\int\langle \nabla u_0, \nabla
(\vf \circ \p) \rangle\ d{\bf{\wt{m}}}=-\int \vf \circ \p( \Delta u_0 +  \langle
\nabla u_0, \nabla \ln k  \rangle )\  d{\bf \wt m}
\]
Finally,  we
obtain
\begin{eqnarray*}
{\rm{\bf (II)}}_{\ell}&=&(m-2)\int \vf \circ \p\left( \D
u_0+{\rm{Div}}\overline{X} +   \langle \nabla u_0+\overline{X},
\nabla \ln k_{\bf v}  \rangle \right) \ d{\bf {\wt
m}}\\
&=& (m-2)\int\vf \circ \p\langle \nabla u_0+\overline{X}, \nabla \ln k_{\bf v}
\rangle  \ d{\bf {\wt m}},
\end{eqnarray*}
where the last equality holds by using (\ref{uzero:def}) and (\ref{volume}).

For ${\rm{\bf (I)}}_{\ell}$, we first observe the convergence of
Martin kernels and harmonic measures.  For any $(x,\xi)=:{\bfv}\in \M \x \pp\M$, the
Martin kernel function $ k^{\l}_{\bf v}(y, \xi)$ converges to  $k_{\bf v}(y,\xi)$
pointwisely as $\l$ goes to zero.  For small $\l$  and fixed $x$,  the function
$\xi\mapsto \nabla \ln k_{x,\xi}^{\l}$ is H\"{o}lder continuous on
$\pp\M$ for some uniform exponent (\cite{H1}).   As a consequence,
we have the convergence of  $\nabla \ln k_{\bfv}^{\l}$ (and hence
$\nabla^{\l} \ln k_{\bfv}^{\l}$) to  $\nabla \ln k_{\bfv}$
when $\l$ tends to zero.  %We also obtain the uniform convergence of
%$k^{\l}_{\xi}(x)$ to $k_{\xi}^0(x)$ and $\nabla \ln k_{\xi}^{\l}(x)$
%to  $\nabla \ln k_{\xi}^{0}(x)$ (as $\l\to 0$) in $\xi$.  So, 
By uniqueness,   the
harmonic measure $\wt{{\bf
 m}}^{\l}$ converges weakly to $\wt{{\bf
 m}}$ ($\l\to 0$) as well.  By Proposition \ref{formulas-l-h-Y-l},
\[ \ell_{\l}=\int\langle\overline X^{\l}, \nabla^{\l} \ln k_{\bf v}^{\l}\rangle_{\l}\ d\wt{{\bf m}}^{\l}
=\int\langle\overline X^{\l}, \nabla \ln k_{\bf v}^{\l}\rangle\
d\wt{{\bf m}}^{\l}.\]
%We have by Proposition \ref{lem-W} and the
%convergence of $\nabla \ln k^{\l}$ and $\wt{{\bf m}}^{\l}$ stated
%above that
Thus,
\begin{eqnarray*}
 {\rm{\bf (I)}}_{\ell}&=&\lim\limits_{\l\to
0}\frac{1}{\l}\int\langle(\overline X^{\l}- \overline {X}^0),
\nabla \ln k_{\bf v}\rangle\ d\wt{{\bf m}}+\lim\limits_{\l\to
0}\frac{1}{\l}\left(\int\langle\overline X,
\nabla \ln k_{\bf v}^{\l}\rangle\
d\wt{{\bf m}}^{\l}-\widehat{\ell}_{\l}\right)\\
&=:&  {\rm{\bf (III)}}_{\ell}+ {\rm{\bf (IV)}}_{\ell}
\end{eqnarray*}
if  $ {\rm{\bf (III)}}_{\ell} $ and ${\rm{\bf (IV)}}_{\ell}$ exist. The quantity $ {\rm{\bf (III)}}_{\ell}$,  by Corollary \ref{X-lam-Lem}, is 
\[
\int\langle-\vf \circ \p \ \overline{X}+\int_{0}^{+\infty}({\rm K}'_{s}(0)-S'_{\bf v}(0){\rm K}_{s}(0))\ ds,
\nabla\ln k_{\bf v}\rangle\ d\wt{{\bf
 m}}.
 \]  
  By
Proposition \ref{formulas-l-h-Y-l},
\[
 \widehat{\ell}_{\l}=-\int \left({\rm{Div}}\overline{X}+ \langle Z^{\l}, \overline{X}\rangle\right)\
  d\widehat{{\bf m}}^{\l}.
\]
For ${\rm{\bf (IV)}}_{\ell}$, let us first calculate $\int
{\rm{Div}}\overline{X}\ d\widehat{{\bf m}}$. We have
\begin{eqnarray*}
\int {\rm{Div}}\overline{X}\ d\widehat{{\bf m}}^{\l}&=& \int
e^{-2\vf ^{\l}\circ \p }{\rm{Div}}\overline{X}\ d\wt{{\bf m}}^{\l}\\
&=& \int  e^{-2\vf^{\l}\circ \p } {\rm{Div}}^{\l}\overline{X}\ d\wt{{\bf
m}}^{\l}+ \int
\left({\rm{Div}}\overline{X}-{\rm{Div}}^{\l}\overline{X}\right)\
d\widehat{{\bf m}}^{\l}\\
&=&\int  e^{-2\vf^{\l}\circ \p } {\rm{Div}}^{\l}\overline{X}\ d\wt{{\bf
m}}^{\l}-m\int \langle\nabla(\vf^{\l}\circ \p ), \overline{X}\rangle\
d\widehat{{\bf m}}^{\l},
\end{eqnarray*}
where the last equality holds since
$({{\rm{Div}}}^{\l}-{{\rm{Div}}})(\cdot)=m\langle\nabla (\vf^{\l} \circ \p ,
\cdot\rangle$ for $g^{\l}=e^{2\vf^{\l}}g$.  Note that
\begin{eqnarray*}
{\rm{Div}}^{\l}(e^{-2\vf^{\l}\circ \p }\overline{X})=e^{-2\vf^{\l}\circ \p }{\rm{Div}}^{\l}\overline{X}-2e^{-2\vf^{\l}\circ \p}\langle
\nabla^{\l}(\vf^{\l} \circ \p ), \overline{X}\rangle_{\l}.
\end{eqnarray*}
So we have%\footnote{I shift the the last quantity of the first line to the second.}
\begin{eqnarray*}
  \int {\rm{Div}}\overline{X}\ d\widehat{{\bf m}}^{\l}&=& \int
{\rm{Div}}^{\l}(e^{-2\vf^{\l}\circ \p }\overline{X})\ d\wt{{\bf m}}^{\l}+\int
2e^{-2\vf^{\l}\circ \p }\langle \nabla^{\l}\vf^{\l}\circ \p, \overline{X}\rangle_{\l}
\ d\wt{{\bf m}}^{\l}\\
&&\ \ \ \ \ \ \ \ \ \ \ \ \ \ \ \ \ \ \ \ \ \ \ \ \ \ \ \ \ \ \ \ -m\int \langle\nabla(\vf^{\l} \circ \p ), \overline{X}\rangle\
d\widehat{{\bf m}}^{\l}\\
&=&-\int\langle\overline{X}, \nabla^{\l}\ln
k_{\bf v}^{\l}\rangle_{\l}\ d\widehat{{\bf m}}^{\l}-(m-2)\int
\langle\nabla (\vf^{\l} \circ \p ), \overline{X}\rangle\ d\widehat{{\bf
m}}^{\l}\\
&=&-\int\langle\overline{X}, \nabla\ln
k_{\bf v}^{\l}\rangle\ d\widehat{{\bf m}}^{\l}-(m-2)\int \langle\nabla
(\vf^{\l} \circ \p ), \overline{X}\rangle\ d\widehat{{\bf m}}^{\l},
\end{eqnarray*}
where,  for the second equality,  we use the leafwise integration by parts formula $\int {\rm{Div}}^{\l} Y \ d{\bf \wt{m}^{\l}} = - \int \< Y, \nabla^{\l} \ln k^{\l}_{\bf v} \>_{\l} \ d{\bf \wt{m}^{\l}}$. 
 This  gives
\begin{equation}\label{ellhatlambda}
 \widehat{\ell}_{\l}=\int\langle\overline{X}, \nabla\ln
k_{\bf v}^{\l}\rangle\ d\widehat{{\bf m}}^{\l}.
\end{equation}
Finally, we obtain
\begin{eqnarray*}
  {\rm{\bf (IV)}}_{\ell}=\lim\limits_{\l\to 0}\int\frac{1}{\l}(e^{2\vf^{\l}\circ \p }-1)\langle\overline{X}, \nabla\ln
k_{\bf v}^{\l}\rangle\ d\widehat{{\bf m}}^{\l}=2\int
\vf\circ \p  \ \langle\overline{X}, \nabla\ln k_{\bf v}\rangle\ d\wt{{\bf
 m}}.
\end{eqnarray*}
\end{proof}

\begin{proof}[Proof of Theorem \ref{critical}]
   Let $(M, g)$ be a negatively
curved compact connected Riemannian manifold. Define the {\it
{volume entropy}} $v_g$ by:
\[ v_g \; = \; \lim\limits_{r \to +\infty } \frac {\ln {\rm{Vol}} (B(x, r))}{r} ,\]
where $B(x, r) $ is the ball of radius $r$ in $\M$.  we have $\ell_g \leq  v_g $, $ h_g
\leq \; v_g^2 $  (see
\cite{LeS} and the references within).  In particular, if $\lambda\in (-1, 1)\mapsto
g^{\lambda}\in \Re(M)$ is a $C^3$ curve of conformal changes of the
metric $g^0=g$,
\[ \ell_{g^\l }  \; \leq \; v_{g^\l }, \quad h_{g^\l } \; \leq \; v_{g^\l}^2 .\]

Assume $(M, g^0)$ is  locally symmetric. Then $\ell _{g^0} = v_{g^0}
$ and $h_{g^0} = v_{g^0}^2.$ Moreover it is known (Katok \cite{Ka})
that $v_0 $ is a global minimum of the volume entropy among metrics
$g$ which are conformal to $g^0$ and have the same volume and
(Katok-Knieper-Pollicott-Weiss \cite{KKPW}) that $\l \mapsto v_{g^\l
}$ is differentiable. In particular $v_{g^\l } $ is critical at $\l
= 0 $. Since, by Theorem \ref{main-thm}, $\ell _{g^\l } $ and
$h_{g^\l } $ are differentiable at $\l = 0 $, they have to be
critical as well.
\end{proof}

\begin{remark} We can also show Theorem \ref{critical} using the
formulas in Theorem \ref{Main-formulas}.  Indeed, the conclusion for
the stochastic entropy follows from (\ref{entropyderivative})
    since for a locally symmetric space, the solutions $u_1$ to (\ref{uone:def}) are constant (\cite{L5}) and $\|\nabla \ln k_{\bf v} \|^2 $  is also constant. The derivative is proportional  to $\int \vf \circ \p  \ d{\bf {m}}, $  which vanishes by (\ref{volume}).

     We also see that the stochastic entropy depends only on the volume for surfaces ($m= 2$). For the drift $\ell
     $, it  is true that for a locally symmetric space,  $\nabla \ln
k_{\bf v}=-\ell \nabla b_{\bf v}$ everywhere. The solutions $u_0$ to (\ref{uzero:def}) are
constant for a locally symmetric space as well (\cite{L5}). So
(\ref{lineardriftderivative}) reduces to
\[(\ell_{\l})'_0=-\int_{\scriptscriptstyle{M_0\times
\pp\M}}\vf \circ \p  \  \langle \int_{0}^{+\infty}({\rm K}'_{s}(0)-S'_{\bf v}(0){\rm K}_{s}(0))\ ds, \nabla\ln k_{\bf v}\rangle\ d\wt{{\bf
 m}},\]
which is zero because the vector $\int_{0}^{+\infty}({\rm K}'_{s}(0)-S'_{\bf v}(0){\rm K}_{s}(0))\ ds$
is orthogonal to ${\bf v}$ and hence is orthogonal to  $\nabla \ln k_{\bf v}$.% \footnote{Now, is this relation correct if we treat is as a leafwise gradient.}
\end{remark}

\small{{\bf{Acknowledgments}}  The second author was partially supported by  NSFC (No.11331007 and 11422104) and Beijing Higher Education Young Elite Teacher Project. She would also like to thank  LPMA and the  Department of Mathematics of the University of Notre Dame for hospitality during her stays. We are really grateful to  the careful reading and the numerous remarks of the referee, which help us to  improve the writing.


\small\begin{thebibliography}{99}

\bibitem[\bf Anc]{An} A. Ancona, Negatively curved manifolds, elliptic
operators and the Martin boundary, \emph {Ann. Math. (2)} {\bf 125}
(1987) 495--536.




\bibitem[\bf Ano1]{A} D. V. Anosov, \emph{Geodesic Flow on Closed
Riemannian Manifolds with Negative Curvature}, {Proc. Steklov Inst.
Math. {\bf 90}} (1967).

\bibitem[\bf Ano2]{Ano} D. V. Anosov, \emph{Tangent fields of transversal foliations in $U$-systems}, Math. Notes Acad. Sci. USSR 2:5 (1968), 818--823.

\bibitem[\bf Ba]{Ba} W. Ballmann, Lectures on spaces of nonpositive curvature,  
With an appendix by Misha Brin. \emph{DMV Seminar}, {\bf 25},  Birkh‰user Verlag, Basel, 1995.


\bibitem[\bf Bi]{Bi} P. Billingsley, \emph{Probability and measure}, Third edition, Wiley Series in Probability and Mathematical Statistics, A Wiley-Interscience Publication. John Wiley $\&$ Sons, Inc., New York, 1995.

\bibitem[\bf Bo]{Bo} J. Bolton,
Conditions under which a geodesic flow is Anosov, 
\emph{Math. Ann.}  {\bf 240} (1979),  103--113.


%\bibitem[\bf Br]{Br81} R. Brooks, The fundamental group and the spectrum of the Laplacian, \emph{Comment. Math. Helv.}  {\bf 56} (1981),  581--598.



\bibitem[\bf BHM]{BHM1} S. Blach\`{e}re, P. Ha\"{\i}ssinsky and P.
Mathieu,  Asymptotic entropy and Green speed for random walks on
countable groups, \emph {Ann. Probab.}  {\bf 36} (2008),
1134--1152.


\bibitem[\bf DGM]{DGM} A. Debiard,  B. Gaveau and     E.  Mazet,  Th\'{e}or\`{e}mes  de comparaison en g\'{e}om\'{e}trie riemannienne, \emph{Publ. Res. Inst. Math. Sci. } {\bf 12} (1976),  391--425.

%\lin{\bibitem[\bf E]{Eb} P. Eberlein,  When is a geodesic flow of Anosov type? I, \emph{J. Differential Geom.}  {\bf 8} (1973), 437--463.}


\bibitem[\bf EO]{EO} P. Eberlein and B. O'Neill, {Visibility manifolds}, \emph{Pacific Journal of Mathematics}, {\bf 46} (1973), 45--109.

\bibitem[\bf El]{El} K. D. Elworthy, \emph{Stochastic Differential Equations on Manifolds},  London Mathematical Society Lecture Note Series, 70. Cambridge University Press, Cambridge-New York, 1982.

\bibitem[\bf Esc]{Esc} J.-H. Eschenburg, Horospheres and the stable part of the geodesic flow, \emph{Math. Z.} {\bf 153} (1977), 237--251.

\bibitem[\bf FF]{FF} A. Fathi and L. Flaminio, Infinitesimal conjugacies and Weil-Petersson
metric, \emph{Ann. Inst. Fourier (Grenoble)} {\bf 43} (1993),
279--299.




\bibitem[\bf Ga]{Ga} L. Garnett, Foliations, the ergodic theorem and Brownian motion, \emph{J. Funct.
Anal.} {\bf 51} (1983), 285--311.


\bibitem[\bf GH]{GH} E. Ghys and P. de la Harpe, Sur les groupes
hyperboliques d'apr\`{e}s Mikhael Gromov, \emph{Progress in
Mathematics} Vol. 83, Birkh\"{a}user Boston Inc., Boston, MA, 1990.

\bibitem[\bf Gre]{Gre} L. W. Green,   A theorem of E. Hopf, \emph{Michigan Math. J.}  {\bf 5} (1958),  31--34.

\bibitem[\bf Gro]{Gr} M. Gromov, Three remarks on geodesic dynamics and fundamental group, \emph{Enseign. Math.} (2) {\bf 46} (2000),  391--402.



\bibitem[\bf Gu]{Gu} Y. Guivarc'h, Sur la loi des grands nombres et le rayon spectral
d'une marche al\'{e}atoire, \emph{Ast\'{e}risque,} {\bf 74} (1980),
47--98.


\bibitem[\bf H1]{H1} U. Hamenst\"{a}dt, An explicite description of
harmonic measure, \emph{Math. Z.} {\bf 205} (1990), 287--299.


\bibitem[\bf H2]{H2}U. Hamenst\"{a}dt, Harmonic measures for compact
negatively curved manifolds, \emph{Acta Math.} {\bf 178} (1997),
39--107.

\bibitem[\bf HIH]{HIH} E. Heintze and H. C. Im Hof, Geometry of horospheres, \emph{J. Differential Geometry} {\bf 12} (1977), 481--491. 

\bibitem[\bf HPS]{HPS} M. Hirsch, C. Pugh and M. Shub, \emph{Invariant
Manifolds,} Lecture Notes in Math. {\bf 583}, Springer, Berlin,
1977.

\bibitem[\bf Hs]{Hs} E.-P. Hsu, \emph{Stochastic Analysis on Manifolds}, Graduate Studies in Mathematics, 38. American Mathematical Society, Providence, RI, 2002.




\bibitem[\bf K1]{K1} V. A. Kaimanovich, Brownian motion and harmonic
functions on covering manifolds. An entropic approach, {\em Soviet
Math. Dokl.} {\bf 33} (1986),  812--816.

\bibitem[\bf K2]{K2} V. A. Kaimanovich, Invariant measures for the geodesic flow and
measures at infinity on negatively curved manifolds, {\em Ann. Inst.
Henri Poincar\'e, Physique Th\'eorique} {\bf 53} (1990),
361--393.


\bibitem[\bf Ka]{Ka} A. Katok, Four applications of conformal equivalence to geometry and dynamics, \emph{Ergod. Theory  Dynam. Systems} {\bf 8} (1988), 139--152.

\bibitem[\bf KKPW]{KKPW} A. Katok,  G. Knieper, M. Pollicott and H. Weiss, Differentiability and analyticity of topological entropy for Anosov and geodesic flows, \emph{Invent. math.} {\bf 98} (1989), 581--597.

\bibitem[\bf Ki]{Ki} J. F. C. Kingman, The ergodic theory of subadditive stochastic processes, \emph{J.  Royal Stat. Soc.}  {\bf B30 } (1968), 499--510. 



\bibitem[\bf L1]{L1} F. Ledrappier, Ergodic properties of Brownian motion on covers of compact negatively-curved manifolds, \emph{Bol. Soc. Bras. Mat.} {\bf 19} (1988), 115--140.



\bibitem[\bf L2]{L5} F. Ledrappier, Central limit theorem in negative
curvature, \emph {Ann. Probab.}   {\bf 23} (1995),
1219--1233.


\bibitem[\bf L3]{L6} F. Ledrappier,
Profil d'entropie dans le cas continu, \emph{Ast\'{e}risque}  {\bf 236} (1996), 189--198.



\bibitem[\bf LS1]{LeS} F. Ledrappier and L. Shu, Entropy rigidity of symmetric spaces without focal points,  \emph{Trans. Amer. Mat. Soc.} {\bf 366} (2014), 3805--3820.

\bibitem[\bf LS2]{LS3} F. Ledrappier and L. Shu, in prepration.

\bibitem[\bf LY]{LY} P. Li and S.-T. Yau, On the parabolic kernel of the Schr\"{o}dinger operator, \emph{Acta Math.}  {\bf 156} (1986), 153--201.



\bibitem[{\bf LMM}]{LMM} R. de la Llave, J.-M. Marco and R. Moriy\'{o}n,
Canonical perturbation theory of Anosov systems and regularity
results for the Livsic cohomology equation, \emph{Ann. of Math.} (2)
{\bf 123} (1986),   537--611.






\bibitem[\bf Ma]{Ma} P. Mathieu, Differentiating the entropy of
random walks on hyperbolic groups,  \emph{Ann. Prob.} {\bf 43} (2015) 166--187.

\bibitem[\bf Moh]{Moh} O. Mohsen,  \emph{Familles de mesures au bord et bas du
spectre}, Thesis,   \'{E}cole Polytechnique, 2007.


\bibitem[\bf Mor]{Mor} M. Morse,
A fundamental class of geodesics on any closed surface of genus
greater than one, \emph{Trans. Amer. Math. Soc.}  {\bf 26} (1924),  25--60.




\bibitem[\bf Mos]{Mos} J. Moser,
On Harnack's theorem for elliptic differential equations,
\emph{Comm. Pure Appl. Math.} {\bf 14} (1961),  577--591.

\bibitem[\bf N]{N} A. A. Novikov,
On moment inequalities and identities for stochastic integrals,
Proceedings of the Second Japan-USSR Symposium on Probability Theory
(Kyoto, 1972), pp. 333--339 in \emph{ Lecture Notes in Math., Vol.
330, Springer, Berlin, 1973.}


\bibitem[\bf PPS]{PPS} F. Paulin, M. Pollicott and B. Schapira,  Equilibrium states in negative curvature,
\emph{Ast\'erisque},  Soc. Math. France, {\bf 373} (2015).


\bibitem[\bf RY]{RY} D. Revuz  and M. Yor,   \emph{Continuous martingales and Brownian motion},  Grundlehren der Mathematischen Wissenschaften, %[Fundamental Principles of Mathematical Sciences],
 293. Springer-Verlag, Berlin, 1999.




\bibitem[\bf SFL]{SFL} M. Shub, A. Fathi and R. Langevin, \emph{Global stability of dynamical
systems}, Springer-Verlag, New York-Berlin, 1987.


\bibitem[\bf Sa]{Sa} L. Saloff-Coste, Uniformly elliptic operators on Riemannian manifolds, \emph{J. Differential Geom.}  {\bf 36} (1992),  417--450.





\bibitem[\bf St]{Sta}G. Stampacchia,
Le probl\`{e}me de Dirichlet pour les \'{e}quations elliptiques du
second ordre \`{a} coefficients discontinus,  \emph{Ann.
Inst. Fourier (Grenoble)} {\bf 15} (1965), 189--258.



%\bibitem[\bf Su]{Su} D. Sullivan, The Dirichlet problem at infinity for a negatively curved manifold, \emph{J. Differential Geometry}, {\bf 18} (1983) 723--732.

\bibitem[\bf T]{T}N. S. Trudinger,
Linear elliptic operators with measurable coefficients,  \emph{Ann.
Scuola Norm. Sup. Pisa} (3) {\bf 27} (1973), 265--308.


\bibitem[\bf W]{Wa} F.-Y. Wang, Sharp explicit lower bounds of heat kernels,  \emph {Ann. Probab.}   {\bf 25} (1997),  1995--2006.

\bibitem[\bf Y]{Y} C. Yue, Brownian motion on Anosov foliations and manifolds of negative curvature, \emph{J. Diff. Geom.}  {\bf 41} (1995), 159--183.
\end{thebibliography}
\end{document}